\documentclass{article}
\usepackage[english]{babel}
\usepackage{graphicx}
\usepackage{amscd}
\usepackage{amsfonts}
\usepackage{latexsym}
\usepackage{amssymb,amsmath}
\usepackage[usenames]{color}
\usepackage{amscd}
\usepackage{psfrag}
\usepackage{amsthm}
\theoremstyle{plain}
\newtheorem{teor}{Theorem}[section]
\newtheorem*{mteor}{Main Theorem}
\newtheorem*{mcor}{Corollary}

\newtheorem{cor}{Corollary}[section]
\newtheorem{prop}[teor]{Proposition}
\newtheorem{lemma}{Lemma}[section]

\newtheorem{remark}{Remark}[section]

\theoremstyle{definition}

\def\C{\mathbb{C}}
\def\S{\mathbb{S}}

\def\m{\mathcal{M}_{\Gamma}}
\def\F{\mathcal{F}}
\def\r{_a|_{V_a}}
\def\L{\mathcal{L}}

\def\A{\mathcal{A}}
\def\D{\mathbb{D}}
\def\K{\mathring{K}}
\def\P{\mathcal{P}}
\def\R{\mathbb{R}}

\def\Z{\mathbb{Z}}
\def\Q{\mathbb{Q}}
\def\H{\mathbb{H}}

\def\l0{_{a_0}}
\def\i0{_{\iota_0}}
\def\ll{_{a}}
\def\f{_{\theta}}

\begin{document}
\title{Mating quadratic maps with the modular group III:\\ The modular Mandelbrot set}

\author{Shaun Bullett 
\and 
Luna Lomonaco}

\maketitle
\begin{abstract}
We prove that there exists a homeomorphism $\chi$ between the connectedness locus $\m$ for the family $\F_a$
of $(2:2)$ holomorphic correspondences introduced by Bullett and Penrose,
and the parabolic Mandelbrot set $\mathcal{M}_1$.
The homeomorphism $\chi$ is dynamical ($\F_a$ is a mating between $PSL(2,\Z)$ and $P_{\chi(a)}$), 
it is conformal on the interior of $\m$, and  it extends to a 
homeomorphism between suitably defined neighbourhoods in the respective one parameter moduli spaces.

Following the recent proof by Petersen and Roesch that $\mathcal{M}_1$ is homeomorphic to the classical Mandelbrot set $\mathcal{M}$, we deduce that $\m$ is homeomorphic to $\mathcal{M}$.
\end{abstract}

\section{Introduction}\label{intro}
Parallels between the dynamical behaviour of iterated rational maps and Kleinian groups have been evident since the work of Fatou and Julia
over a century ago: see Sullivan's celebrated {\it dictionary} between the two areas in \cite{S}. In 1994 the first examples of iterated holomorphic correspondences
on the Riemann sphere behaving as {\it matings} between a rational map and a Kleinian group were exhibited by the first author together with Christopher Penrose \cite{BP}.  
The rational map was a quadratic polynomial $Q_c:z\to z^2+c$ and the Kleinian group was the modular group $\Gamma=PSL(2,\Z)$, equipped with the generators 
$\alpha(z)=z+1$ and $\beta(z)=z/(z+1)$.

In \cite{BP} a $(2:2)$ holomorphic correspondence $\F:z \to w$ is termed a {\it mating} between $Q_c$ and $PSL(2,\Z)$ if there exists a completely invariant open 
simply-connected subset $\Omega$ of the sphere such that:
\begin{enumerate}
\item  there is a conformal bijection 
 $\phi: \Omega \rightarrow \H$ conjugating $\F|_\Omega$ to $\alpha |_\H$ and $\beta |_\H$ (where $\H$ denotes the upper half-plane), and
 \item $\widehat \C \setminus \Omega =  \Lambda_- \cup \Lambda_+$, where $\Lambda_- \cap \Lambda_+$ is a single point 
 and there exist homeomorphisms $\varphi_{\pm}: \Lambda_{\pm} \rightarrow K_c$ conjugating
 $\F|_{\Lambda_-}$ to $Q_{c}|_{K_c}$ and  $\F|_{\Lambda_+}$ to $Q^{-1}_{c}|{K_c}$ respectively.
 \end{enumerate}
The matings exhibited in \cite{BP} were in the family of 
$(2:2)$ holomorphic correspondences $\F_a:z\to w$ defined by the polynomial relation obtained from the following equation by multiplying through by denominators:
\[
\left(\frac{aw-1}{w-1}\right)^2+\left(\frac{aw-1}{w-1}\right)\left(\frac{az+1}{z+1}\right)
+\left(\frac{az+1}{z+1}\right)^2=3 \tag{1.1}
\]
See Section \ref{corr} below for an explanation of why a mating between a quadratic rational map and $PSL(2,\Z)$ has to be in this family, up to conformal conjugacy. 
The {\it limit set} $\Lambda(\F_a)=\Lambda_-(\F_a)\cup\Lambda_+(\F_a)$ of the correspondence $\F_a$ is defined for all $a$ in the {\it Klein combination locus}
${\mathcal K}$ (see Section 2 or \cite{BL1} for the definition of $\mathcal K$).
The {\it connectedness locus} $\mathcal{C}_{\Gamma}$ for the family $\F_a$ is the 
set of values of the parameter $a\in{\mathcal K}$ for which the limit set $\Lambda(\F_a)$ is connected. %(see Section 2 or \cite{BL1}).
Denoting by
$\overline{\D}(4,3)$ the closed disc with centre $a=4$ and radius $3$, the intersection
$\m:={\mathcal C}_\Gamma\cap \overline{\D}(4,3)$ is called the \textit{modular Mandelbrot set}.
In \cite{BP} it was 
%proved that, for $a=4$, the correspondence $\F_a$ is a mating between $PSL(2,\Z)$ and a quadratic polynomial $Q_c$ with $c\in \mathcal M\cap \R$, and it was
conjectured that for every $a\in \m$ 
the correspondence $\F_a$ is a mating 
between $PSL(2,\Z)$ and some $Q_c$ with $c\in \mathcal{M}$, the classical Mandelbrot set, and that $\m$ is homeomorphic to  $\mathcal M$. 
A computer plot of $\m$ is displayed in Figure \ref{mg}. 

Considerable progress was made on the first conjecture using `pinching' techniques (see  \cite{BH}),
but it was only fully resolved (in \cite{BL1}) with the application of the theory of `parabolic-like mappings'  \cite{L1}, developed 
by the second author of the present paper. Her key observation was that
the family $\F_a$ presents a persistent \textit{parabolic} fixed point $P_a$ of multiplier $1$ at the intersection of $\Lambda_{-}$ and $\Lambda_{+}$, which indicates that a family of maps with a persistent parabolic fixed point may be a better model than quadratic polynomials, that family being $Per_1(1)$, that is to say maps which are conjugate to some 
$P_A(z)=z+1/z+A,\,\,\, A \in \C$. These maps are the parabolic analogues of quadratic polynomials (for which there is a persistent super-attracting fixed point). The parabolic Mandelbrot set $\mathcal{M}_1$ (Figure \ref{mg}, centre) is the connectedness locus of this family: the set of 
$B=1-A^2\in \C$ such that the complement $K_A$ of the parabolic basin of attraction of infinity is connected. The parabolic Mandelbrot set $\mathcal M_1$ was recently proved to be homeomorphic to $\mathcal{M}$ (\cite{PR}).

 \begin{figure}

\begin{center}
 \begin{tabular}{rcl}

% \begin{minipage}{3.5cm}
%\centering
%\includegraphics[width= 3.5cm]{MG.eps}
% \end{minipage}
 % \hspace{0.5cm}
% \begin{minipage}{3.5cm}
%\centering
%\includegraphics[width= 3.5cm]{M_1.eps}
% \end{minipage}
%  \hspace{0.5cm}

% \begin{minipage}{3.5cm}
%\centering
%\includegraphics[width= 3.5cm]{ma.eps}

\begin{minipage}{3.5cm}\centering
\includegraphics[width= 3.2cm]{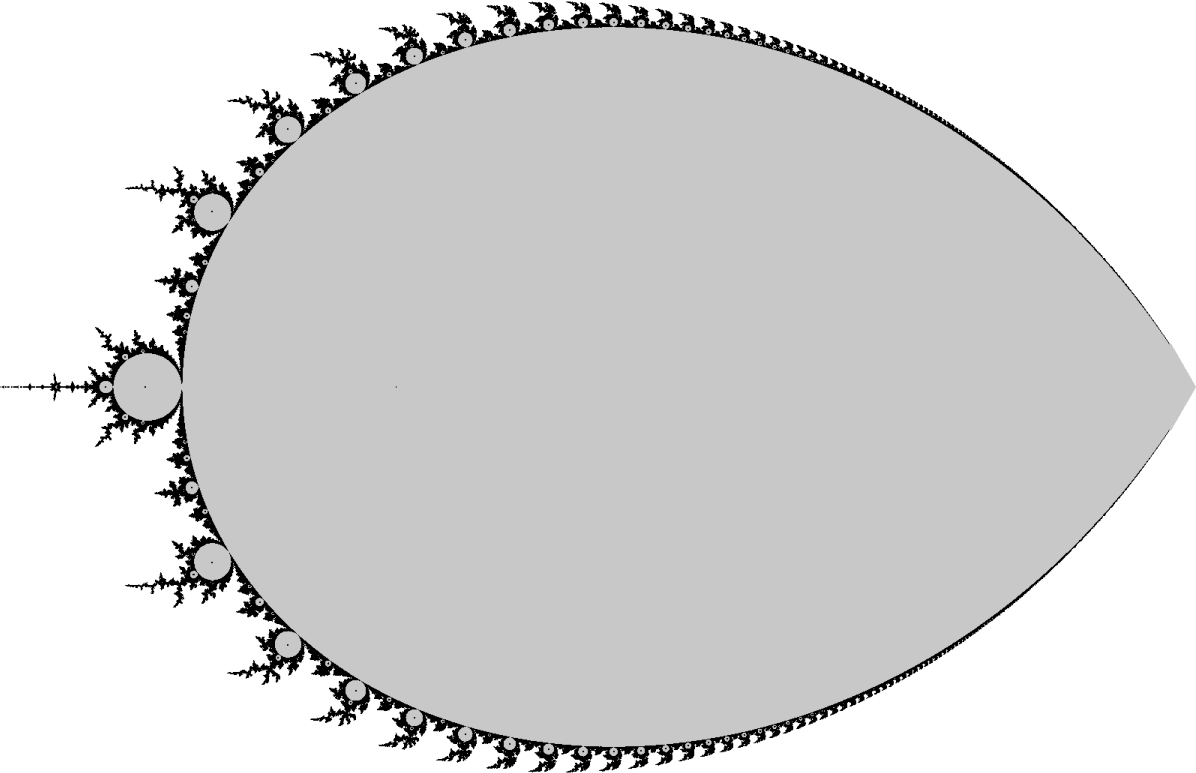}
 \end{minipage}
  \hspace{0.25cm}
\begin{minipage}{3.5cm}
\centering
\includegraphics[width= 2.9cm]{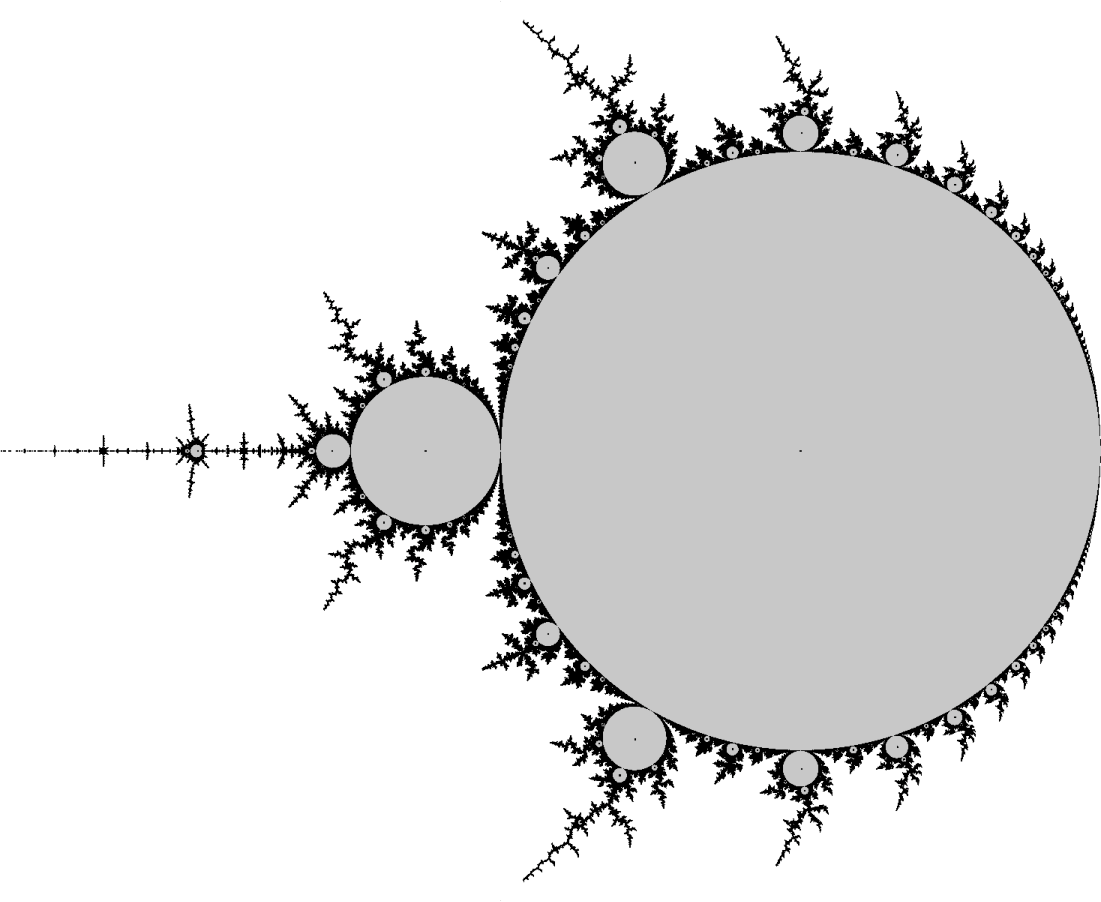}
 \end{minipage}
  \hspace{0.25cm}

  \begin{minipage}{3.5cm}
\centering
\includegraphics[width= 3.2cm]{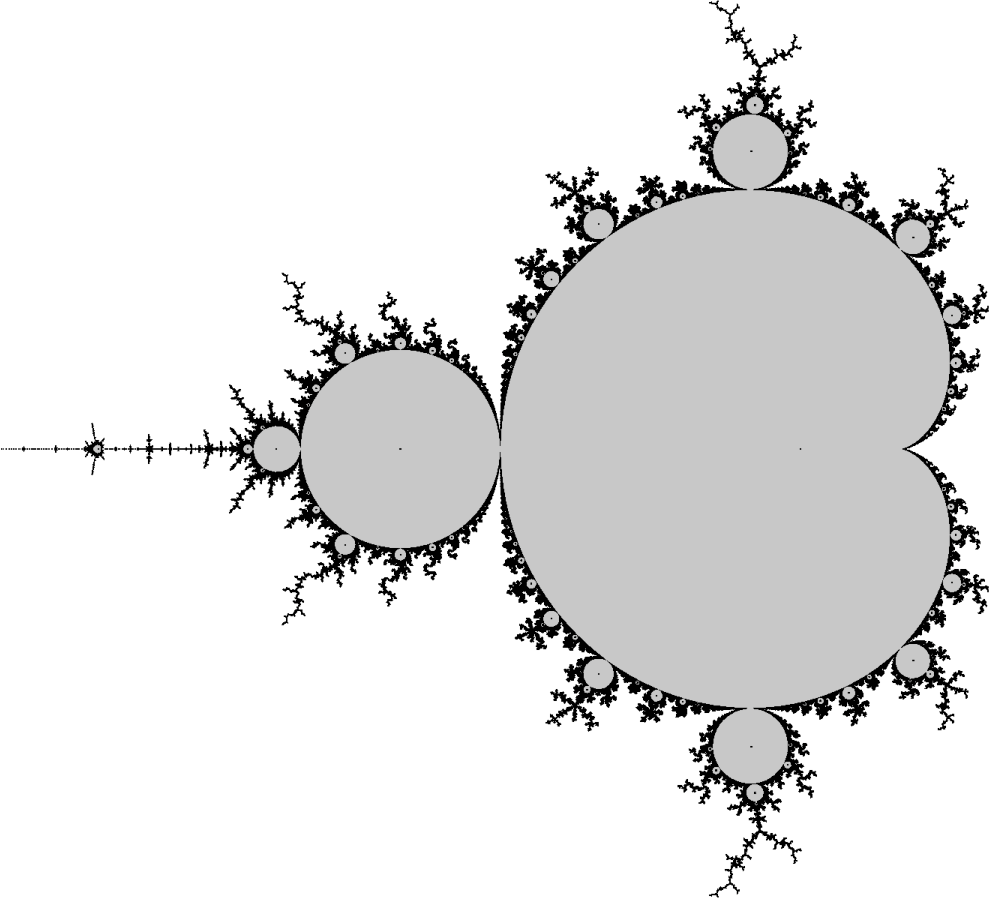}

 \end{minipage}

\end{tabular}
\end{center}
\caption{\small On the left, the modular Mandelbrot set $\m$, which we prove in this paper to be 
homeomorphic to the 
parabolic Mandelbrot set $\mathcal{M}_1$ (in the centre). On the right is the classical Mandelbrot set $\mathcal M$. 
Dynamical space plots of $\F_a$, for $a$ at the centres of three components of $\m$, are displayed in Figure \ref{per123}, while dynamical plots of the elements of $Per_1(1)$ at the centres of the corresponding components of $\mathcal{M}_1$ are displayed in Figure \ref{perplots}. The homeomorphism $\chi:\m \to \mathcal{M}_1$ sends centres to centres. }
\label{mg}
\end{figure} 

%In \cite{BL1} we proved  that for every $a \in {\mathcal C}_\Gamma$, the connectedness locus of the family, the correspondence $\F_a$ is a mating 
%between some $P_A$ (with $B=1-A^2\in \mathcal M_1$) and the modular group. Our proof employed surgery to convert $\F_a$ into a {\it parabolic-like map}, 
%followed by the {\it Straightening Theorem} of \cite{L1} to change this map into the form $P_A$.\\

In the present paper we resolve the second conjecture proposed in \cite{BP}:

\begin{mteor}\label{main_thm}
The modular Mandelbrot set $\m$ is homeomorphic to the parabolic Mandelbrot set $\mathcal{M}_1$, via a homeomorphism $\chi$ which has the properties: 

(i) For every $a\in \m$, the correspondence $\F_a$ is a mating between the rational map $P_{\chi(a)}$ and the modular group $PSL(2,\Z)$; 

(ii) $\chi$ is a conformal homeomorphism between the interior $\mathring{\m}$ of $\m$ and the interior $\mathring{M}_1$ of $\mathcal{M}_1$;

(iii) when restricted to the complement of a closed neighbourhood (in the induced topology) $N\subset\m$ of the root point $a=7$ of $\m$, the homeomorphism $\chi$
extends to a homeomorphism between an open set containing $\m\setminus N$ in the $a$-plane and an open set containing $\mathcal{M}_1\setminus\chi(N)$
in $Per_1(1)$.

\end{mteor}

Observe that (iii) tells us that the limbs and sublimbs of $\m$ are laid out in the $a$-plane in combinatorially the same arrangement as the corresponding limbs and sublimbs of $\mathcal{M}_1$. Another consequence of the Main Theorem is:

\begin{mcor}\label{whole}
The modular Mandelbrot set $\m$ is the whole of the connectedness locus $\mathcal{C}_{\Gamma}$. \end{mcor}

Thus the conjectures of \cite{BP} have all been answered positively, once they have been translated into what is now 
clearly the right setting by substituting
`parabolic quadratic rational maps' for `quadratic polynomials', and `$\mathcal{M}_1$'  for `$\mathcal M$'. 
Moreover, following the recent proof by Petersen and Roesch \cite{PR} of Milnor's conjecture that $\mathcal{M}_1$ is homeomorphic to the classical Mandelbrot set $\mathcal M$,
our Main Theorem proves that $\m$ is homeomorphic to $\mathcal{M}$, as originally conjectured in \cite{BP}.

There is as yet no complete combinatorial description of the classical Mandelbrot set $\mathcal M$, as such a description has to await the resolution of 
MLC, the celebrated conjecture that $\mathcal M$ is locally connected. The fact that  
$\mathcal{M}$ and $\m$ have now been shown to be homeomorphic offers the 
possibility of exploring $\mathcal M$ through another model, one which for example has a different set of Yoccoz inequalities, 
yet to be fully worked out, governing the sizes of its limbs
and sub-limbs (see \cite{BL2}).\\

\noindent{\textbf{The anti-holomorphic case.} 
In independent work, Lee, Lyubich, Makarov and Mukherjee investigated anti-holomorphic matings in a recent article \cite{LLMM}. They established the existence of a homeomorphism 
between a combinatorial model of the connectedness locus of a certain family of anti-holomorphic maps (matings, in the sense of \cite{LLMM}), and a combinatorial model of the Tricorn, the anti-holomorphic analogue of the Mandelbrot Set. Their proof introduces an alternative `straightening’ strategy, different from the approach in \cite{L1} and the present paper.}\\

\noindent{\textbf{Strategy of proof and layout of paper.}}
%invariant arcs emerging from the parabolic fixed point
In Sections \ref{corr} and \ref{Per} we summarise the background facts needed concerning the family $\F_a$ and the family $Per_1(1)$. In Section \ref{surgery} we develop a surgery construction which converts $\F_a$ into a quadratic rational map in the family
$P_A$, without the intermediate parabolic-like stage of \cite{BL1}, and moreover does so uniformly with the parameter. 
The basic idea is to replace the complement 
of the backwards limit set $\Lambda_-(\F_a)$ in $\widehat \C$ by a copy of the basin $\A_A(\infty)$ of the parabolic fixed point $\infty$ of a rational map of the form
$P_A$, so that $\Lambda_-(\F_a)$ becomes the filled Julia set of such a $P_A$. As a model for the parabolic basin we take the Blaschke product 
$$h(z)=\frac{z^ 2+1/3}{z^2/3+1}\ \  \rm{on}\ \ \widehat \C\setminus \overline \D.$$
We borrow a stratagem from \cite{L1} to overcome the difficulty of gluing a smooth set outside a set with cusps, and so realising the replacement: for every $\F_a$ we construct forward invariant arcs $\gamma_a$ emanating from the parabolic fixed point, which move holomorphically with $a$, and which in a neighbourhood of $\Lambda_{a,-}$ separate the expanding dynamics of $\F_a$ from the parabolic dynamics. We choose $a_0 \in \mathring \m$ and glue the dynamics of $h$ outside a set $O_{a_0}$ containing the limit set $\Lambda_{a_0}$ for $\F_{a_0}$ and having $\gamma_{a_0}$ on the boundary, and to obtain uniformity with respect to the parameter we ensure the boundaries of $O_a$ move holomorphically with respect to $a$.
%, separate the dynamics of we \textit{dividing arcs} $\gamma_a$
%and to 
%To realise this replacement operation by surgery, in principle what we need at a chosen base point $a_0$ in parameter space is a quasiconformal homeomorphism 
%between a fundamental domain for $\F_{a_0}$, encircling $\Lambda_{a,-}$, and a fundamental domain for $h$, encircling $\D$. 
%obtain uniformity with respect to the parameter we carry out this surgery on sets which move holomorphically with respect to $a$,, forward invariant arcs belonging to repelling petals for the parabolic fixed point and moving holomorphically with the parameter.
To start the construction of $O_a$, we first choose 
a (pinched) neighbourhood of $\m\setminus\{7\}$ in parameter space in which to work. This is a lune $\L\f$, the open set bounded by two arcs of circles intersecting with 
angle $2\theta$ between them at the points $a=1$ and $a=7$ (the root point of $\m$). The existence of such a neighbourhood for some value of $\theta$ in the half-open interval 
$[\pi/3,\pi/2)$ was proved in \cite{BL2} as an application of a new Yoocoz inequality derived there. In the Appendix to the present paper we show that for every 
$a\in\L\f$ there exists a {\it dynamical space lune} $V_a$ which moves holomorphically with $a$ 
and contains $\Lambda (\F_a)$. To guarantee the holomorphic motion of the arcs $\gamma_a$ we restrict our attention to a doubly truncated lune 
$\mathring K \subset \L\f$ (Section \ref{dtl}). The truncation at the end $a=7$ of the lune $\L\f$ is by the removal of an arbitrarily small disc 
neighbourhood $N$ of the root point, but rather than overload the notation $K$ with subscripts or superscripts we invite the reader to keep in mind
that our holomorphic motions parametrised by $\mathring K$, and hence our subsequent surgery construction for $\F_a,\  a\in \mathring K$ and definition of
$\chi:\K \to Per_1(1)$, are all a priori dependent on $N$.

In Section \ref{parlune} we define the {\it central set} $O_a$ to be the connected component of $\F_a^{-1}(\overline{V}_a)\setminus \gamma_a$ containing the limit set 
$\Lambda_{a,-}$, and in Section \ref{holom} we construct a holomorphic motion of the complement $\widehat \C \setminus O_a$ of this set.
%The set moves holomorphically
In Section \ref{exmap} we choose a base point $\,a_0 \in \mathring K\cap\m$ for our holomorphic motion 
and construct an `external map' $g$ for $\F_{a_0}$, by conjugating the dynamics of $\F_{a_0}$ by the uniformization map $\alpha$ from the complement 
of the limit set $\Lambda_{a_0,-}$ to the complement of the closed unit disc, and we define the sets $O_g$ and $\gamma_g$ for $g$ by taking the 
image under $\alpha$ of the respective sets for $\F_{a_0}$. 
Using Fatou coordinates (see Lemma \ref{arcs}) we construct an equivariant quasi-symmetric map between $\gamma_g$ and the
corresponding forward invariant arc $\gamma_h$ for $h$: extending this we define a quasiconformal map $f$ between the complement of $O_g$ and 
the complement of the corresponding set $O_h$ for $h$. Then, precomposing $f$ by $\alpha$, we glue the dynamics of $h$ on the complement of $O_h$ 
to the dynamics of $F_{a_0}$ on $O_{a_0}$.
Composing $f$ with our holomorphic motion of the complement of $O_a$, we obtain the surgery construction for the whole family $\F_a$ (Section \ref{wf}).

In Section \ref{chi} we define the map  $\chi:\m \to \mathcal{M}_1$, using the fact that when $a\in \m$ the surgery yields a unique  
$P_A\in \mathcal{M}_1$ hybrid equivalent to $\F_a$, and that by Proposition \ref{welldefined} this $P_A$ does not depend  on the choice of dividing curves %$\gamma_{a_0,i},\,i=\{1,2\},$ 
or the various choices of quasiconformal homeomorphism made during the surgery construction.
We then prove that $\chi:\m \to \mathcal{M}_1$ is injective (Proposition \ref{inj}).

While $\chi$ is also well-defined on $\K\setminus \m$, its definition there depends on choices made during the surgery, and it is by no means obvious a priori that $\chi$ is injective on the whole of $\K$, nor that it is continuous there.
But in Section \ref{extension} we show that on $\K\setminus \m$ the map $\chi$ can be interpreted as the same map (Proposition \ref{same}) as that obtained from
the position of the critical value  of $\F_a|_{V_a}$ in analogy to Douady and Hubbard's map in their analysis of families of polynomial-like maps \cite{DH}. 
By applying Lyubich's formulation (Chapter 6 of \cite{Lyu}) of the methods introduced by Douady and Hubbard, we prove that with suitable choices made in the initial surgery
at $a=a_0$ the extension of $\chi$ from $\m$ to $\K\setminus\m$ is locally quasiregular (Proposition \ref{extnqr}), that
%In Section \ref{continuity} we again follow 
%Douady and Hubbard, as formalised by Lyubich, employing the Mañé-Sad-Sullivan decomposition to prove continuity of 
$\chi$ is continuous on the boundary of $\m\cap \K$ (Proposition \ref{conti}), and that
$\chi$ is holomorphic on the interior of $\m\cap \K$ (Propositions \ref{hypercom} and \ref{queer}).
Finally we deduce  
%. From the continuity of $\chi$ on $\K$, its injectivity on $\m\cap \K$, its
%local quasiregularity on $\K\setminus\m$, and elementary properties of degrees of proper maps between surfaces, we deduce that 
that $\chi$ maps 
$\K$ homeomorphically onto its image in $Per_1(1)$ (Proposition \ref{homeo_on_lune}).

Since the surgery preserves limit sets, and hence their connectedness or otherwise, $\chi(\m \setminus\{7\})$ is contained in $\mathcal{M}_1\setminus\{1\}$. 
However it is not obvious that $\mathcal{M}_1\setminus\{1\} \subset \chi(\m \setminus\{7\})$. 
If the root $a=7$ of $\m$ were contained in the domain $\K$ of our holomorphic motion, we could apply $\chi$ to a loop encircling $\m$ and
it would follow that $\mathcal{M}_1 \subset \chi(\m)$, from the known connectedness of $\mathcal{M}_1$. %But the holomorphic motion is not defined at $a=7$.
%It would have been obvious if the root $\{7\}$ was contained in our parameter plane, as we could have taken a loop starting and ending at the root and encircling $\m$, but we cannot because the root is missing.
But $\K$ does not contain the root of $\m$. We get around this problem by
decomposing $\m$ into its main hyperbolic component and limbs $L_{p/q}$ (Proposition \ref{limb_properties}), and encircling each limb 
with a loop starting and ending at the root of the limb (see Lemma \ref{enc}). In Section \ref{final}
we complete the proof of the Main Theorem 
by proving (Corollary \ref{chi_on_roots}) that $\chi: \m\setminus\{7\} \to \mathcal{M}_1\setminus\{1\}$ extends continuously to the respective root points 
$a=7$ and $B=1$.
%, which lie outside $\K$ of the holomorphic motion. 
%Corollary \ref{chi_on_roots}) are 
%proved in Sections \ref{chi_surj} and \ref{final} using properties of the limb structures of $\m$ and $\mathcal{M}_1$.  
In Section \ref{proofcor} we deduce the Corollary that $\m$ is the whole of the connectedness locus of the family $\F_a$. \\% ${\mathcal C}_\Gamma$ of the family $\F_a$.\\

\textbf{Acknowledgments.}
This research has been partially supported by the 
Funda\c{c}\~ao de amparo a pesquisa do estado de S\~ao Paulo (Fapesp, processes 
2016/50431-6, 2017/03283-4), the Cnpq (406575/2016-19), the prize L'ORÉAL-UNESCO-ABC Para Mulheres na Ciência
and the Serrapilheira Institute (grant number Serra-1811-26166).

\section{The families $\F_a$ and $Per_1(1)$}
\subsection{The family $\F_a$}\label{corr}
A holomorphic correspondence on $\widehat \C$ is a multivalued map defined by a polynomial relation $P(z,w)=0$. The correspondence is $(n:m)$ when the polynomial $P(z,w)$ 
which defines it has degree $n$ in $z$ and $m$ in $w$, meaning that each $z$ has $m$ corresponding (images) $w$ and each $w$ has $n$ corresponding (pre-images) $z$.
In particular, a holomorphic $(2:2)$ correspondences on  $\widehat \C$ in a $2$-valued map (with a $2$-valued inverse map)
$$\F:z \rightarrow w$$
defined implicitly by an equation $P(z,w)=0$ where $P$ has the form
$$P(z,w)= (az^2+bz+c)w^2+(dz^2+ez+f)w+ gz^2+hz+j.$$
The family of holomorphic $(2:2)$ correspondences introduced in \cite{BP}  are the implicit functions $\F_a: z\to w$ defined by
the equation (1.1) (Section \ref{intro}), repeated here for the convenience of the reader:
$$\left(\frac{aw-1}{w-1}\right)^2+\left(\frac{aw-1}{w-1}\right)\left(\frac{az+1}{z+1}\right)+\left(\frac{az+1}{z+1}\right)^2=3,$$
where $a\in \C$.
In the current paper such a correspondence will usually either be expressed in terms of the coordinate $z$, as in equation (1.1), 
or in terms of the  coordinate
$$Z=\frac{az+1}{z+1},$$
in which the equation of the correspondence becomes
$$J_a(W)^2+Z.J_a(W)+Z^2=3,$$
where $J_a$ is the unique conformal involution which has fixed points at $Z=1$ and $Z=a$. In the coordinate $z$, the conformal involution
 has fixed points $z=0$ and $z=\infty$, and there it is $J_a(z)\equiv J(z)=-z$. To simplify computations, we will also use the coordinates 
$\zeta=Z-1$ in Proposition \ref{divarcs}, and the coordinates
$z'=(a-1)z$ in Propositions \ref{truncatelune} and  \ref{divarcs}, and for Lemma \ref{hollunes}.

When written in terms of the coordinate $Z$,
it is apparent that $\F_a:Z \to W$ can also be written in the form
$$W=J_a\circ Cov_0^Q(Z),$$
where $Cov_0^Q$ is the {\it deleted covering correspondence} of the cubic polynomial $Q(Z)=Z^3-3Z$, that is to say 
(see \cite{BL1}) $Cov_0^Q$ is the correspondence 
$$Z\to W \ \Leftrightarrow\ \frac{Q(W)-Q(z)}{W-Q}=0.$$
Why investigate this particular family of correspondences? 
The reason is that every $(2:2)$ holomorphic correspondence $\F$ with the property that 
$\widehat \C$ is
partitioned into completely invariant $\Omega$ and $\Lambda$ as above, with a conjugacy between $\F$ on $\Omega$ and 
$PSL(2,\Z)$ on $\H$, is conformally conjugate to some $\F_a$. To see why this is so,
observe that the $(2:2)$ correspondence defined by $\alpha:z\to z+1$ and $\beta:z \to z/(z+1)$ on $\H$ satisfies the `diagram condition' below,
since $\alpha\circ\beta^{-1}\circ\alpha=\beta\circ\alpha^{-1}\circ\beta$.

\begin{picture}(200,100)

\put(150,20){$z_3$}
\put(150,44){$z_2$}
\put(150,68){$z_1$}
\put(184,20){$w_3$}
\put(184,44){$w_2$}
\put(184,68){$w_1$}
\put(162,22){\vector(1,1){20}}
\put(162,26){\vector(1,2){20}}
\put(162,42){\vector(1,-1){20}}
\put(162,50){\vector(1,1){20}}
\put(162,66){\vector(1,-2){20}}
\put(162,70){\vector(1,-1){20}}
\label{mot}
\end{picture}\\
This diagram condition says that the two `zig-zag' orbits (forwards, then backwards, then forwards) starting at any initial $z_i$ 
arrive at the same destination $w_i$. An alternative description is given by regarding an `arrow' in the diagram as a point of the 
curve (surface)  $P(z,w)=0$, for example the arrow $z_1\to w_2$ 
corresponds to the point $(z_1,w_2)$. The curve, the graph of the correspondence, comes equipped with covering involutions, 
$I_+$ and $I_-$:
$$I_+(z,w)=(z,w') \mbox{\ \ where \,\,\,} P(z,w')=0,$$
$$I_-(z,w)=(z',w) \mbox{\.\ where \,\,\,} P(z',w)=0,$$
The diagram condition above is equivalent to asking that 
$$( I_- \circ I_+)^3=( I_- \circ I_+)\circ( I_- \circ I_+)\circ( I_- \circ I_+) =Id.$$
For this condition to hold on $\Omega$, it must (by analytic continuation) hold on the whole of $\widehat \C$.
This implies that our correspondence factorises as a deleted covering correspondence $Cov_0^Q$ of a degree $3$ rational map $Q$, 
(the rational map identifying together the points $z_1, z_2$ and $z_3$ in the diagram condition)
followed by a Mobius transformation $M$ (the zig-zag map sending each $z_i$ to the corresponding $w_i$). Thus
$$\F=M \circ Cov_0^Q.$$
Moreover on $\Omega$ the M\"obius transformation $M$ restricts to an involution (since $\alpha\circ\beta^{-1}\circ\alpha$ is an 
involution on $\H$), so, again by analytic continuation, $M$ must  be an involution (which we shall denote by $J$ in this paper) on 
$\widehat \C$, and thus
$$\F=J \circ Cov_0^Q.$$
Every degree $3$ rational map has $4$ critical points (counted with multiplicity). Our cubic $Q$ must have a double critical point (this corresponds to the point of ${\H}$ where $z \to z+1$ and $z \to z/(z+1)$ coincide), and two single critical points (else the correspondence 
would factorise into a pair of M\"obius transformations), and as we can post-compose $Q$ by any M\"obius transformation without altering
$Cov_0^Q$, we can normalise $Q$ to the specific polynomial 
$$Q(Z)=Z^3-3Z$$
which is what it will be for the rest of this paper.
Finally, we note that one consequence of the conditions we have imposed on the restriction of $\F$ to $\Lambda_+$ and 
$\Lambda_-$ is that $\Lambda_-\cap \Lambda_+=\{P\}$ is a fixed point of $\F$ (that is, $P\in \F(P)$), and as $P$ is also a fixed point
of $J$ (since $J$ is the zig-zag map, and it is easily seen that our conditions imply that the zig-zag map sends $\Lambda_-$ to $\Lambda_+$) we deduce that $P$ must be fixed by $Cov_0^Q$,
and it is therefore one of the critical points $Z=\pm1$ of $Q$. Replacing the $Z$-coordinate by $-Z$ if necessary, we can choose
this critical point to be $Z=+1$. So $J=J_a$ has fixed points $Z=1$ and $Z=a$. 

Having answered the question of why consider the family $\F_a$, our next task is to show that for some $a\in \C$ the Riemann sphere
$\widehat \C$ can indeed be partitioned into a completely invariant open set $\Omega$ and a completely invariant closed set $\Lambda$ with
the properties we would like.
By construction, for certain $a \in \C$, there exists a fundamental domain for $\F_a$, which 
we can obtain by intersecting fundamental domains for 
$Cov_0^Q$ and $J_a$, constructed as follows.
   \begin{figure}[hbt!]
\centering
 \psfrag{a}{\tiny $a$}
 \psfrag{A}{\tiny $\overline \Delta_J^{st}$}
  \psfrag{-2}{\tiny $-2$}
 \psfrag{1}{\tiny $1$}
  \psfrag{2}{\tiny $-2$}
   \psfrag{D}{\tiny $\Delta_{Cov}^{st}$ fundamental domain for $Cov_0^Q$}
 \psfrag{C}{\tiny $Cov^Q_0$}
 \psfrag{L}{}
\includegraphics[width= 3cm]{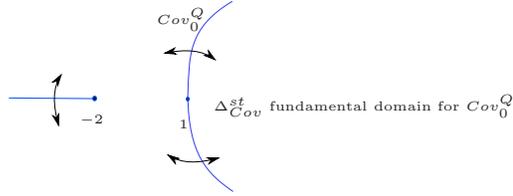}
\caption{\small Construction of standard fundamental domain for $Cov_0^Q$.}
\label{costfd}
  \end{figure}
The point $Z=1$ is a critical point for $Q$ with co-critical point $Z=-2$ (that is $Q(1)=Q(-2)$). The three blue half-lines
in Figure \ref{costfd} are the inverse image $Q^{-1}((-\infty,-2])$ of the half line $(-\infty,-2]$ on the quotient (image) Riemann sphere. Thus
the deleted covering correspondence  
$Cov_0^Q$ of $Q$ sends each of these half-lines to the other two. 
The open set $\Delta_{Cov}^{st}$ bounded by the two half-lines starting at $Z=1$ and
running off to $\infty$ at angles $\pm \pi/3$ with the positive real axis,
and containing the point $Z=2$, is the {\it standard} fundamental domain for $Cov_0^Q$. The complement of the round disc with boundary passing through $Z=1$ and $Z=a$ (illustrated in Figure \ref{fd3}, on the left) is the {\it standard} fundamental domain for $J_a$, which we 
denote $\Delta_J^{st}$.

 \begin{figure}

\begin{center}
 \begin{tabular}{rcl}
 \begin{minipage}{4.5cm}
\centering
 \psfrag{a}{\tiny $a$}
 \psfrag{A}{\tiny $ \Delta_J^{st}$}
  \psfrag{-2}{\tiny $-2$}
 \psfrag{1}{\tiny $1$}
  \psfrag{2}{\tiny $2$}
 \psfrag{C}{\tiny $\Delta_{Cov}^{st}$}
 \psfrag{L}{}
\includegraphics[width= 4cm]{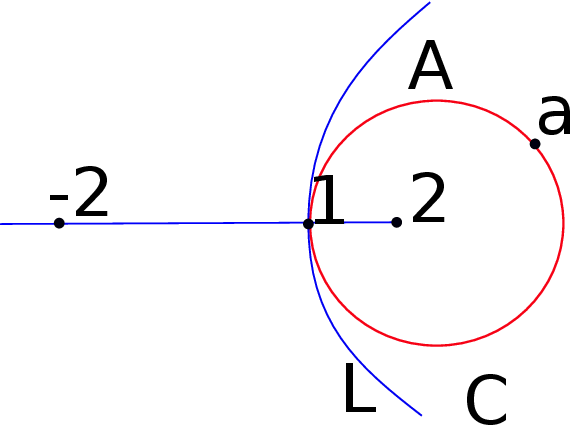}
 \end{minipage}
  \hspace{1.5cm}

  \begin{minipage}{4.5cm}
 \centering
 \psfrag{a}{\tiny $a$}
 \psfrag{A}{\tiny $ \Delta_J^{st}$}
  \psfrag{F}{\tiny $\F_a^{-1}(\Delta_J^{st})$}
 \psfrag{e}{\tiny $\F_a^{-2}(\Delta_J^{st})$}
  \psfrag{2}{\tiny $(2:1)$}
 \psfrag{1}{\tiny $(1:2)$}
  \psfrag{3}{\tiny $(1:1)$}
 \centering
\includegraphics[width= 4cm]{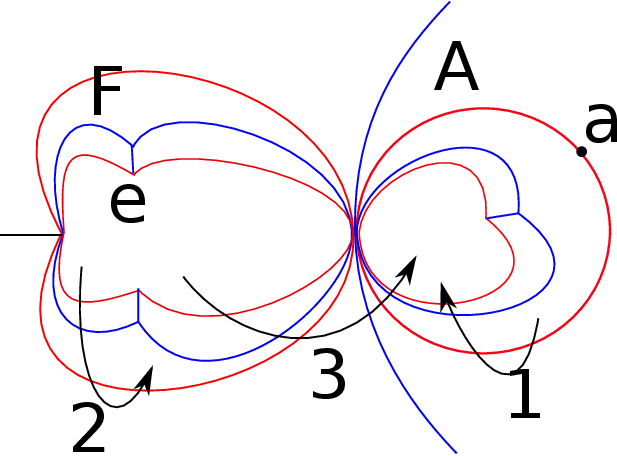}
 \end{minipage}

\end{tabular}
\end{center}
\caption{\small On the left, standard fundamental domains for $Cov_0^Q$ and $J_a$. On the right, images and preimages of fundamental domains.}
\label{fd3}
\end{figure}

 When $a\in {\overline \D(4,3)}\setminus \{1\}$ the complements of $\Delta_{Cov}^{st}$ and $\Delta_J^{st}$ intersect 
 in the single point $Z=1$
 and 
 $$\Delta_{corr}^{st}:= \Delta_J^{st} \cap \Delta_{Cov}^{st}$$ is a fundamental domain for $\F_a$ on the union $\Omega$ of all images 
 of $\Delta_{corr}^{st}$ under (mixed) iteration of $\F_a$ and $\F_a^{-1}$ (the correspondence $\F_a$ is undefined at $a=1$). More generally,
 the \textit{Klein combination locus} $\mathcal{K}$ is the set of parameters $a \in \C$ for which there exist fundamental domains
 $\Delta_J$ and $\Delta_{Cov}$, bounded by Jordan curves and such that $\Delta_J \cup \Delta_{Cov}= \widehat \C \setminus \{1\}$.
 For every $a\in \mathcal{K}$ the set 
$$\Delta_{corr}:= \Delta_J \cap \Delta_{Cov}$$ is a fundamental domain for $\F_a$ acting on the union of all images of $\Delta_{corr}$
(see \cite{BL1}). Note that $\mathcal K$ contains ${\overline \D(4,3)}\setminus\{1\}$.

By construction, for every $a\in \mathcal{K}$ we have the following behaviour: $Z=1$ is a parabolic fixed point for $\F_a$ of multiplier $1$; both images of the complement of $\Delta_J$ belong to the complement of $\Delta_J$, that is: $\F_a((\Delta_J)^c) \subset (\Delta_J)^c$
and the restriction of $\F_a$ to $\Delta_J$ is a ($1:2$) correspondence; while both pre-images of $\Delta_J$ lie inside $\Delta_J$, 
so the restriction of $\F_a$
to the preimage of $\Delta_J$ is a $(2:1)$ map (see Proposition 3.4 in \cite{BL1} and Figure \ref{fd3}, on the right).
The \textit{forward limit set} $\Lambda_{a,+}$ for $\F_a$ is defined as
$$\Lambda_{a,+}=\bigcap_{n=0}^\infty \F_a^n(\hat{\mathbb C}\setminus\Delta_{J}),$$
the \textit{backward limit set} $\Lambda_{a,-}$ for $\F_a$ is defined to be
$$\Lambda_{a,-}=\bigcap_{n=0}^\infty \F_a^{-n}(\overline{\Delta}_{J})=J(\Lambda_{a,+})$$
and the \textit{limit set} $\Lambda_a$ for $\F_a$ 
is defined to be
$\Lambda_a=\Lambda_{a,+}\cup\Lambda_{a,-}$ (by Proposition 3.4 in \cite{BL1} we have $\Lambda_{a,+}\cap \Lambda_{a,-}=\{1\}$).
The regular set $\Omega_a$ is defined to be $\widehat{\mathbb C}\setminus\Lambda_a$: it is tiled by the images of $\Delta_{corr}$.
The limit set of $\F_a$ is shown in grey in Figure \ref{per123} for three different values of $a$: the pictures are plots in the plane of 
the coordinate $z$, the red and blue lines are the boundaries of the standard domains (transferred from the $Z$-coordinate to the 
$z$-coordinate), and their images under mixed iteration of $\F_a$.\\

\begin{figure}
\begin{center}
\centering
\includegraphics[width=3.5cm]{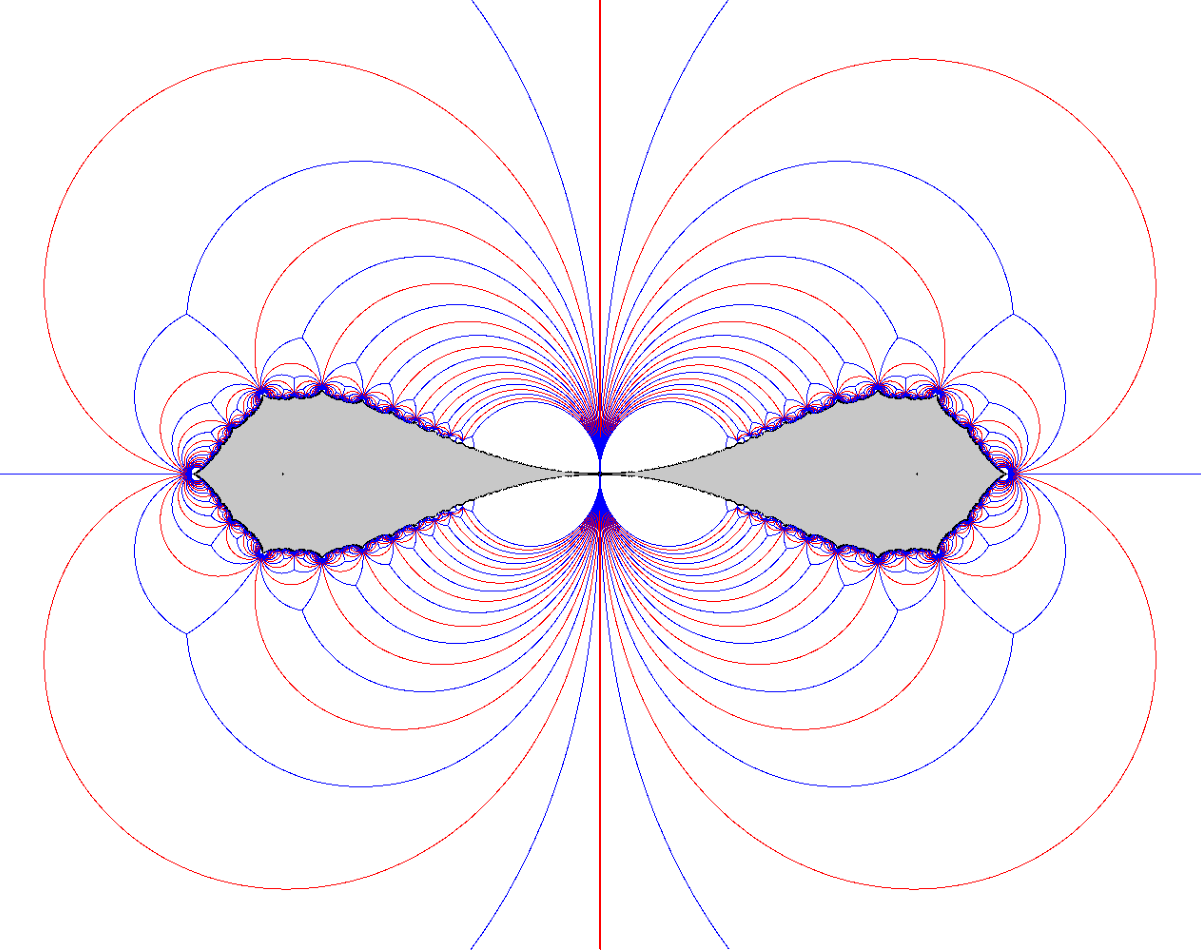}
\hspace{0.5cm}
\centering
\includegraphics[width=3.5cm]{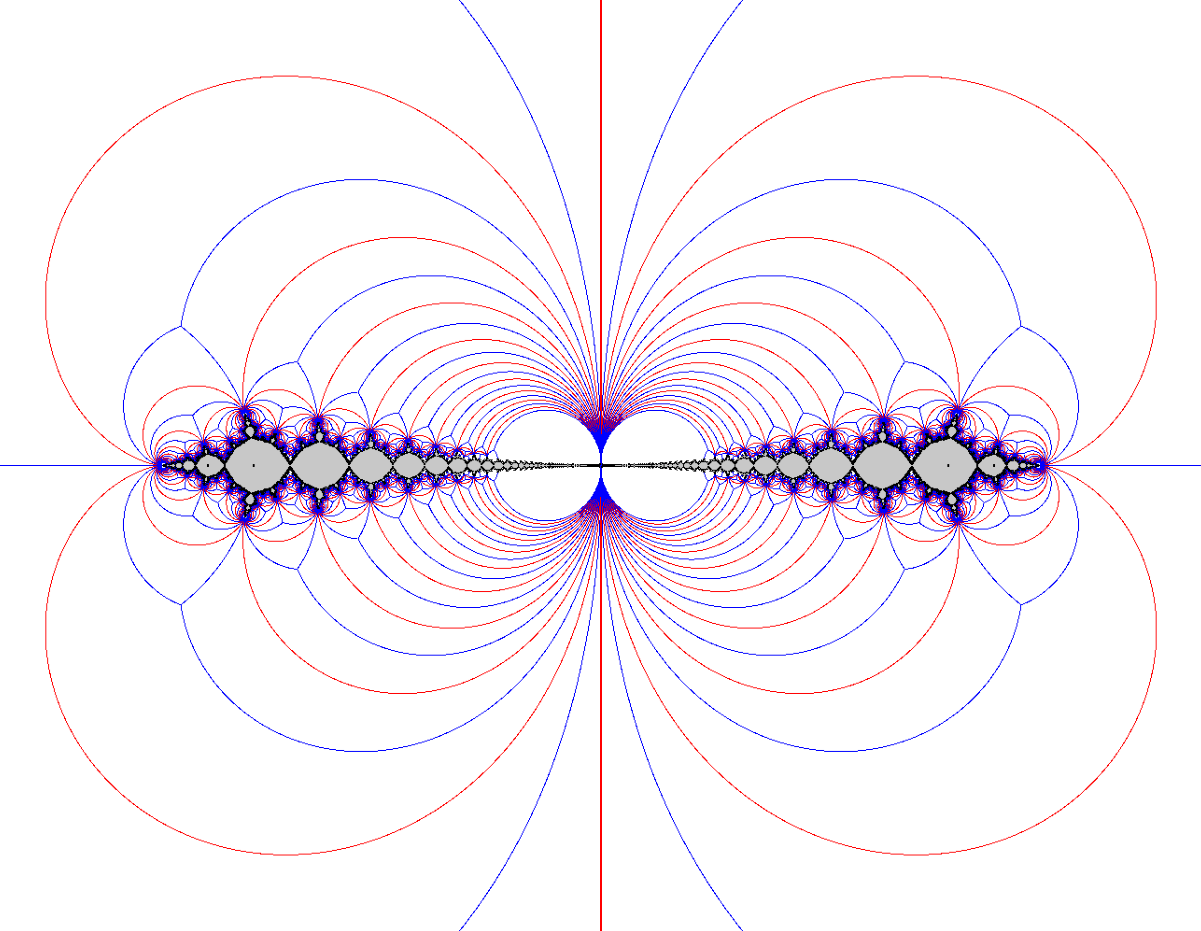}
\hspace{0.5cm}
\centering
\includegraphics[width=3.5cm]{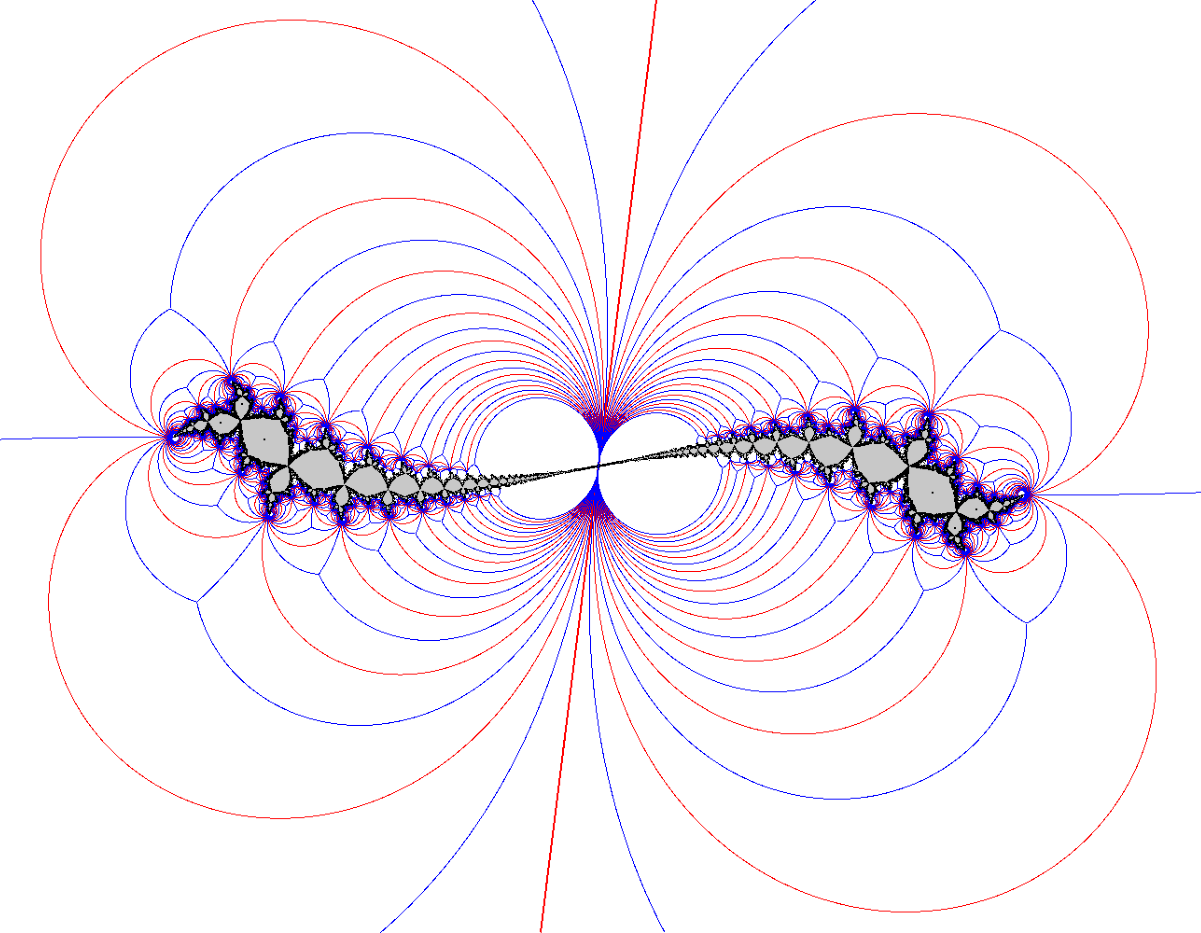}
\end{center}
\caption{\small From left to right, regular and limit sets of correspondences $\F_a$ plotted in the $z$-coordinate plane, for $a$ at the centres of 
the main hyperbolic component of $\mathcal{M}_\Gamma$, the period $2$ component and a period $3$ component.}\label{per123} 
\end{figure}

The correspondence $\F_a$ has certain associated {\it characteristic points}: a conformal conjugacy between 
$\F_a$ and $\F_{a'}$ necessarily sends each of these to the corresponding characteristic point of $\F_{a'}$. The first such
point is the persistent parabolic fixed point $P=P_a$ of $\F_a$, characterised by the fact that it is both a fixed point of $J$ and
a critical point of $Q$: in the $Z$-coordinate it is the point $Z=1$, and in the $z$-coordinate $P$ is $z=0$. The branch of $\F_a$ which fixes
$P$ has derivative $1$ at $P$. Further characteristic points are the other critical points $Z=-1$ and $Z=\infty$ of $Q$, and the other fixed 
point, $Z=a$ of $J_a$: we remark that $Z=-1$ can also be regarded as the critical point $c_a$ of the $(2:1)$ restriction of $\F_a$ to 
$\Lambda_{a,-}$. 
The cross-ratio of any four distinct characteristic points is preserved by every conformal conjugacy between $\F_a$ and $\F_{a'}$. 
Applying this to the four points $Z=1,-1,\infty$ and $a$ we have:

\begin{lemma}\label{unique}
If $\F_a$ is conformally conjugate to $\F_{a'}$ then $a=a'$. $\qed$
\end{lemma}
  
\subsection{$Per_1(1)$}\label{Per}
$Per_1(1)$ is Milnor's notation for the space of conformal conjugacy classes of quadratic rational maps having a parabolic fixed point 
with multiplier $1$. Normalising such a map to put the parabolic fixed point at $z=\infty$, and critical points at $\pm1$, $Per_1(1)$ 
is the family 
$$Per_1(1):= \{[P_A] | P_A(z)=z+1/z+A,\,\,\,\,\,\,\, A\in \C\}$$
of conformal conjugacy classes $[P_A]$. Here $P_A\sim P_{A'}$ if and only if $A'=\pm A$ (note that $P_A\sim P_{-A}$, by the involution 
$z \rightarrow -z$, which interchanges the critical points).
For every $A  \in\C$, the parabolic basin of attraction of infinity,  $\mathcal{A}_A(\infty)$, is completely invariant, and we can define the filled Julia set $K_A$ to be its complement, 
that is, 
$$K_A:= \widehat \C \setminus \mathcal{A}_A(\infty).$$
This is the parabolic counterpart of the definition of  filled Julia set for any polynomial $P$ on $\widehat \C$.
Note that the filled Julia set $K_A$ is well defined for every $A \neq 0$. The map $P_0(z)=z+1/z$ is the unique member of $Per_1(1)$ for which the multiplicity of the parabolic fixed point at $\infty$ is equal to $3$, which means that for this map there exist two attracting petals, and both are completely invariant Fatou components: the Julia set of $P_0$ is the imaginary axis, and the two attracting petals are the positive and the negative half planes $\{z\in \C | Re(z)>0\}$ and $\{z\in \C | Re(z)<0\}$ respectively. For consistency with \cite{L1}, we set $K_{0}=\{z\in \C | Re(z)<0\}$.

 \begin{figure}
 
\begin{center}
 \begin{tabular}{rcl}
 \begin{minipage}{3cm}
\centering
\includegraphics[width= 2.5cm]{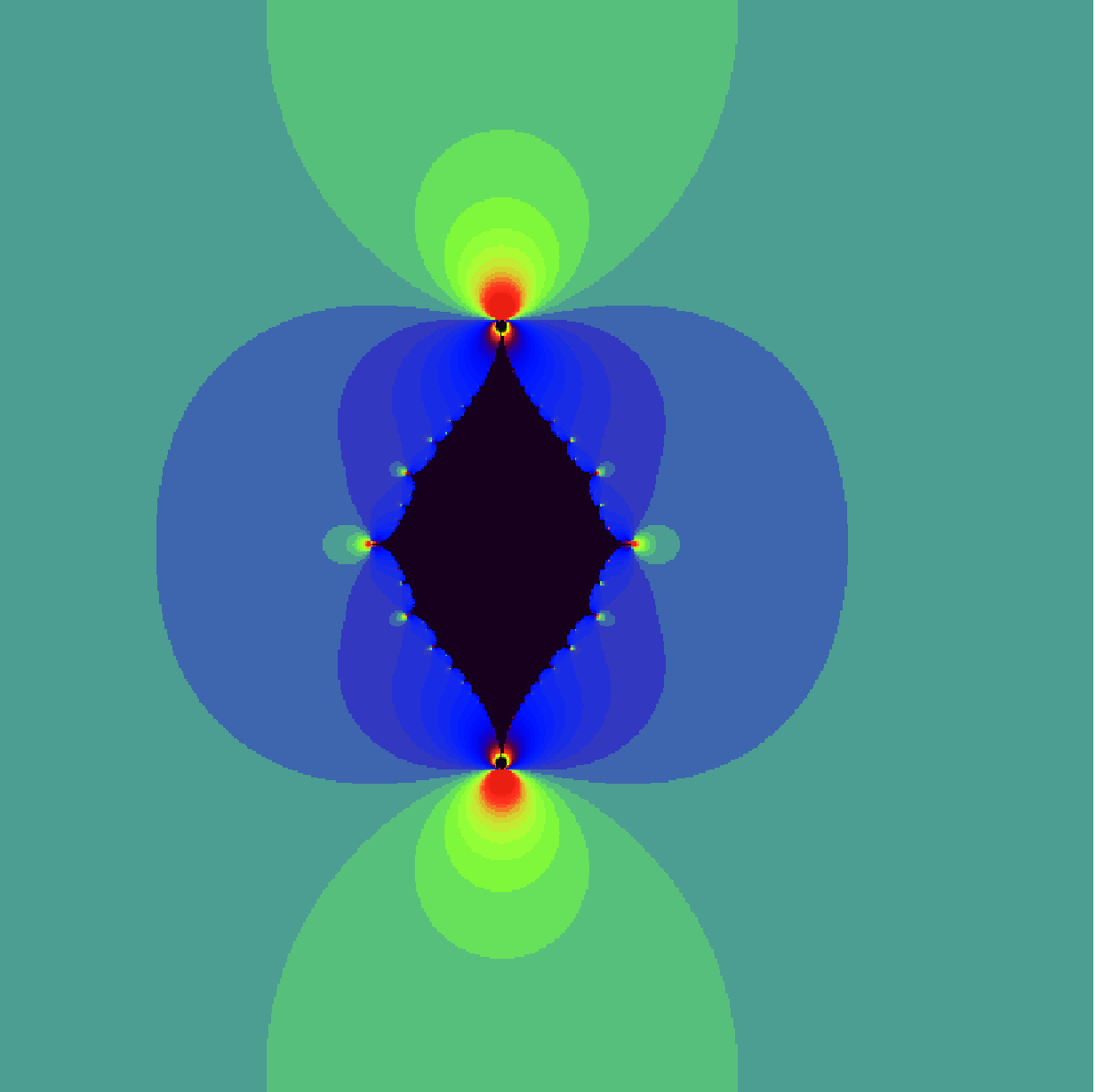}
 \end{minipage}
   \hspace{0.2cm}
  \begin{minipage}{3cm}
\centering
\includegraphics[width= 2.5cm]{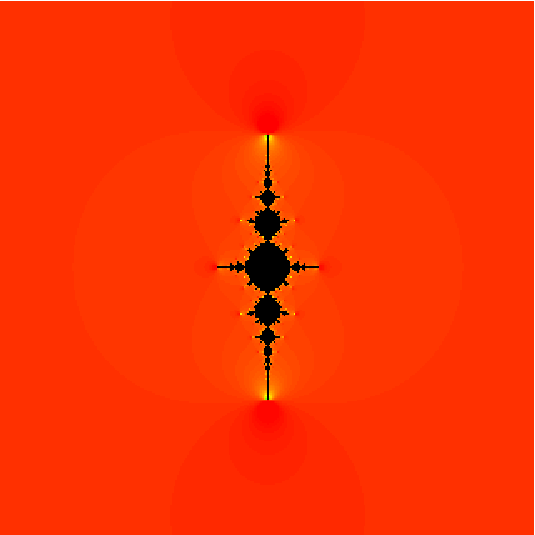}
 \end{minipage}
  \hspace{0.2cm}
 \begin{minipage}{3cm}
\centering
\includegraphics[width= 2.5cm]{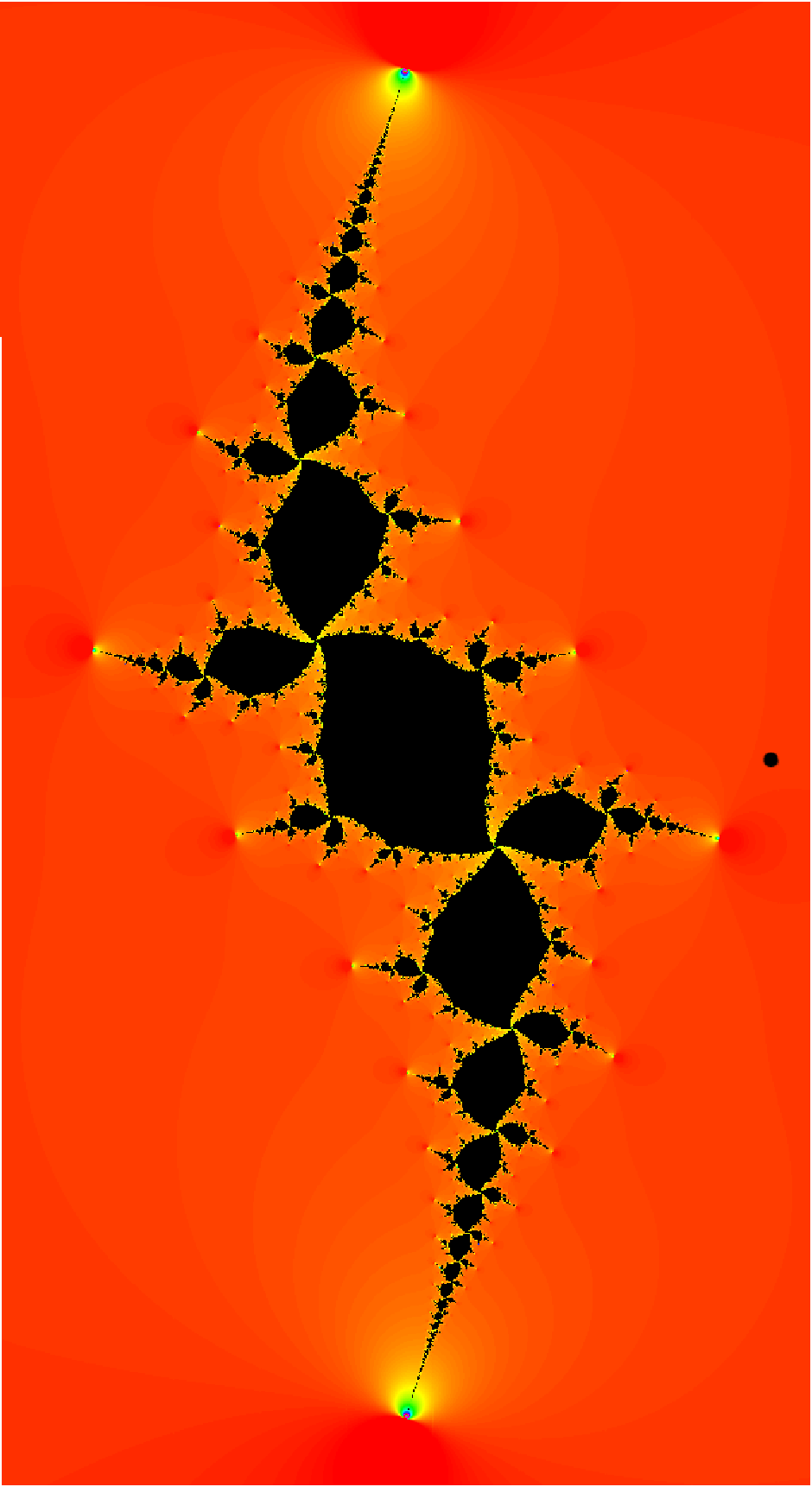}
 \end{minipage}

\end{tabular}
\end{center}
\caption{\small From left to right, the filled Julia set $K_A$ of $P_A$, for $B=1-A^2$ at the centres of 
the main hyperbolic component of $\mathcal{M}_1$, the period $2$ component and a period $3$ component.}\label{perplots}
\end{figure}

The parabolic Mandelbrot set $\mathcal{M}_1$ (Figure \ref{mg}, center) is the connectedness locus for $Per_1(1)$, which is usually parametrised by 
$$B=1-A^2,$$
the multiplier of the fixed point of $P_A$ other than $z=\infty$. In Figure \ref{perplots} we show the filled Julia set $K_A$ of $P_A$, for $B=1-A^2$ at the centres of 
the main hyperbolic component of $\mathcal{M}_1$, the period $2$ component and a period $3$ component.

For $A=0$, the map $P_0(z)=z+1/z$ is conformally conjugate on the Riemann sphere to the map
$$h(z)=\frac{z^2+1/3}{z^2/3+1}$$
via the conjugating function $\phi(z)=\frac{z+1}{z-1}$.
Using Fatou coordinates, it is easy to see that, for every $A$ such that $B=1-A^2 \in \mathcal{M}_1$, the dynamics of the map $P_A$ is conformally conjugate on $\mathcal{A}_A(\infty)$ to the dynamics of $P_0$ on $\mathcal{A}_0(\infty)=\H_l=\{x+iy\,|\,x>0\}$. The second author proved that this conjugacy still exists in the disconnected case, although restricted; more precisely, when $P_A$ has disconnected Julia set, there exists a conformal conjugacy between $P_0$ and $P_A$ defined between fundamental annuli about the respective filled Julia sets (see Proposition 4.2 in \cite{L1}). Roughly speaking, this means that the map $h$ encodes the dynamics of the $P_A$ on its basin of attraction of infinity. The Main Theorem will be obtained by gluing the map $h$ to the outside of the backward limit set $\Lambda_{a,-}$ of $\F_a$, using quasiconformal surgery (see Section \ref{reviewsurgery}).
We observe that $h|\S^1$ is topologically conjugate to the map $z\rightarrow z^2$, and hence via the Minkowski question mark function to the
union of the generators $\alpha,\beta$ of the modular group on $[-\infty,-1]$ and $[-1,0]$ respectively (once we have identified the ends $-\infty$ and $0$
of the negative real axis). So with hindsight the construction in \cite{BP} could have started with the family $Per_1(1)$ in place of quadratic polynomials. 

\section{The surgery}\label{surgery}
In this Section we develop a surgery construction which, for $a$ in a certain open subset $\mathring K$ of parameter space, will convert the correspondence $\F_a$ 
%with $a$ in an open subset $\mathring K$ 
%of a certain parameter lune $\mathcal{L}_{\theta}$
into a rational map in $Per_1(1)$, uniformly with respect to the parameter $a$. This construction is at the heart of 
our definition of the map $\chi:\m \to \mathcal{M}_1$ (see section \ref{reviewsurgery}), 
and of its extension to $\mathring K \setminus \m$ (see \ref{extension}). 
%
%The correspondence to be subjected to surgery will be the restriction of $\F_a$
%to a certain {\it dynamical lune} $V_a$, rather than its restrictions to rather more general domains as described in \cite{BL1}.
%As $V_a$ will play a major part in the following parts of this paper, we now introduce this set formally.\\

\subsection{The parameter space lune $\mathcal{L}_{\theta}$ and dynamical space lune $V_a$}\label{dynlune}

A {\it lune} is the name given in Euclidean geometry to the region trapped between two intersecting circular arcs. All our lunes will be open sets, in other words
a {\it lune} will not include its boundary arcs. The {\it vertex angle} of a lune is the angle between the arcs at their intersection points (the {\it vertices}).
It will be crucial to our construction that we restrict the correspondence $\F_a$ to a lune $V_a$ of vertex angle strictly less than $\pi$, such that the parabolic 
fixed point $P_a$ of $\F_a$ is one of the vertices of $V_a$, and $V_a$ contains $\Lambda_{-,a}\setminus\{P_a\}$. 
Moreover we will need the boundary $\partial V_a$ to move holomorphically with respect to $a$. 
In \cite{BL2}, as an application of a new Yoccoz-type inequality derived there for family $\F_a$, we proved that there exists an angle 
$\theta \in [\pi/3,\pi/2)$ such that $\m$ is contained in the closure of the lune (of vertex angle $2\theta$)
$$\L\f:=\{a: \big|arg\big(\frac{a-1}{7-a}\big)\big|<\theta\}.$$
More precisely, we proved in Theorem 3 of  \cite{BL2} that
$\m \subset \mathcal{L}_{\theta}\cup\{7\}$ and $ \m \cap \partial \mathcal{L}_{\theta} =\{7\}.$

We now define the lune $L_a$ in the $Z$-coordinate to be 
the open set bounded by the two arcs intersecting at $Z=1$ (the point $P_a$) and $Z=a$ which pass through $Z=1$ at 
angles $\pm \theta$ to the positive real axis. %(where $\theta$ is given by Theorem 3 of \cite{BL2} cited above). 
In the $z$-coordinate $L_a$ is a sector (as arcs of circles through $z=0$ and $z=\infty$ are straight lines); this sector varies with $a$ but its motion
is holomorphic with respect to $a$: indeed in the new coordinate defined by $z'=(a-1)z$ the lune $L_a$ is independent of $a$, so stationary.
In the Appendix of the current paper (Proposition \ref{dynamic-lune} and Corollary \ref{intersections}) we show that
for every $a \in  \mathcal{L}_{\theta} \cup \{7\}$, the forward image $\F_a(\overline{L_a})$ is contained in  $L_a\cup\{P_a\}$. It follows that the forward limit set 
$\Lambda_{+,a}$ is contained in $L_a\cup\{P_a\}$.
%Hence, for all $a$ in the pinched neighbourhood $\L\f$ of $\m$, there exists a dynamical lune moving holomorphically with the parameter and containing $\Lambda_{a,+}$ in its closure.

A computation (see Proposition 3.5 in \cite{BL1}) shows that when $a\ne 7$, in the coordinate $\zeta=Z-1$
the power series expansion of the branch of $\F_a$ which fixes $\zeta=0$ has the form:
$$\zeta \to \zeta + \frac{a-7}{3(a-1)}\zeta^2 + \ldots,$$
so the repelling direction is $arg(\frac{a-1}{a-7})$. For technical reasons which will become apparent shortly, we now choose 
$\hat\theta \in (\theta,\pi/2)$, and let $\hat L_a$ denote the lune which has vertices at $Z=1$ and $Z=a$, bounded by arcs which meet the real Z-axis at angles
$\pm \hat\theta$ (so  $L_a \subset \hat L_a$), and define
$$V_a:=J(\hat L_a).$$
%Let $\hat\theta \in (\theta,\pi/2)$, and let $\hat L_a$ be the dynamical lune in $Z$-coordinates bounded by the two arcs of circles
%intersecting at $Z=1$ and $Z=a$ and passing through $Z=1$ at angles
%$\pm \hat\theta$ to the positive real axis. \textcolor{red}{Note that, as $\hat \theta>\theta$, then $\mathcal{L}_{\theta} \subset \mathcal{L}_{\hat\theta}$, and since for all $a \in \mathcal{L}_{\hat\theta}$ we have $\Lambda_{-,a}\subset J(\hat L_a)\cup \{P_a\}$ (see Corollary \ref{intersections}), this is still the case for $a\in \mathcal{L}_{\theta}$.} 
%We define: 
%$$V_a:=J(\hat L_a).$$
%We remark that in coordinate-free terms $V_a$ is the lune bounded by arcs (of round circles) which intersect at the fixed 
%points of $J_a$, and which at the parabolic point $P_a$ make angles of $\pm \hat\theta$ with the attracting direction. 
For all $a\in \L\f$ the repelling parabolic direction is compactly contained in $V_a$ since $\hat\theta>\theta$ 
(this will be important for the proof of Proposition \ref{divarcs}). By construction
$V_a$ moves holomorphically with $a$. Moreover $\F_a^{-1}(\overline{V_a})\subset V_a\cup\{P_a\}$ and hence $\Lambda_{a,-} \subset V_a\cup \{P_a\}$.
\begin{lemma}\label{cr}
  $\m$ is the set of parameters $a \in \L\f\cup \{7\}$ for which the critical point $c_a$ is in $\Lambda_-(\F_a)$. In particular, $\m$ is closed.
\end{lemma}
\begin{proof}
From the definition of $\m$, for each $a \in \m$ the set $\Lambda_-(\F_a)$ is connected.
 If the critical point $c_a$ is not in $\Lambda_-(\F_a)$, then  
 some $\F_a^{-n}(V_a)$ is a pair of topological discs. Hence $\Lambda_-(\F_a)$ is disconnected (indeed a Cantor set).
 Conversely, if $c_a\in \Lambda_-(\F_a)$ then each $\F_a^{-n}(V_a)$ is connected and hence so is $\Lambda_-(\F_a)$.
\end{proof}

\subsection{The parameter space $\mathring K\subset \L\f$}\label{parlune}
Define $V'_a:=\F_a^{-1}(V_a)$. 
%We define the \textit{fundamental croissant} to be
 %$$\mathcal{C}_a:=V_a \setminus V_a'.$$
%Note that the fundamental croissant contains one of its boundary curves but not the other \textcolor{red}{(I thought $V_a$ was open?)} \textcolor{blue}{I didn't remember how we defined it. does it make a difference?}. 
For our surgery construction we will need the critical point $c_a$ of $\F\r$ to be in $V'_a$, so we first restrict our parameter space 
to the subset $\L\f'\subset \L\f$ for which the critical value $v_a$ of $\F\r$ is in $V_a$.

%the set $V_a'$ to move holomorphically. For this to happen, it must retain the same homeomorphism type, so if $V'_a$ is connected at the start of the motion 
%it must remain connected
%throughout. We therefore restrict our attention
%to the subset $\L\f'\subset \L\f$ for which the critical value $v_a$ of $\F\r$ is in $V_a$ (equivalently, the critical point $c_a$ is in $V'_a$).

\begin{prop}\label{truncatelune}
$\L\f'$ is a simply connected open subset of $\L\f$, with the properties that
$\m \subset \L\f' \cup \{7\}$ and 
$\partial \L\f'\cap \m =\{7\}$.
\end{prop}
\begin{proof}
Write $\F_a$ as $J_a \circ Cov_0^Q$ (see Section \ref{corr}).
In $Z$ coordinates, the critical point of $Cov_0^Q$ which belongs to $V_a$ is $Z=-1$, and the corresponding 
critical branch of $Cov_0^Q$ sends $-1 \to 2$, so the critical value of $\F_a$ in $V_a$ is $J(2)$. 

For the remainder of the proof it will be convenient to change to the coordinate $$z'=(a-1)z=(a-1)(Z-1)/(a-Z),$$ where 
$\hat{L}_a$ and $V_a$ are independent of $a$.
In this new coordinate $\hat{L}_a$ is the subset of the right-hand half-place bounded by the straight lines through
at the origin at angles $\pm\hat\theta$ to the positive real axis,  
the involution is $J(z)=-z$, and the critical value is $v_a=-(a-1)/(a-2)$.

By definition ${\mathcal L}'_{\theta}$ is the subset of ${\mathcal L}_\theta$ for which $v_a\in V_a$, that
is to say for which $J(v_a)\in \hat{L}_a$, equivalently $(a-1)/(a-2)\in \hat{L}_a$: the $a\in {\mathcal L}_\theta$ for which $arg((a-1)/(a-2))\in (-\hat\theta,\hat\theta)$. Thus ${\mathcal L}'_{\theta}$ is obtained from ${\mathcal L}_\theta$
by excising its intersection with the pair of closed round discs which have boundary circles passing through both $a=1$ and $a=2$ at angles $\pm\hat\theta$ to the real axis. The `truncated lune' ${\mathcal L}'_{\theta}$, union the point $a=7$, certainly contains $\mathcal{M}_{\Gamma}$ since when $a\in \mathcal{M}_{\Gamma}$ we know that $v_a \in \Lambda_-\subset V_a$, and the boundary of ${\mathcal L}'_{\theta}$ can only meet that of $\mathcal{M}_\Gamma$ at $a=7$ since ${\mathcal L}'_{\theta}$ is open and $\mathcal{M}_{\Gamma}$ is closed.
\end{proof}

\begin{lemma}\label{set}
For $a \in \mathcal{L}'_{\theta} \cup \{7\}$ the restriction $\F_a|_{V'_a}: V_a' \rightarrow V_a$
 is a branched covering map of degree $2$, which is holomorphic in both variables and has a persistent parabolic fixed point at $P_a \in \partial V_a$. 
\end{lemma} 

\begin{proof}
For every $a \in \mathcal{L}'_{\theta} \cup \{7\}$ we have $V_a\subset \Delta_J$, 
and by Proposition 3.4 in \cite{BL1} (see also the discussion in Section \ref{corr} and Figure \ref{fd3})  the restriction of $\F_a$ to 
domain $\F_a^{-1}(\Delta_J)$ and codomain $\Delta_J$
is a degree $2$ holomorphic map, so that $\F_a|_{V_a'}: V_a' \rightarrow V_a$ is also holomorphic and of degree $2$ in the dynamical
variable. The fact that it is also holomorphic in the parameter $a$ follows from its explicit representation as a polynomial relation.
\end{proof}

\begin{lemma}\label{hollunes}
 For $a \in \mathcal{L}'_{\theta}$, the boundary $\partial V_a$ of the lune $V_a$ and its inverse image $\partial V'_{a}:= \F_a^{-1}(\partial V_a)$ move holomorphically.
\end{lemma}
\begin{proof}
Choose a base point $a_0 \in \mathcal{L}'_{\theta}$. To say that $\partial V_a$ moves holomorphically is to say that there exists a  homeomorphism 
$\psi_a: \partial V_{a_0} \rightarrow \partial V_a$ depending
holomorphically on the parameter $a$. But in the coordinate  $z'=(a-1)z$ we may take $\psi_a$ to be the identity
as in this coordinate $V_a$ is independent of $a$: 
%(see the paragraph following the statement of Proposition \ref{dynamic-lune} in the Appendix) 
hence $V_a$ moves holomorphically in any coordinate varying holomorphically with $a$.
The correspondence $\F_a: V'_{a }\rightarrow V_a$ is holomorphic in both variables, and for $a \in \mathcal{L}'_{\theta}$ the critical point belogs to $V_a'$, so by lifting the holomorphic motion
of $\partial V_a$ we obtain a holomorphic motion of $\partial V'_{a}$. 
\end{proof}

\subsubsection{The doubly truncated lune $\mathring K\subset \L\f'$}\label{dtl}
Let $N$ be a small round closed disc neighbourhood of $a=7$ in the plane of the parameter $a$, and let $$K= \overline \L\f' \setminus {\mathring N} \mbox{ (see Figure \ref{k})}.$$ 
%\textcolor{red}{Need to check the lettering in all figures as we have changed notation.}
Then $K$ is a compact set and its interior $\mathring K$ is a topological disc. 
The \textit{doubly truncated lune} $\mathring K$ is the set of parameter values for which we will perform the surgery and which will ultimately be the domain 
on which shall define the map $\chi$. For the surgery construction, the existence of forward invariant 
arcs $\gamma_{a,i},\,i\in\{1,2\}$ for $\F_a$, with the properties listed in the following Proposition, 
will be essential
%obtain by surgery a continuous map 
(see Section \ref{s1}). 

 \begin{figure}
\centering
\psfrag{a}{\small $a=7$}
\psfrag{b}{\small $a=1$}
\psfrag{c}{\small $a=2$}
\psfrag{k}{\small $K$}
\psfrag{2}{\small $2\theta$}
\psfrag{3}{\small $2\hat\theta$}
	\includegraphics[width= 8cm]{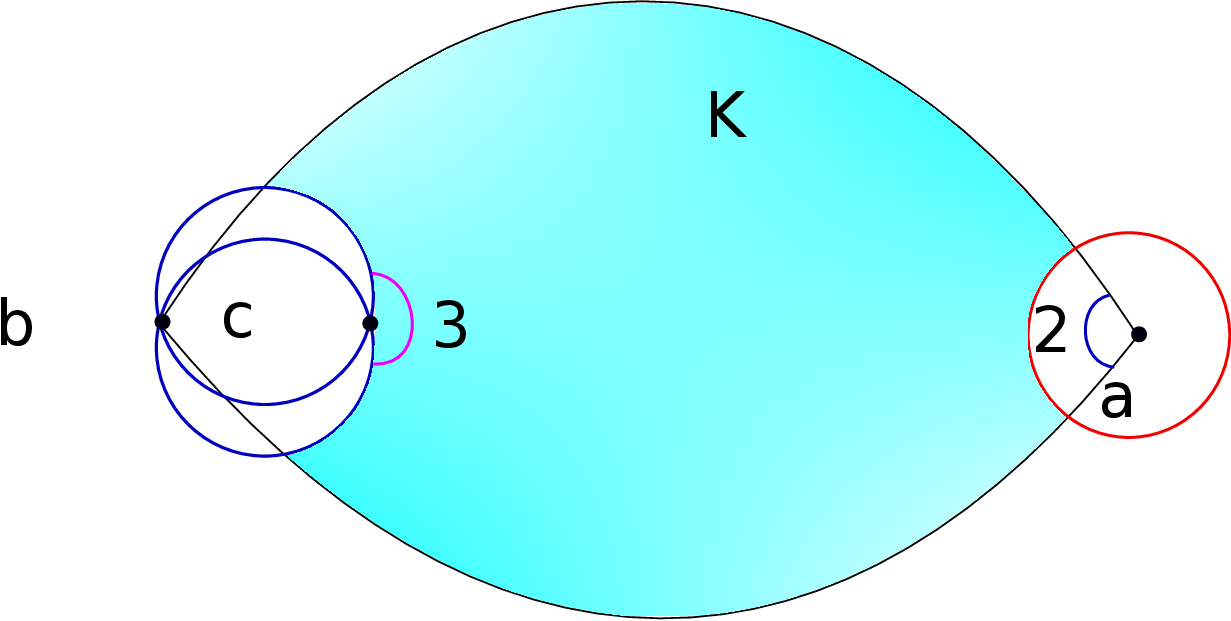}
\caption{\small The doubly truncated lune $K$ in parameter space. Excising the left-hand end of the lune $\L\f$ gives a truncated lune $\L\f'$ with the property that $\forall a \in \L\f'$ the critical
value $v_a$ of $\F_a$ lies in the dynamical lune $V_a$ (Proposition \ref{truncatelune}). Removing the intersection of $\L\f'$ with an arbitrarily small disc $N$ 
at the right-hand end gives $K$ with the property that there exists $\gamma_a$ in $V_a$ moving holomorphically with $a \in \mathring K$ (Proposition \ref{divarcs}).}
\label{k}
\end{figure}
\begin{prop}\label{divarcs}
There exists a family of forward invariant arcs $\gamma_{a,i},\,i\in\{1,2\}$ for $\F_a$, parametrised by $a\in{ \mathring K}$, such that each
$\gamma_a:= \gamma_{a,1}\cup \gamma_{a,2}$ is connected,  moves holomorphically with $a\in \mathring K$, and meets $\partial V_a$ transversally.
\end{prop}
\begin{proof}
For the proof of this proposition we will work in the $z'=(a-1)z$ coordinate, in which the lune $V_a$ is always the subset of the left-hand half-plane bounded 
by the straight lines through the origin at angles $\pm \hat \theta$ to the negative real axis, and the parabolic fixed point $P_a$ is the origin. So we 
will write $V$ for $V_a$, and $P$ for  $P_a$.
%Let us recall that the parabolic repelling direction is compactly contained in $V$ for all $a\in \mathcal{L}_{\theta}$ (see Section \ref{dynlune}).

The idea of the proof is the following. Using pre-Fatou coordinates (denoted by $\psi_a$) and compactness, we will construct a subset $\mathring K$ of $\mathcal{L}_{\theta}'$ and repelling petals $\Xi_a$ of $\F_a$ at $P$ moving holomorphically with the parameter sufficiently large enough that
 there exist a pair of points $W^{i}\in \partial V,\,i\in\{1,2\}$ belonging to $\Xi_a$ for all $a \in \mathring K$. Then, using Fatou coordinates (denoted by $\phi_a$) sending the repelling direction to the real axis and depending holomorphically on the parameter $a$, we can define $\gamma_{a,i}$ to be the pre-image under Fatou coordinates
of the horizontal straight line between the image of $W^i$ under Fatou coordinates and $-\infty$. By construction, $\gamma_{a,i},\,i\in\{1,2\}$ are forward invariant under $\F_a$, the path $\gamma_a:= \gamma_{a,1}\cup \gamma_{a,2}$ is connected and, since $\phi_a(W^1)$ and $\phi_a(W^2)$ move holomorphically with $a$, the whole of $\gamma_a$ also moves holomorphically with $a$.\\

 \begin{figure}
\centering
\psfrag{a}{\small $\alpha$}
\psfrag{t}{\small $\theta$}
\psfrag{g}{\small $\beta$}
\psfrag{V}{\small $V$}
\psfrag{w}{\small $w_{a,+}$}
\psfrag{c}{\small $w_{a,-}$}
	\includegraphics[width= 8cm]{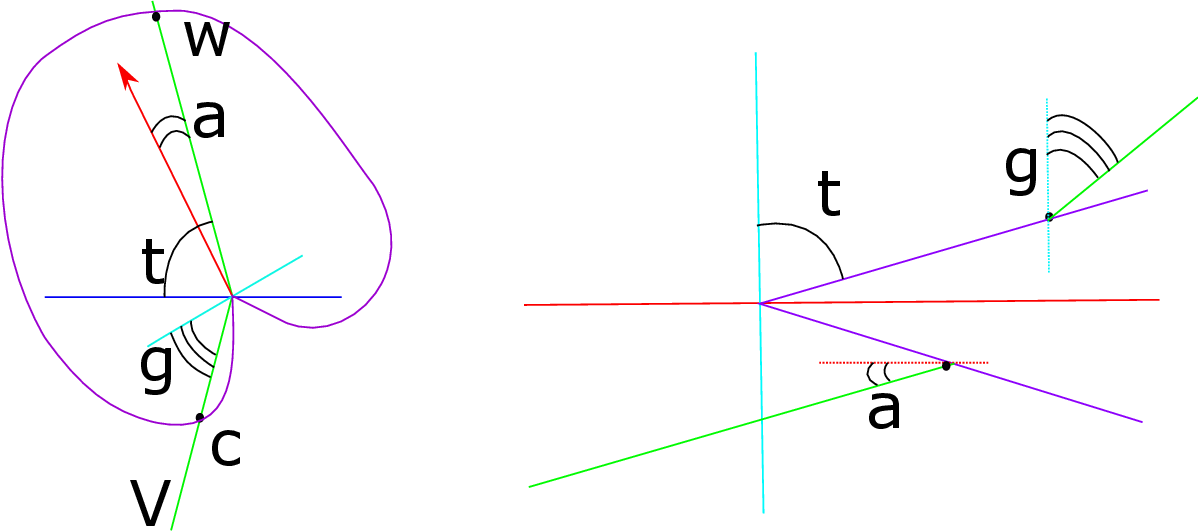}
\caption{\small For $a \in \L\f$, the repelling direction is contained in $V$, and so
 $\psi_a([w_{a,i},P)),\,i\in\{1,2\}$ (represented as the green half lines on the right) are not parallel to the real axis}\label{tr}
\end{figure}

We proceed to the details. 
%A computation (see Proposition 3.5 in \cite{BL1}) shows that, 
As remarked in Section \ref{dynlune}, in $\zeta=Z-1$ coordinates, when $a\ne 7$, 
the power series expansion of the branch of $\F_a$ which fixes $\zeta=0$ has the form:
$$\zeta \to \zeta + \frac{a-7}{3(a-1)}\zeta^2 + \ldots.$$
Since the change of coordinates from $\zeta=Z-1$ to $z'=(a-1)(Z-1)/(a-Z)$ has derivative equal to $+1$ at $\zeta=0$, 
the power series expansion in $z'$ of the branch of $\F_a$ fixing $P_a$ has the same coefficients in its terms of degree $\le 2$ as those in
the $\zeta$ coordinate, so this power series has the form:
$$z' \to z' + \frac{a-7}{3(a-1)}z'^2 + \ldots$$
Let $\psi_a(\zeta)=-1/(b(a)\zeta)$, where $b(a)= \frac{a-7}{3(a-1)}$, be pre-Fatou coordinates, and define the map
$$F_a= \psi_a \circ \F_a \circ \psi_a^{-1}.$$ Then a computation shows that
$$F_a(w)=w+1+d(a)/w+O(1/w^2)$$
for a constant $d(a)$ (see Section 2.1.2 in \cite{Sh}). 
Let $c\geq 2$ be sufficiently large that, setting $\tau:= \arctan(1/c)$, we have $(\pi/2-\hat\theta)/2<\tau< \pi/2-\hat\theta$ 
(where $2\hat \theta$ is the angle of $V$ at $P$), and let $R=R(a)$ be sufficiently large that whenever
f$|w|>R/(2c)$ we have $|F_a(w)-(w+1)|<1/(2c)$. Now, setting and so
defining 
$$X_a:= \{w \in \C\ :\  Re(w)<  c|Im(w)|-R\},$$ $F_a$ is defined and injective on $X_a$.
Recalling that $K=\overline \L\f'\setminus {\mathring N}$ for a small neighbourhood of $a=7$, let
$\widehat R= max_{a \in K}R(a)$  and define 
$$X:= \{w \in \C \ :\ Re(w)<  c|Im(w)|-\widehat{R}\}.$$ Then the set  $\Xi_a:= \psi_a^{-1}(X)$ is a repelling petal for $\F_a$ at $P$,
that moves holomorphically with $a\in \mathring K$, and that contains a neighbourhood of $P$ in $\partial V$ for all $a \in \mathring K$. 
Hence, for every $a\neq 7$, there exist $w_{a,i} \in \partial V,\,i\in\{1,2\}$ such that $\psi_a([w_{a,i},P))\in X,\,i\in\{1,2\}$, and since 
the repelling direction is contained in $V$,
for $w_{a,i}$ sufficiently close to $P$ the half-lines $\psi_a([w_{a,i},P))$ are nowhere parallel to the real axis (see Figure \ref{tr}). 
By the compactness of $K$, there exists a pair of points $W^i \in \partial V,\,i\in\{1,2\}$, one on each side of $P$, which lie in $\Xi_a$ for all $a \in K$ and 
have the property that the half-lines $\psi_a([W^i,P)),\,i\in\{1,2\}$ are nowhere parallel to the real axis.
 
Now let $\phi_a : \Xi_a \rightarrow \hat X:=\phi_a(\Xi_a) \subset \C$ be Fatou coordinates depending holomorphically on the parameter $a$ and sending the repelling direction to the real axis (this is, $\phi_a= \Phi_a\circ \psi_a$, where $\Phi(w)=w-d(a)+c+O(1/w^2)$ and $c$ is a constant, see Proposition 2.2.1 in \cite{Sh}). Thus in particular $W_{a,i}:=\phi_{a}(W^{i}),\,i\in\{1,2\}$ move holomorphically. 
Let $\hat \gamma_{a,i},\,i\in\{1,2\}$ denote the horizontal straight lines between $W_{a,i}$ and $-\infty$. Then $\hat \gamma_{a,i},\,i\in\{1,2\}$ are contained in $\hat X$, are forward invariant, are transversal to $\phi_a(\partial V)$, and move holomorphically with $a \in \mathring K$. Setting $\gamma_{a,i}:= \phi_a^{-1}(\hat \gamma_{a,i}),\,i\in\{1,2\}$ and $\gamma_a:= \gamma_{a,1}\cup \gamma_{a,2}$, 
the family of $\gamma_a$ for $a\in \mathring K$  has the desired properties.
\end{proof}

\subsection{Holomorphic motions}\label{holom}
%Recall that the fundamental croissant is
% $$\mathcal{C}_a:=V_a \setminus V_a'.$$
The curves $\gamma_{a,i},\,i=\{1,2\}$ divide $V'_a$ into three topological discs. Two of these
%which we denote $\mathcal S_{a,i},\,i=\{1,2\}$, 
are in the 
repelling petal $\Xi_a$ for $\F_a$ at $P_a$. The third contains the critical point $c_a$ and the limit set $\Lambda_{a,-}$ of $\F_a$, and it will play a central role in our surgery: we call this set
the \textit{central set} (Figure \ref{hm}), and we denote it by $O_a.$
Later, we will define a corresponding set $O_h$ for the Blaschke product $h$; our surgery construction will glue $\F_a$ on $O_a$ to $h$ 
on $\widehat \C \setminus O_h$.
%Let $B_{a,i}$ for $i=1,2$ denote the pair of open subsets of $\mathcal{C}_a$ bounded by $\gamma_{a,i}$ and $\partial V$ and not 
%containing the critical point (see Figure \ref{cross}).
%Then $\mathcal{S}_{a,i}=\bigcup_{n>0}\F^{-n}(B_{a,i})$. 
%In particular, $\mathcal{S}_{a,i}$ does not contain the critical point $c_a$, and $\mathcal S_{a,\pm}\cap \Lambda_{a,-}=\emptyset$.
We define the fundamental pinched annulus to be
$$\P\A_a:= V_a \setminus O_a.$$
Then $\P\A_a$ does not contain the critical point, nor does it intersect the backward limit set.
%\begin{figure}
%\centering
%\psfrag{b}{\small $B_{a,1}$}
%\psfrag{c}{\small $B_{a,2}$}
%\psfrag{s}{\small $\mathcal{S}_{a,1}$}
%\psfrag{t}{\small $\mathcal{S}_{a,2}$}
%\psfrag{W}{\small $W^1$}
%\psfrag{Y}{\small $W^2$}
%\psfrag{f}{\small $\F_{a}$}
%\psfrag{g}{\small $\gamma_{a,1}$}
%\psfrag{d}{\small $\gamma_{a,2}$}
%\includegraphics[width= 8cm]{croissant3.eps}
%\caption{\small On the left are $\F_a: V_a' \rightarrow V_a$ and the dividing arcs $\gamma_{a,i},\,i=1,2$. On the right, coloured in green, violet and blue is the fundamental pinched annulus $\P\mathcal{A}_{a}$. In particular, in green we have $B_{a,i}$ for $i=1,2$;  in green and blue we have the fundamental croissant $\mathcal{C}_{a}= V_a\setminus V_a'$, and in violet 
%we have the sets $\mathcal{S}_{a,i}$, for $i=1,2$.
%\textcolor{red}{The interior white region is the set we call $O_a$, the blue region is $Q_a$, and $O_a\cup Q_a$ is $U_a$.}
%}
%\label{cross}
%\end{figure}
Recall that the set $\K\subset \L\f'$ is a topological disc, and that by Lemma \ref{hollunes} 
%both $V_a$ and $V_a'$ move holomorphically with $a$, 
for $a \in \mathcal{L}_{\theta}'$
%Therefore, by the $\lambda$-Lemma, 
both $\partial V_a$ and $\partial V_a'$ move holomorphically. Since their intersection is precisely the parabolic fixed point, 
their union also moves holomorphically.
By Proposition \ref{divarcs}, the curve $\gamma_a:= \gamma_{a,1}\cup \{P_a\}\cup \gamma_{a,2}$ moves holomorphically for $a\in \mathring K\subset \mathcal{L}_{\theta}'$.
Hence, choosing a base point $a_0$ in the set  $\mathring \m \setminus N \subset \K$, we can define the holomorphic motion 
 $$\tau: \K\times (\partial(V_{a_0}) \cup \partial(V_{a_0}')\cup \gamma_{a_0}) \rightarrow \C$$
 holomorphic in $a$ (fixing $z$) and quasiconformal in $z$ (fixing $a$), such that
 $$\tau_{a}(\partial(V_{a_0}) \cup \partial(V'_{a_0})\cup \gamma_{a_0})=\partial(V_{a}) \cup \partial(V'_{a})\cup \gamma_{a}.$$

 Since $\K$ is conformally homeomorphic to a disc, by Slodkowski's  Theorem 
 the motion $\tau$ 
 extends to a holomorphic 
 motion of the complex plane:
  $$\overline\tau: \K \times \C \rightarrow \C,$$
     which we can then restrict to a holomorphic motion of the exterior of the central set $O_a$:
    $$\widetilde\tau: \mathring K  \times \widehat \C \setminus O_{a_0}\rightarrow\widehat \C,$$
and by construction $\widetilde\tau_a(\P\A_{a_0})=\P\A_a$.
 %where for each $a\in \K$ the qc homeomorphism $\widetilde\tau_a$ has image $\widehat \C \setminus O_{a_0}$.
  % Let us call
   
 %  such that
 % $$\tau_{a}(\P\mathcal{A}_{a_0})=\P\mathcal{A}_{a}.$$
%Note that we can restrict the holomorphic motion to the fundamental croissant:  
 %    $$\tau: \K \times \mathcal{C}_{a_0} \rightarrow \C.$$

\begin{figure}
\centering
\psfrag{a}{\small $\mathcal{C}_{a_0}$}
\psfrag{b}{\small $\mathcal{C}_a$}
\psfrag{c}{\small $\mathcal{C}_{a_0}$}
\psfrag{d}{\small $\mathcal{C}_a$}

\psfrag{a}{\small $\mathcal{C}_{a_0}$}
\psfrag{g}{\small $\F_a$}
\psfrag{f}{\small $\F_{a_0}$}
\psfrag{t}{\small $\tau_a$}
	\includegraphics[width= 8cm]{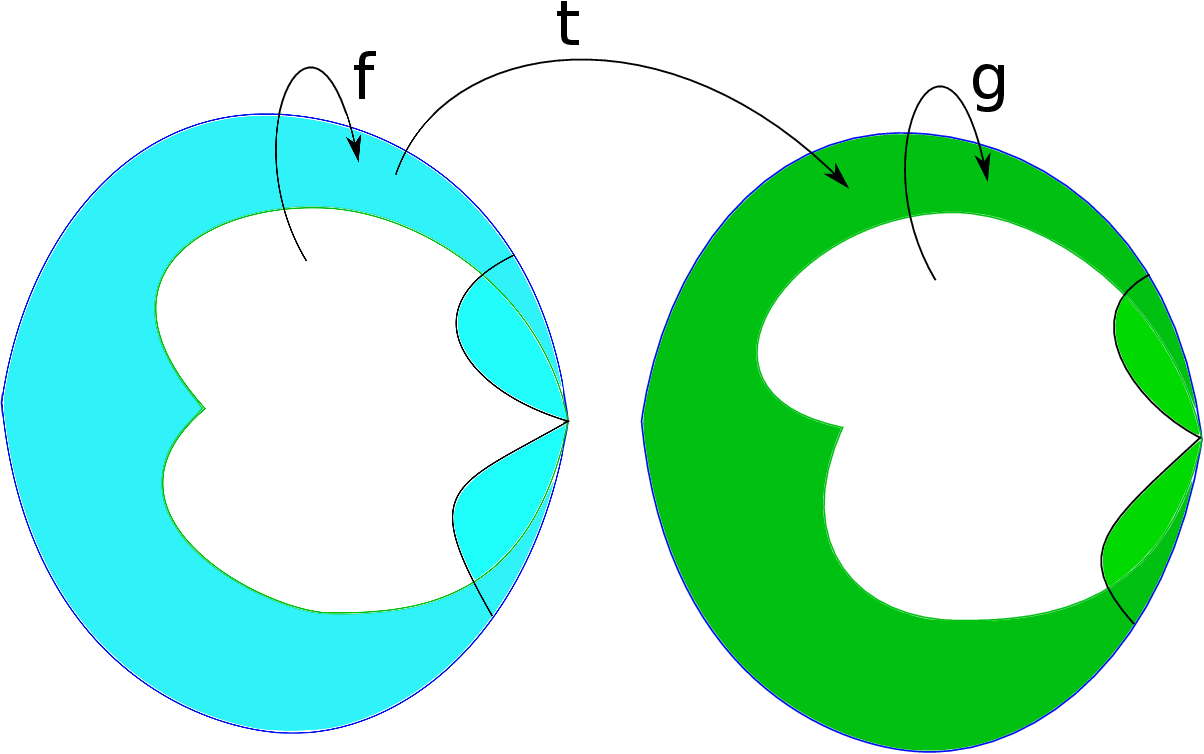}
\caption{\small Holomorphic motion $\tau_a$ between the fundamental pinched annulus $\P\mathcal{A}_{a_0}$ for $\F_{a_0}$ (in blue) and the fundamental pinched annulus $\P\mathcal{A}_a$ for $\F_a$ (in green).
The interior white region we call the {\it central set} $O_a$.}
\label{hm}
\end{figure}

\subsection{Surgery construction}\label{reviewsurgery}
We showed in \cite{BL1} that for every $a$ in the Klein combination locus $\mathcal{K}$, the correspondence
$\F_a$ can be modified by surgery to become a parabolic-like map, with the consequence that $\F_a$ is hybrid equivalent
to a member of $Per_1(1)$ for every $a\in \mathcal{K}$, a unique such member if $\Lambda_-(\F_a)$ is connected.

We now perform the surgery in a different fashion, going directly from the correspondences $\F_a$ to rational maps in $Per_1(1)$, 
by adapting the methods of \cite{L1}.  
The map 
$$h(z)=\frac{z^2+1/3}{1+z^2/3}$$ is an external map for $P_A(z)=z+1/z+A$ for all $A\in \C$ (\cite{L1}), and the idea of the surgery is to glue this external map
acting on $\widehat \C \setminus \overline \D$ onto 
$\widehat \C \setminus \Lambda_{a,-}$ by quasiconformal surgery. In order to do this we first fix $a_0 \in \mathring \m$, we uniformize the complement of $\Lambda_{a_0}$ to the complement 
of the closed unit disc and we conjugate $\F_{a_0}$ with the uniformization, constructing an `external map' for $\F_{a_0}$. We use this external map to construct a quasiconformal map $f$ from
$\widehat \C \setminus O_{a_0}$ to the corresponding set for $h$, which we use to convert $\F_{a_0}$ into a member of $Per_1(1)$. %We then move this construction by composing with holomorphic motion. 
Finally, composing with the holomorphic motion of $\widehat \C \setminus O_{a}$ moves this to a surgery on $\F_a$.
%In order to obtain this quasiconformal map $f$, we first construct an external map $g$ for $\F_{a_0}$, dividing arcs $\gamma_g$ for it, and a quasiconformal map between 
 %$\gamma_g$ and dividing arcs $\gamma_h$ for $h$, using Fatou coordinates. % which we use to obtain a quasiconformal maps between $O_g$ and $O_h$, and hence $f$.
\begin{figure}
\centering
\psfrag{g}{\small $g:=\alpha\circ\F_{a_0}\circ\alpha^{-1}$}
\psfrag{b}{\small $B_{a,1}$}
\psfrag{c}{\small $B_{a,2}$}
\psfrag{s}{\small $\mathcal{S}_{a,1}$}
\psfrag{R}{\small $\alpha$}
\psfrag{t}{\small $\mathcal{S}_{a,2}$}
\psfrag{j}{\small $\gamma_{g,1}=\alpha(\gamma_{a_0,1})$}
\psfrag{v}{\small $\gamma_{g,2}$}
\psfrag{h}{\small $h$}
\psfrag{z}{\small $\gamma_{h,1}=\phi_h^{-1}\circ \phi_g(\gamma_{g,1})$}
\psfrag{n}{\small $\gamma_{h,2}=\phi_h^{-1}\circ \phi_g(\gamma_{g,2})$}
\psfrag{p}{\small $\psi_1$}
\psfrag{k}{\small $\psi_2$}
\psfrag{f}{\small $\F_{a_0}$}
\psfrag{x}{\small $\gamma_{a_0,1}$}
\psfrag{d}{\small $\gamma_{a_0,2}$}
	\includegraphics[width= 7.5cm]{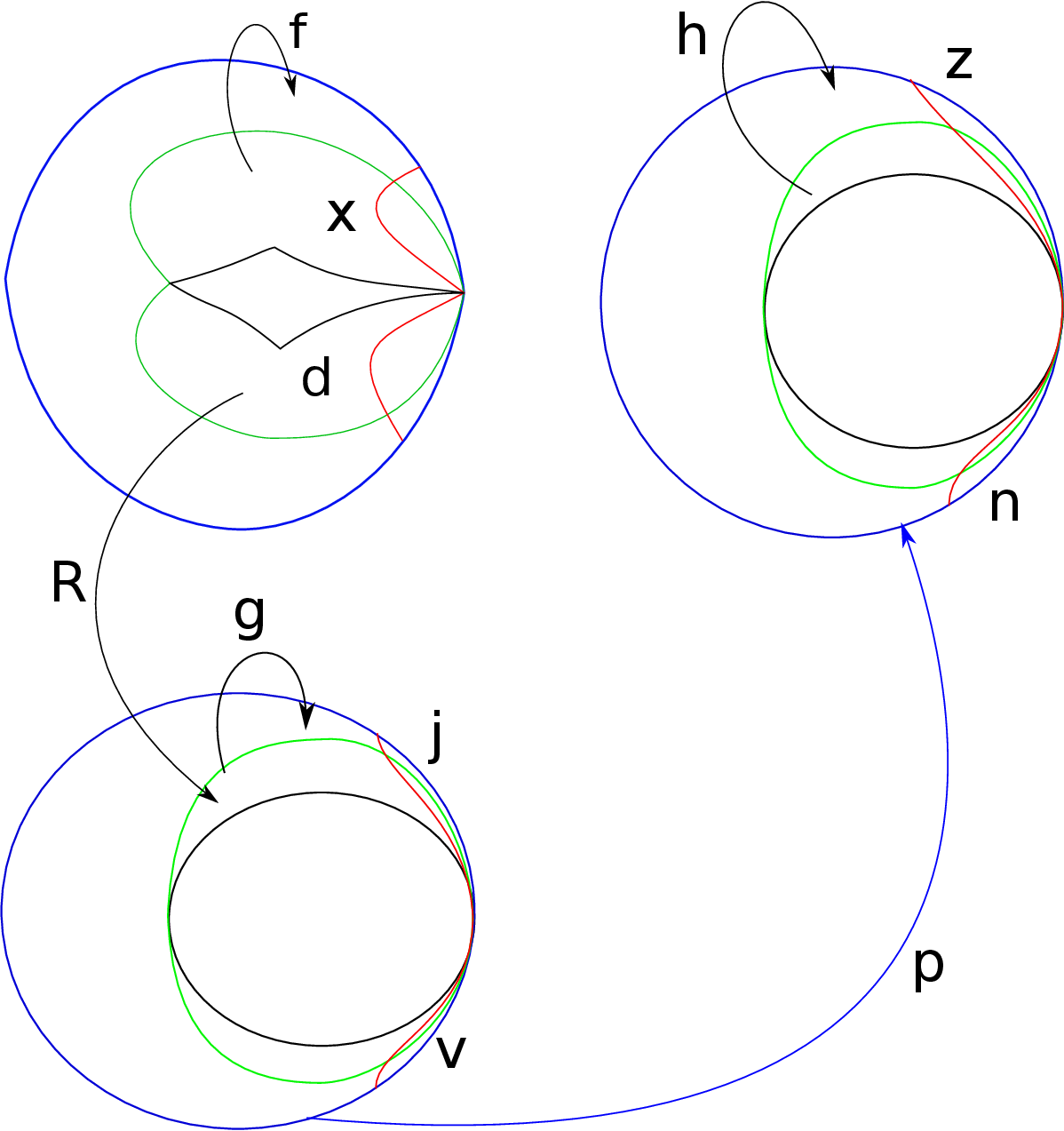}
\caption{\small Upper left, we have $\F_a: V_a' \rightarrow V_a$ and the dividing arcs $\gamma_{a,i},\,i=\{1,2\}$. Below, we have $g: V_g'\rightarrow V_g$ and its dividing arcs  $\gamma_{g,i}=\alpha(\gamma_{a_0,i})$. On the right, $h: V_h'\rightarrow V_h$ and its dividing arcs $\gamma_{h,i}:= \phi_h^{-1}\circ \phi_g(\gamma_{g,i})$.}
\label{domainh}
\end{figure}

\subsubsection{An `external map' $g$ for $\F_{a_0}$}\label{exmap}
  Let $a_0 \in \mathring \m$, and let $\alpha: \widehat \C \setminus \Lambda_{a_0,-} \rightarrow \widehat \C \setminus \overline \D$ be the Riemann map, normalized by $\alpha(\infty)=\infty$ and $\alpha(\gamma_{a_0}(t))\rightarrow 1$ as $t\rightarrow 0$. Define $\widehat V_g:=\alpha(V\setminus \Lambda_{a_0,-})$, and $\widehat V'_g:=\alpha(V'_{a_0}\setminus \Lambda_{a_0,-})$ (see Figure \ref{domainh}). Then the map 
$$\widehat g:= \alpha\circ \F_{a_0}\circ \alpha^{-1}: \widehat V'_g \rightarrow \widehat V_g,$$ 
is a holomorphic degree $2$ covering by construction.
Let $\tau(z)= 1/\bar{z}$ be the reflection with respect to the unit circle, and define the sets $\check V_g:=\tau(\widehat V_g)$,
$\check V'_g:=\tau(\widehat V'_g)$, $V_g:=\widehat V_g \cup \S^1 \cup \check V_g$ and 
$V'_g:=\widehat V'_g \cup \S^1 \cup \check V'_g$.
Applying the strong reflection principle with respect to $\S^1$ we can analytically extend the map  $\widehat g: \widehat V'_g
\rightarrow \widehat V_g$ to $$g: V'_g \rightarrow V_g.$$
%$\widetilde g: \widetilde{V'_g} \rightarrow \widetilde{V_g}$. 
By construction, $g$ has a parabolic fixed point 
at $z=1$ %with multiplicity $3$, and hence 
with two repelling petals $\Xi_{g,1},\,\Xi_{g,2}$, which intersect the unit circle $\S^1$. 
%The restriction $\widetilde g|_{\S^1}: \S^1\rightarrow \S^1$ is a holomorphic degree $2$ covering map with a parabolic fixed point at 
%$z=1$ of 
%\textcolor{red}{the same} multiplicity, 
%$3$ 
%and preimage at (say) $w\in \S^1$.
%Define the sets $$V_g:= \widehat V_g\cup \{1\},\,\,\,V'_g:=\widehat V'_g\cup \{1,w\},$$ and define the map $g$ to be the restriction of 
%$\widetilde g$ to $V'_g$, that is
%$$g:=\widetilde g|_{V'_g}: V'_g \rightarrow V_g.$$
Define $\gamma_{g,i}:=\alpha(\gamma_{a_0,i})$, for $i=1,2$. Then by construction $\gamma_{g,i}\in \Xi_{g,i},\,\,i=1,2$. 
Let $\phi_{g,1},\,\phi_{g,2}$ be repelling Fatou coordinates defined in $ \Xi_{g,1},\Xi_{g,2}$ respectively, with axis tangent 
to the unit circle at the parabolic fixed point.
Define $\gamma_g:=\gamma_{g,1}\cup\{1\}\cup\gamma_{g,2}$, and $O_g:= V_g \setminus \alpha(\P\A_{a_0})$. Then $\partial O_g:= \alpha(\partial O_{a_0})\cup \{1\}$ is a quasicircle, as it is a piecewise $C^1$-curve with no zero angle. 
%The map $g$ will play the role of an external map for $\F_{a_0}$, and it will be crucially important for constructing a quasiconformal map between a fundamental annulus $\A_{a_0}$ for $\F_{a_0}$ and a fundamental annulus $%%\A_h$ for $h$ (see Proposition \ref{arcs}).

%\subsubsection{Quasiconformal map}\label{s1}
\subsubsection{Gluing $\F_a$ on $O_{a_0}$ to $h$ on $\widehat{\C}\setminus O_h$}\label{s1}
Choose $\epsilon>0$, let $V_h:= \D(-\epsilon,1+\epsilon)$, and $V'_h:=h^{-1}(V_h)$. 
The map $h$ has two repelling petals $\Xi_{h,i},\,\,i=1,2$ which intersect the unit circle; let $\phi_{h,1},\,\,\phi_{h,2}$ be repelling 
Fatou coordinates defined on $ \Xi_{h,1},\Xi_{h,2}$ respectively and with axis tangent to the unit circle at the parabolic fixed point.
On $\overline V_h$ define the arcs $\gamma_{h,i}:= \phi_{h,i}^{-1}\circ \phi_{g,i}(\gamma_{g,i})$, and 
$\gamma_h:=\gamma_{h,1}\cup\{1\}\cup \gamma_{h,2}$.
Let $\phi=\phi_h^{-1}\circ \phi_g: \gamma_g \rightarrow \gamma_h$ be
 $$\phi :=\left\{
\begin{array}{cl}
\phi_{h,1}^{-1}\circ \phi_{g,1} &\mbox{on  } \gamma_{g,1}\\
\phi_{h,2}^{-1}\circ \phi_{g,2} &\mbox{on  } \gamma_{g,2}\\
\end{array}\right.$$
Then $\phi$ is a conjugacy between $g|\gamma_g$ and $h|\gamma_h$. This conjugacy is quasisymmetric by Proposition \ref{arcs}, which we delay until the end of 
this subsection to avoid interrupting the construction of the surgery.

\begin{remark} Other choices are possible for $V_h$ (see Proposition \ref{same}).
%we shall make an alternative choice based on the 
%{\it Milnor tiling} (Figure \ref{tiles}) of the disc. Since the choice is made once and for all at the initial parameter value $a_0$, the whole construction goes through with no problem, but the resulting extension of $\chi$ outside $\m$ may change.}
\end{remark}

Let $z_{g,i}:=\partial V_g \cap \gamma_{g,i},\,i=\{1,2\}$ and $z_{h,i}:=\partial V_h \cap \gamma_{h,i},\,i=\{1,2\}$.
Let $\psi_1: \partial V_g \rightarrow \partial V_h$ be a diffeomorphism sending the points $z_{g,1}$ and $z_{g,2}$ to $z_{h,1}$ and $z_{h,2}$ respectively,
and let $\psi_2:\partial V'_{g} \rightarrow \partial V_h'$ be a lift of 
$\psi_1$ to double covers (so that $\psi_1\circ g=h\circ \psi_2$).
Note that $\gamma_{g}$ divides $V_g \setminus V'_g$ into three connected components, two of which have the parabolic fixed point $1$ on the boundary. Call the third one $Q_g$ (so $\partial Q \cap \{1\} = \emptyset$). Similarly, let $Q_h$ be the connected component of $(V_h \setminus V'_h)\setminus \gamma_h$ with $\partial Q_h \cap \{1\} = \emptyset$.
%Let $Q_g:=\mathcal C_g \setminus (B_{g,1}\cup B_{g,2})$ (cf $Q_a$ in Figure \ref{cross}) and let
% the one not boundarying the parabolic fixed point $1$. Similarly, 
% $\gamma_{h}$ divides $V_h \setminus V'_h$ in three connected components, let 
%$Q_h: =\mathcal C_h \setminus (B_{h,1}\cup B_{h,2})$.
% be the one not boundarying the parabolic fixed point $1$.
 Then both $\partial Q_{g}$ and  $\partial Q_{h}$  are quasicircles, as they are piecewise $C^1$ curves without zero angles, so that the quasisymmetric map
 $$\Phi_Q: \partial Q_{g} \rightarrow \partial Q_h$$ defined as:
 
 $$\Phi_{Q} :=\left\{
\begin{array}{cl}
\psi_1 &\mbox{on  } \partial V_g\\
\psi_2 &\mbox{on  } \partial V'_{g} \\
\phi_h^{-1}\circ \phi_g\ &\mbox{on  } \gamma_{g,i},i=[1,2] \cap \partial Q_g\\
\end{array}\right.$$
extends to a quasiconformal map $$\Psi_{Q}: \overline Q_g \rightarrow \overline Q_h.$$

The arc $\gamma_{h}$ divides $V'_h$ into three connected components: the set $O_h$ is the one containing the unit disc.
Define $U_g:= O_g \cup Q_g$, and similarly $U_h:= O_h \cup Q_h$.
As both $\partial U_{g}$ and  $\partial U_{h}$  are piecewise $C^1$ curves without zero angles, 
%and so  are quasicircles, 
the quasisymmetric map
$\Phi_{U}: \partial U_{h} \rightarrow \partial U_h$ defined as:
  $$\Phi_{U} :=\left\{
\begin{array}{cl}
\psi_1 &\mbox{on  } \partial V_g\\
\phi_h^{-1}\circ \phi_g\ &\mbox{on  } \gamma_{g,i},i=[1,2] \\
\end{array}\right.$$
extends to a quasiconformal map $\Psi_{U^c}: \widehat \C \setminus U_g \rightarrow \widehat \C \setminus U_h.$
Define the quasiconformal map $\Psi_{O^c}: \widehat \C \setminus O_{g} \rightarrow \widehat \C \setminus O_{h}$ to be
$$\Psi_{O^c} :=\left\{
\begin{array}{cl}
\Psi_Q &\mbox{on  } Q_g\\
\Psi_{U^c} &\mbox{on  } \widehat \C \setminus U_g. \\
\end{array}\right.$$

We can now define the quasiconformal map $f:\widehat \C \setminus O_{a_0} \rightarrow \widehat \C \setminus O_h$ to be
$$f:= \Psi_{O^c} \circ \alpha.$$

Putting everything together, define $F_{a_0}:\widehat \C \rightarrow \widehat \C$ to be
 
    $$F_{a_0}:=\left\{
\begin{array}{cl}
\F_{a_0}&\mbox{on  } O_{a_0}\\
f^{-1}\circ h\circ f &\mbox{on  }  \widehat \C \setminus O_{a_0}.\\
\end{array}\right.$$
Then $F_{a_0}$ is continuous, so quasiregular. On  $\widehat \C \setminus O_{a_0}$ define the Beltrami form 
$$\hat\mu_{a_0}:=f^*(\mu_0),$$ 
and on $\widehat \C$ the Beltrami form

 $$ \mu_{a_0}:=\left\{
\begin{array}{cl}
\hat\mu_{a_0} &\mbox{on  }\widehat \C \setminus O_{a_0}\\
(\F_{a_0}^n)^*({ \hat\mu_{a_0}})&\mbox{on  } \F_{a_0}^{-n}(\mathcal{C}_{a_0}) \cap O_{a_0}\\
\mu_0 &\mbox{ on } \Lambda_{a_0,-}\\
\end{array}\right.$$
and note that, by construction, $\mu_{a_0}$ is $F_{a_0}$-invariant and has $|\mu_{a_0}|_{\infty}<1$.\
 
It follows by the Measurable Riemann Mapping Theorem that there exists a quasiconformal homeomorphism 
$ \phi_{a_0}: \widehat \C \rightarrow\widehat \C$, integrating $\mu_{a_0}$, and sending the parabolic fixed point $P_a$ to $\infty$, its preimage $S_a$ to $0$, and the critical point $c_{a_0}$ to $-1$.
Define  $\widetilde F_{a_0}$ to be the holomorphic map of degree two 
$$\widetilde F_{a_0}:=\phi_{a_0} \circ F_{a_0} \circ \phi_{a_0}^{-1}\,: \widehat \C \rightarrow \widehat \C,$$
and observe that $\widetilde F_{a_0}$ is a quadratic rational map, with a parabolic fixed point at $\infty$ having multiplier $1$, 
with the other preimage of this fixed point at $0$, and with a critical point at $-1$,
%Thus $\widetilde F_{a_0} \in Per_1(1)$, 
concluding our surgery construction converting
the correspondence $\F_{a_0}$ into an element of $Per_1(1)$.\\

As promised in the first paragraph of this subsection, we now present a proof that the map $\phi:\gamma_g \to \gamma_h$ defined there
is quasisymmetric.
This employs a similar argument to that in Proposition 5.3 (ii) of \cite{L1}: we give details for the convenience of the reader.

\begin{prop}\label{arcs}
The map  $\phi:=\phi_h^{-1}\circ \phi_g: \gamma_g \rightarrow \gamma_h$ defined as 

 $$\phi =\left\{
\begin{array}{cl}
\phi_{h,1}^{-1}\circ \phi_{g,1} &\mbox{on  } \gamma_{g,1}\\
\phi_{h,2}^{-1}\circ \phi_{g,2} &\mbox{on  } \gamma_{g,2}\\
\end{array}\right.$$
is a quasisymmetric conjugacy between $g|\gamma_g$ and $h|\gamma_h$. 
\end{prop}
\begin{proof}

It is clear that $\phi$ is a conjugacy and that it is quasisymmetric on each of the two halves $\gamma_{g,i}$ $i=1,2$, of $\gamma_g$, since $\phi_{g,i}$ and 
$\phi_{h,i}$ are Fatou coordinates for a repelling petal of $g$ and a repelling petal of $h$ respectively, hence they are diffeomorphisms, and since the curves $\gamma_{g,i}$, and  $\gamma_{h,i}$ $i=1,2$, are quasiarcs (arcs of quasicircles). So it just remains to check that $\phi$ is quasisymmetric at the parabolic point, where the curves $\gamma_{g,i}$ $i=1,2$, join; we will show this by looking at the asymptotics of Fatou coordinates at the parabolic point, following the proof of Proposition 5.3(ii) in \cite{L1}.

Let $T(z)=z+1$, then $\widehat g(z):=T\circ g\circ T^{-1}(z)=z+ az^3+h.o.t.$ for some $0\ne a\in \C$.
%, to move the parabolic fixed point to the origin. So, near $z=0$ 
%we have $$\widehat g(z):=T\circ g\circ T^{-1}(z)=z+ az^3+h.o.t.,\,\,\,\,a\in \C.$$ 
Let $\widehat \Xi_{g,i}:=T^{-1}(\Xi_{g,i})$ be repelling petals for $\widehat g$, set $\widehat \gamma_{g,i}:=T^{-1}(\gamma_{g,i})$, and $\widehat \gamma_g:=\widehat\gamma_{g,1}\cup \{0\}\cup \widehat\gamma_{g,2}$. Let $\widehat \phi_{g,i}: \widehat\Xi_{g,i}\rightarrow \H_l$ be repelling Fatou coordinates for $\widehat g$ with axis tangent to the imaginary axis at the parabolic fixed point. 
Then $$\widehat\phi_{g,i}= \Phi_{g,i} \circ I_g,$$ where
 $I_g(z):= -\frac{1}{2az^{2}}$ 
 conjugates $\widehat g$ to the map $\widehat G(w)= w+1+\frac{\hat a}{w}+O(\frac{1}{w^2})$ on $\widehat\Xi_{g, i},i=1,2$, 
 and the map $\Phi_{g, i}(w) \,=\, w\, - \,\hat{a}\log(w)\,
 + \,c_{i} \,+ \,o(1)$ conjugates $\widehat G(w)$ to the map $T(z)=z+1$ on $I_g(\widehat\Xi_{g, i}),i=1,2$,
 (see Proposition 2.2.1 in \cite{Sh}). 
Define $\Gamma_{g,1}:=I_g(\widehat \gamma_{g,1})$,  $\Gamma_{g,2}:=-I_g(\widehat \gamma_{g,2})$ (see Figure \ref{da}), and 
$$\Gamma_g:=\Gamma_{g,1}\cup\{\infty\}\cup\Gamma_{g,2}.$$ 
\begin{figure}
\centering
\psfrag{g}{\small $\widehat g$}
\psfrag{p}{\small $0$}
\psfrag{h}{\small $\widehat h$}
\psfrag{1}{\small $\widehat \gamma_g$}
\psfrag{2}{\small $\widetilde I_g$}
\psfrag{3}{\small $\Gamma_g$}
\psfrag{0}{\small $\infty$}
\psfrag{4}{\small $\widehat \Phi_g$}
\psfrag{5}{\small $\gamma$}
\psfrag{6}{\small $\widehat \Phi_h$}
\psfrag{7}{\small $\Gamma_h$}
\psfrag{8}{\small $\widetilde I_h$}
\psfrag{9}{\small $\widehat \gamma_h$}
	\includegraphics[width= 8cm]{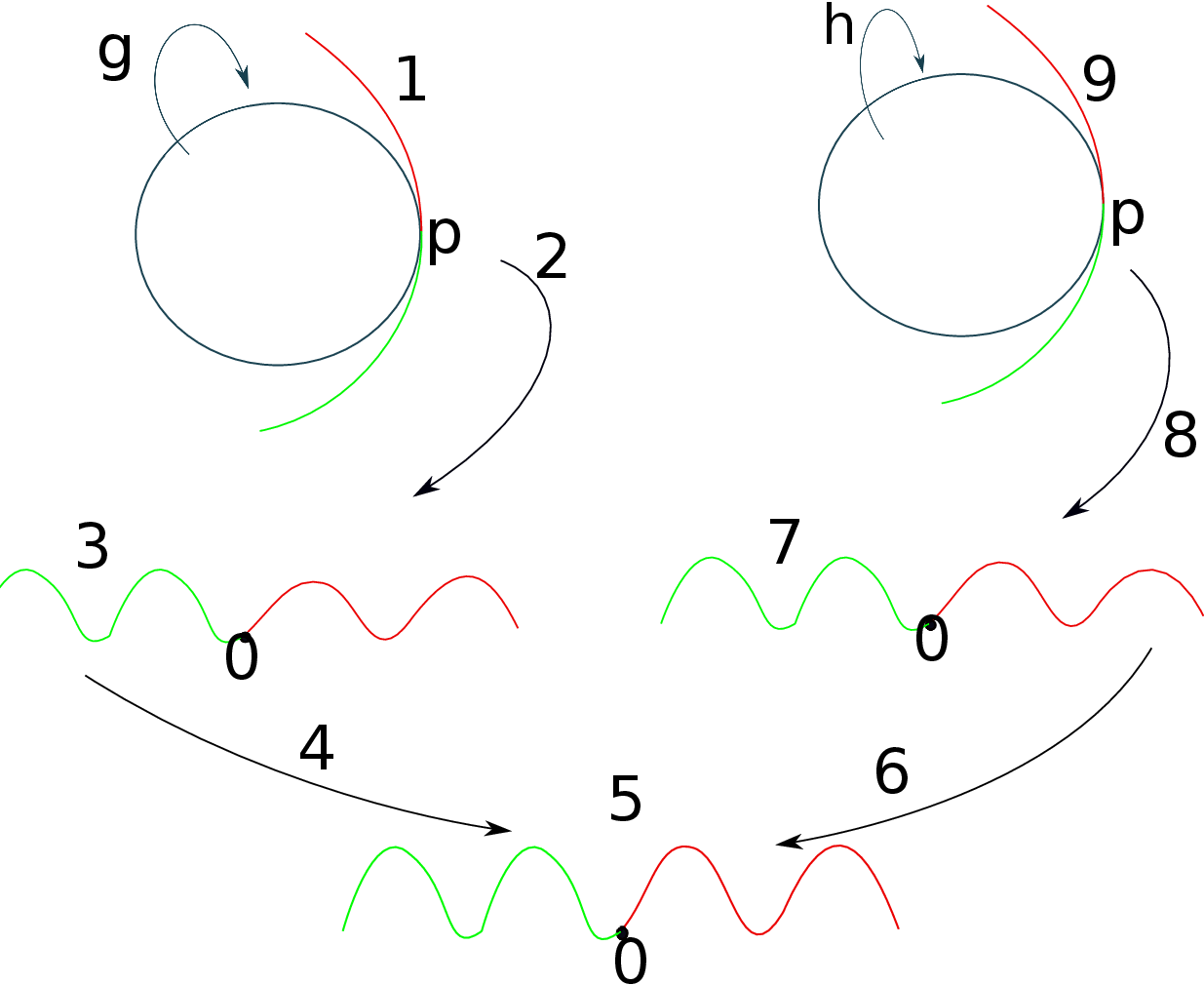}
\caption{\small Construction of the map $\widehat\phi_h^{-1}\circ\widehat\phi_{g}=T\circ \phi_{h}^{-1}\circ \phi_{g}\circ T^{-1}=\widetilde{I_h}^{-1} \circ \widetilde\Phi_h^{-1}\circ \widetilde\Phi_g \circ\widetilde{I_g} : \widehat\gamma_{g} \rightarrow
\widehat\gamma_h$.}
\label{da}
\end{figure}
The map
$\widetilde I_g: \widehat \gamma_g \rightarrow \Gamma_g$ (see Figure \ref{da}, top right), defined as:
$$ \widetilde{I_g}(z)=\left\{
\begin{array}{cl}
I_g(z) &\mbox{on  } \widehat\gamma_{g,1}\\
-I_g(z) &\mbox{on  } \widehat\gamma_{g,2}\\
\end{array}\right.
$$  
is quasisymmetric on a neighborhood of $0$. 
%We will show that we can extend Fatou coordinates to obtain a diffeomorphism between $\Gamma_g$ and the corresponding $\Gamma_h$ for $h$, so the composition $\widehat\gamma_{g} \rightarrow
%\widehat\gamma_h$ is quasisymmetric.
Set $\gamma_{1}:=(\Phi_{g,1} \circ I_g)(\widehat \gamma_{g,1}),$ $\gamma_{2}:=(\Phi_{g,2} \circ I_g)(\widehat\gamma_{g,2})$ and $\gamma:=\gamma_1\cup\{\infty\}\cup-\gamma_2$. Define the map  $\widetilde \Phi_g: \Gamma_g \rightarrow \gamma$ as
follows (see Figure \ref{da}, bottom right):
$$\widetilde \Phi_g(w)=\left\{
\begin{array}{cl}
\Phi_{g,1}(w) &\mbox{on  } \Gamma_{g,1}\\
- \Phi_{g,2}(-w) &\mbox{on  } \Gamma_{g,2}\\
\end{array}\right.
$$  
Note that the map $\widetilde \Phi_g$ is conformal on $\Gamma_g \setminus \infty$.
Since $\Phi_{g, i}(w) \,=\, w\, - \,\hat{a}\log(w)\,
 + \,c_{i} \,+ \,o(1)$, the maps
$\Phi_{g,i}$ have derivatives $\Phi'_{g,i}=1+o(1)$ at $\infty$, hence the map $\widetilde \Phi_g: \Gamma_g \rightarrow \gamma$
is a diffeomorphism.

Repeating the process for the map $h$, we can write the map 
$\widehat\phi_h^{-1}\circ\widehat\phi_{g}$ as
$$\widehat\phi_h^{-1}\circ\widehat\phi_{g}=T\circ \phi_{h}^{-1}\circ \phi_{g}\circ T^{-1}=\widetilde{I_h}^{-1} \circ \widetilde\Phi_h^{-1}\circ \widetilde\Phi_g \circ\widetilde{I_g} : \widehat\gamma_{g} \rightarrow
\widehat\gamma_h$$ 
which is quasisymmetric.
 \end{proof}

\subsubsection{Surgery for the whole family}\label{wf}

The surgery so far has been to convert a single correspondence $\F_{a_0}$ into a rational map in $Per_1(1)$. We now
move this holomorphically to make a surgery construction for the whole family $\F_a,\ a\in \K$. 
In Section \ref{holom} we constructed a holomorphic motion 
 $$\widetilde\tau: \mathring K  \times \widehat \C \setminus O_{a_0}\rightarrow\widehat \C,$$
 with restriction $\widetilde \tau_{a_0}(\P\A_{a_0})=\P\A_a$.
  Let
 $f: \widehat \C \setminus O_{a_0}\rightarrow \widehat \C \setminus O_{h}$ be the quasiconformal map we constructed in Section \ref{s1}. We now 
define the map $$T_a:=f\circ\widetilde\tau_a^{-1}: \widehat \C \setminus O_{a}\rightarrow\widehat \C \setminus O_h$$
 that will play the role of a \textbf{tubing} (see Figure \ref{cwf}). \\
     \begin{figure}
\centering
\psfrag{s}{\small $\tau: \widehat \C \setminus O_{a_0}\rightarrow \widehat \C \setminus O_{a}$}
\psfrag{T}{\small $T_a:=f\circ\tau_a^{-1}: \widehat \C \setminus O_{a}\rightarrow\widehat \C \setminus O_h$}
\psfrag{b}{\small is the tubing}
\psfrag{h}{\small $h$}
\psfrag{f}{\small $\F_{a_0}$}
\psfrag{a}{\small $\Phi_{\Delta_g,\Delta_h}$}
\psfrag{d}{\small $\F_a$}
\psfrag{g}{\small $f$}
	\includegraphics[width= 6.3cm]{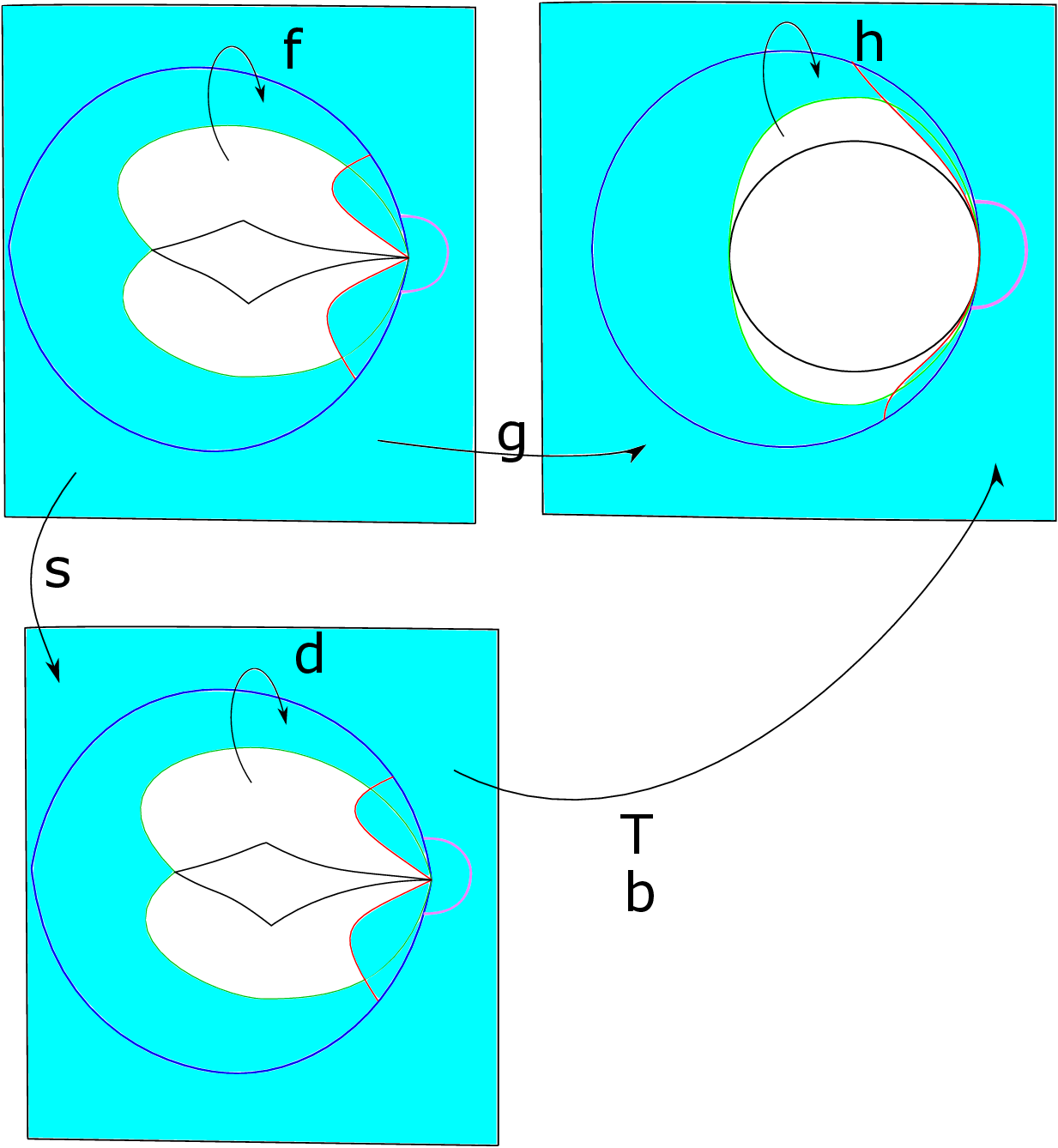}
\caption{\small Construction for the whole family}
\label{cwf}
\end{figure}

For $a \in \mathring K$ define the map $F_a:\widehat \C \rightarrow \widehat \C$ to be
$$ F_{a} :=\left\{
\begin{array}{cl}
\F_{a} &\mbox{on  } O_a \\
(T_a)^{-1}\circ h\circ T_a &\mbox{on  } \widehat \C \setminus O_{a}\\
\end{array}\right.$$
and note that this map is continuous by construction, and hence quasiregular of degree $2$.
Define the Beltrami form $\hat\mu_a=T_a^*(\mu_0) $ on $\widehat \C \setminus O_{a}$, and on $\widehat \C$ define the Beltrami form  
 $$ \mu_{a}: =\left\{
\begin{array}{cl}
\hat \mu_a&\mbox{on  } \widehat \C \setminus O_{a}\\
(F^n)^*(\hat \mu_a) &\mbox{on  } O_{a} \setminus \Lambda_{a,-}\\
\mu_0 &\mbox{on  } \Lambda_{a,-}\\
\end{array}\right.$$
By construction, $\mu_a$ is $F_a$-invariant and $|\mu_a|_{\infty}<1$. 
Hence, by the Measurable Riemann Mapping Theorem, there exists a unique quasiconformal homeomorphism 
$ \phi_{a}: \widehat \C \rightarrow\widehat \C$, integrating $\mu_{a}$, and sending the parabolic fixed point $P_a$ to $\infty$, 
the critical point $c_{a}$ to $+1$, and $T_a^{-1}(\infty)$ to  $-1$.
The map 
$$\widetilde F_{a}:=\phi_{a} \circ F_{a} \circ \phi_{a}^{-1}\,: \widehat \C \rightarrow \widehat \C,$$
is a holomorphic degree $2$ map, and hence a quadratic rational map, with a parabolic fixed point at $\infty$ having multiplier $1$, 
and critical points at $-1$ and $+1$. Thus $\widetilde F_{a}(z)=z+1/z+A$ for some (unique) $A\in \C$.
\begin{remark}\label{dil}
Note that, for each $a \in \K$, the dilatation of the integrating map $\phi_a$ is the same as that of the tubing $T_a$.
Indeed, on $\widehat \C \setminus O_{a}$ we have $\mu_a= T_a^*(\mu_0)$, and on the preimages of $\P\A_a \in O_a$ the Beltrami form $\mu_a$ is obtained by spreading $T_a^*(\mu_0)$ by the dynamics of $\F_a|_{O_a}$ (a holomorphic map, so it does not change the dilatation of $\mu_a$), while on $\Lambda_{a,-}$ we have $\mu_a=\mu_0$. 
\end{remark}

\begin{remark}
By the Measurable Riemann Mapping Theorem, there also exists a unique quasiconformal map $ \psi_{a}: \widehat \C \rightarrow\widehat \C$ integrating $\mu_{a}$ and sending $P$ to $\infty$, 
$c_{a}$ to $-1$ and $T_a^{-1}(\infty)$ to $+1$. If $\phi_{a} \circ F_{a} \circ \phi_{a}^{-1}(z)= z+1/z +A$, for some $A\in \C$, we have that $\psi_{a} \circ F_{a} \circ \psi_{a}^{-1}(z)=z+1/z-A$. 
\end{remark}

\begin{prop}\label{welldefined}
For each $a\in \K \cap \m$ there exists a unique $B \in \mathcal{M}_1$ such that $P_A(z)=z+1/z+A$, where $B=1-A^2$, is hybrid equivalent to $\F_a$ on a doubly pinched neighbourhood of $\Lambda_{a,-}$.
\end{prop}
\begin{proof}
Assume that there exist $\phi_1: \widehat \C \rightarrow \widehat \C$ and $\phi_2: \widehat \C \rightarrow \widehat \C$ quasiconformal maps such that 
$\widetilde F_{a,1}=\phi_{1} \circ F_{a} \circ \phi_{1}^{-1}=P_{A_1}$, and
$\widetilde F_{a,2}=\phi_{2} \circ F_{a} \circ \phi_{2}^{-1}=P_{A_2}$.
Then $P_{A_1}$ is hybrid equivalent to $P_{A_2}$ on a doubly pinched neighbourhood of the filled Julia set $K_{A_1}$, and so, by Proposition 6.5 in \cite{L1}, $P_{A_1}$ is M\"obius conjugate to $P_{A_2}$. Hence $(A_1)^2=(A_2)^2$, so $B_1=B_2$.
\end{proof}

\section{The map $\chi: \m \rightarrow M_1$ is a homeomorphism}\label{chi}
Our surgery construction converts a correspondence $\F_a$, $a\in\K$, into a rational map $P_A:a \to z+1/z+A$. A priori this rational map might depend on
all the choices made along the way in the construction, but if $a\in \m\cap\K$ then Proposition \ref{welldefined} guarantees that $P_A$ is unique up to 
$A\sim -A$. Moreover as $\F_a$ has connected limit set $\Lambda_{a,-}$, the hybrid equivalent map $P_A$ has connected filled Julia set $K_A$ and so
we may define a map
$$\chi:\m\cap\K \to \mathcal{M}_1$$
by $\chi(a)=B$, where $B=1-A^2$. We shall first prove that $\chi$ is injective, using the Rickmann Lemma (see Proposition \ref{inj}).
We will then define a quasiregular extension of $\chi |_{\m\cap\K}$ to $\K\setminus\m$
(Section \ref{extension} and Proposition \ref{extnqr}), which will turn out to be identical to the function $\chi$ that we have already defined via surgery
(Proposition \ref{same}).
In Section \ref{continuity}, we prove that $\chi$ is continuous on the whole of $\K$, using the Ma\~n\'e-Sad-Sullivan
partition of $\L\f'$ into $\mathring \m$, $\partial \m\setminus\{7\}$ and $\L\f'\setminus \m$ (Section \ref{mss}), and the now classical method
of Douady and Hubbard \cite{DH} in the formulation presented by Lyubich \cite{Lyu}. 
%More precisely, we prove that 
%$\chi$ is continuous on $\partial\m \cap \mathring K$ using precompactness of families of quasiconformal homeomorphisms 
%and quasiconformal rigidity on $\partial \mathcal{M}_1$ (the `miracle of continuity on the boundary', see Proposition \ref{conti}); and we prove  
%holomorphicity of $\chi$ 
%on the interior of $\m$ by using the multiplier map on hyperbolic components (Proposition \ref{hypercom}) and holomorphic motion 
%on queer components (Proposition \ref{queer}). 
In Section \ref{hom} we show that $\chi: \K \rightarrow \chi(\K)$ is a homeomorphism (Proposition \ref{homeo_on_lune}).  
Finally, an analysis of the positions and diameters of the limbs of $\m$ in a neighbourhood of its root point $a=7$ and those of the corresponding limbs of ${\mathcal M}_1$
will enable  us to prove that every $B\in \mathcal{M}_1\setminus \{1\}$ is in the image of $\chi(\K)$ (for $K=\L\f'\setminus N)$ with the neighbourhood $N$ of $a=7$ sufficiently small (Section \ref{chi_surj}), 
and to deduce that $\chi$ is a homeomorphism from $\m$ to $\mathcal{M}_1$ (Section \ref{final}).

\subsection{Injectivity of $\chi:\m\cap\K\rightarrow \mathcal{M}_1$}\label{injectivity}
%We start by proving the following Lemma, which will enable us to prove injectivity of $\chi|_{\m}$ in the next Proposition \ref{inj}.
\begin{lemma}\label{doublypinchedconjugacy}
If $\chi(a_1)=\chi(a_2)$, then  $\F_{a_1}$ and $\F_{a_2}$ are hybrid equivalent on quadruply pinched neighbourhoods of 
$\Lambda_{a_1}=\Lambda_{a_1,-}\cup\Lambda_{a_1,+}$ and $\Lambda_{a_2}=\Lambda_{a_2,-}\cup\Lambda_{a_2,+}$ respectively.
\end{lemma}
\begin{proof}

Assume  $\chi(a_1)=\chi(a_2)$. Then the composition $\phi_{a_1,a_2}:=\phi_{a_2}^{-1}\circ \phi_{a_1}$ is a hybrid
conjugacy between $\F_{a_1}$ and $\F_{a_2}$ on $O_{a_1}$. 

Note that $O_{a_1}$ 
is a doubly pinched neighbourhood of 
$\Lambda_{a_1,-}$, pinched at $P_{a_1}$ and $S_{a_1}$ (where $S_{a_1}$ is the preimage of $P_{a_1}$), and that
$O_{a_1} \subset V_{a_1} \subset \Delta_J$ (where $\Delta_J$ is the fundamental domain for the involution, which in $z$ coordinates is $J(z)=-z$). 
So the set 
$$\hat O_{a_1}=O_{a_1} \cup J(O_{a_1})$$ is a quadruply pinched neighbourhood of $\Lambda_{a_1}$, pinched at $S_{a_1}$, at 
$J(S_{a_1})$, and doubly pinched (from both sides) at the parabolic fixed point $P_{a_1}$.
The map 
$\overline \phi_{a_1,a_2}: \hat O_{a_1} \rightarrow \hat O_{a_2}$ defined as:
$$\overline \phi_{a_1,a_2}(z):= \left\{
\begin{array}{cl}
\phi_{a_1,a_2}(z) &\mbox{ if } z \in O_{a_1}\\
J(\phi_{a_1,a_2}(J(z)))
 &\mbox{ if } z \in J(O_{a_1})\\
\end{array}\right.
$$\\
is a hybrid conjugacy between $\F_{a_1}$ and $\F_{a_2}$ on a quadruply pinched neighbourhood of $\Lambda_{a_1}$.
\end{proof}

\begin{prop}\label{inj}
The straightening map $\chi: \m \cap \mathring K \rightarrow \mathcal{M}_1$ is injective.
\end{prop}
\begin{proof}
 If $\chi|_{\m}$ is not injective, there exist two different correspondences $\F_{a_1}$ and $\F_{a_2}$, both with connected 
 limit set, hybrid equivalent to the same map $P_A$.
 Then, by Lemma \ref{doublypinchedconjugacy} above
 there exists a hybrid conjugacy $\overline \phi_{a_1,a_2}$ between $\F_{a_1}$ and $\F_{a_2}$ defined on the quadruply pinched neighbourhood $\hat O_{a_1}$ of $\Lambda_{a_1}$.

On the other hand, 
for every $a\in\m$ the Riemann map $R_a: \Omega_a \rightarrow \H$ conjugates the action of $\F_a$
on the regular set $\Omega_a=\widehat \C \setminus \Lambda_a$ to the action of the modular group on the upper half plane $\H$ (see \cite{BL1}, Theorem A).
Hence, the map $R_{a_1,a_2}:= R_{a_2}^{-1} \circ R_{a_1}: \Omega_{a_1} \rightarrow \Omega_{a_2}$ is a holomorphic conjugacy between $\F_{a_1}$ and $\F_{a_2}$ on their regular sets, and the map $\Psi: \widehat \C \rightarrow \widehat \C$ defined by:
$$ \Psi :=\left\{
\begin{array}{cl}
\overline \phi_{a_1,a_2} &\mbox{on  } \Lambda_{a_1}\\
R_{a_1,a_2} &\mbox{on  } \Omega_{a_1} \\
\end{array}\right.
$$
is a conjugacy between $\F_{a_1}$ and $\F_{a_2}$ on the whole Riemann sphere, holomorphic on $\mathring \Lambda_{a_1}$ and on $\Omega_{a_1}$. 
If $\Psi$ is continuous, then by the Rickmann Lemma it is holomorphic on $\widehat \C$, and
then by Lemma \ref{unique}, $\F_{a_1}=\F_{a_2}$. So we need to show
that $\Psi$ is continuous, more specifically that both $\overline \phi_{a_1,a_2}(z)$ and $R_{a_1,a_2}(z)$ tend to the same point when $z\in \Omega_{a_1}$  tends to $\Lambda_{a_1}$.

Let $S_{a_1}:= R_{a_1}(\hat O_{a_1} \cap \Omega_{a_1}) \subset \H$, $S_{a_2}:= R_{a_2}(\overline \phi_{a_1,a_2}(\hat O_{a_1}\cap \Omega_{a_1}))\subset \H$, 
and let $\overline \Psi$ denote the quasiconformal homeomorphism
$$\overline \Psi:=R_{a_2} \circ \overline \phi_{a_1,a_2} \circ R_{a_1}^{-1}: S_{a_1} \rightarrow S_{a_2}.$$
Foliate $S_{a_1}\subset \H$ by vertical line segments $L_x$, one for each $x\in \R$. Its image
$S_{a_2}={\overline \Psi}(S_{a_1})\subset \H$ is foliated by the paths ${\overline \Psi}(L_x)$. When $x \in \Q$, the path $R_{a_1}^{-1}(L_x)$ in $\Omega_{a_1}$
lands on $\Lambda_{a_1}$ (see \cite{BL2}, Proposition 1 and Corollary 1). The landing point is a preperiodic, indeed pre-fixed, point $z_0$ of $\F_{a_1}$ (or $\F_{a_1}^{-1}$
if $z_0\in \Lambda_{a_1,+}$).  The path $\overline \phi_{a_1,a_2}(R_{a_1}^{-1}(L_x))$
in $\Omega_{a_2}$ lands on $\Lambda_{a_2}$ at the corresponding pre-fixed point $\overline \phi_{a_1,a_2}(z_0)$ of $\F_{a_2}$.
But $\overline \phi_{a_1,a_2}(R_{a_1}^{-1}(L_x))=R_{a_2}^{-1}(\overline \Psi(L_x))$, so we deduce that the path $\overline \Psi(L_x)$ in $\H$ lands on the real axis at $x$.
Since this holds for every $x\in \Q$, and $\overline \Psi$ preserves the order of the leaves of the foliation of $S_{a_1}$, it follows that for every $x\in \R$ the path
$\overline \Psi(L_x)$ lands on the real axis at $x$, and thus that $\overline \Psi:S_{a_1} \to S_{a_2}$ extends continuously to the identity map on
$\partial\H=\widehat \R$. 
We deduce that for every sequence $z_n\in \Omega_{a_1}$ converging to a point $\hat z\in \Lambda_{a_1}$, the sequences $R_{a_2}(\overline \phi_{a_1,a_2}(z_n))$ and 
$R_{a_1}(z_n)$ in $\H$ converge to the same point of $\partial \H$. Hence the sequences $\overline \phi_{a_1,a_2}(z_n)$ and $R_{a_1,a_2}(z_n)=R_{a_2}^{-1}\circ R_{a_1}(z_n)$ in 
$\Omega_{a_2}$ converge to the same point (necessarily $\overline \phi_{a_1,a_2}({\hat z})$) in $\Lambda_{a_2}$.
\end{proof}

\subsection{The map $\chi$ is locally quasiregular on $\K\setminus \m$}\label{extension}
In this section, we show that with the right choice of $V_h$ we can write $\chi$ on $\K \setminus \m$ as the composition of iterated lifted 
extended tubings and two conformal maps (see Section \ref{ext} and Proposition \ref{same}). The
local quasiregularity of $\chi$ on $\K \setminus \m$ follows (see Proposition \ref{extnqr}).

\subsubsection{Extending the holomorphic motion $\tau$ by iterated lifting}\label{liftmotion}

Define $$\P\A^1_a:=\F_a^{-1}(\P\A_a)\cap O_a$$ 
and for each $k>1$ set 
$$\P\A^k_a:=\F^{-(k-1)}_a(\P\A^1_a).$$
Note that $\F_a:\P\A^1_a \to \P\A_a$ is a double covering over $\P\A_a\cap O_a$ but only a single covering over 
$\P\A_a\cap (\widehat C \setminus O_a)$,
and that all the $\F_a: \P\A^{k+1}_a \to \P\A^k_a$ are double-sheeted coverings.
The sets $\P\A_a$, $\P\A_a^k$ $(k\ge 1)$ and $\Lambda_{a,-}$ are disjoint, and their union is $V_a$.
Define corresponding subsets of the parameter space $\K$:
$$K^0:=\{a\in \K | v_a \in \P\A_a\} \ {\rm and} \ K^k:=\{a\in \K | v_a \in \P\A^k_a\}\ (k \ge 1),$$
and set 
$$K_{k}:=\mathring{K} \setminus \bigcup_{j=0}^{k-1} K^j,$$
noting that $(\m\cap\K)\subset K_{k}\subset  K_{k-1} \subset\ldots\subset K_0=\K$, and that each $K_k$ is the disjoint union of 
$K_{k-1}$ and $K^k$. Our base point $a_0$ is in $\m\cap \K$, so $a_0\in K_k$ for all $k\ge 0$.

When $a\in K_1$, 
the critical value $v_a$ is not in $\P\A_{a_0}$, so we can lift the qc homeomorphism
$\tau_a: \P\A_{a_0}\to \P\A_a$ to a qc homeomorphism of covers:
$$\tau^1_a:= \F_a^{-1}\circ \tau_a \circ \F_{a_0}\,\,:\P\A_{a_0}^1\rightarrow \P\A_a^1.$$
Despite the fact that $\P\A_{a_0}^1$ is defined as the intersection of $\F_{a_0}^{-1}(\P\A_{a_0}^1)$ with the complement of $O_{a_0}$, 
and the cover $\F_{a_0}: \P\A_{a_0}^1 \to \P\A_{a_0}$ is therefore double-sheeted in one part and single-sheeted in another, the lift exists
since the covering $\F_a: \P\A_a^1 \to \P\A_{a}$ is partitioned in exactly the corresponding way. 
When $a\in K_k$, $k\ge1$, we can iterate 
the process and lift $\tau^1_a$ to $2^{k-1}$-fold covers (this time they have $2^{k-1}$ sheets everywhere):
$$\tau^k_a:= \F_a^{-(k-1)}\circ \tau^1_a \circ \F^{k-1}_{a_0}\,\,:\P\A^k_{a_0}\rightarrow \P\A^k_a.$$
Denoting by $\tau_k$ the union over all $a\in K^k$ of $\tau_a$ and $\tau_a^j, \ 1\le j\le k$,
we have :
$$\tau_k: K_k \times (\P\A_{a_0} \cup \P\A_{a_0}^1 \cup... \cup \P\A_{a_0}^k) \rightarrow \C.$$
This union $\tau_k$ of lifts of $\tau$ is well-defined for all $a\in K_k$. Moreover it is continuous at the boundary between 
$\P\A_{a_0}$ and $\P\A^1_{a_0}$ because by definition the motion of the inner boundary $\partial V_{a_0}'$ of $\P\A_{a_0}$ is the lift (via $\F_{a_0}^{-1}$)
of that of the outer boundary $\partial V_{a_0}$. Similarly $\tau_k$ is continuous at the boundary between each $\P\A^j_{a_0}$ and $\P\A^{j+1}_{a_0}$. 
%We do not know that $K_k$ is connected, so technically $\tau_k$ is only a holomorphic motion 
%on each connected component of $K_k$, but nevertheless, even when $a_0$ and $a$ are in different components of $K_k$, we still 
%have a qc homeomorphism $\tau_a\cup\tau^1_a\cup\ldots\cup\tau^k_a$ from 
%$\P\A_{a_0} \cup \P\A_{a_0}^1 \cup... \cup \P\A_{a_0}^k$ to $\P\A_{a} \cup \P\A_{a}^1 \cup... \cup \P\A_{a}^k$, by sending 
%each point of $\P\A_{a_0}^j,\ (1\le j\le k)$ first to 
%$\P\A_{a_0}$ (via $\F_{a_0}^j$),  then applying the motion $\tau$ to move it to $\P\A_a$, and finally the appropriate branch of 
% $\F_a^{-j}$ to send it to $\P\A_a^j$.
For $a_0$ and $a\in \m\cap \K$, we can repeat the iterative process indefinitely, obtaining in the limit a well-defined lift:
$$\tilde \tau: (\m\cap\K) \times (V_{a_0} \setminus \Lambda_{a_0,-}) \rightarrow \C,$$
which is a holomorphic motion on each connected component of $\m\cap\K$.

\subsubsection{Extending the tubing}\label{ehm}
 In Section \ref{wf} we defined the tubing $T$ as the family of qc homeomorphisms
 $$T_a:=f\circ\widetilde\tau_a^{-1}: \widehat \C \setminus O_{a}\rightarrow\widehat \C \setminus O_h,$$
 one for each $a\in \K$, where 
 $f:=\Psi_{O^c}\circ \alpha: \widehat \C \setminus O_{a_0}\rightarrow \widehat \C \setminus O_{h}$
 is a quasiconformal homeomorphism chosen during the surgery construction on $\F_{a_0}$.
 % and $O_h$ is the bounded component of the 
 %complement of $\A_h$. We recall in particular that $f$ was constructed as an extension of a qc homeomorphism $\Psi:\A_{a_0} \to \A_h$.
 Denote $f( \P\A_{a_0})\subset \widehat \C \setminus O_h$ by $\P\A_h$.
 As $\F_{a_0}$ has connected limit set, there is no obstruction to lifting 
 $f_|:\P\A_{a_0}\to \P\A_h$ to $2^j$-fold covers for all $j$. Denoting this lift
 by $f^{(j)}$, and the lift of $\tau_a:\P\A_{a_0}\to \P\A_a$ to $2^j$-fold covers constructed in the previous subsection (\ref{liftmotion}) by $\tau_a^j$, 
 we extend the tubing $T_a$ by defining the qc homeomorphisms:
 $$T_a^j:= f^{(j)}\circ (\tau_a^j)^{-1}:\P\A_a^j \to \P\A_h^j :=h^{-j}(\P\A_h)\cap O_h\ \ (1\le j <\infty)$$
 for each $a \in K_k\subset\K$. 
 We note that $\bigcup_{j=0}^k T^j_a$ is continuous at the boundary between each $\P\A^j_a$ and $\P\A^{j+1}_a$, as 
 $\bigcup_{j=0}^k(\tau_a^j)^{-1}$ has this property.
 We have the following commuting diagram of qc homeomorphisms, where the bottom
 row (read left to right) is the composition $T_a$, and the top row is its lift $T_a^j$:
 
 \begin{picture}(220,100)
 
 \put(45,20){$\P\A_a$} 
 \put(125,20){$\P\A_{a_0}$}
 \put(210,20){$\P\A_h$}
 \put(120,23){\vector(-1,0){50}}
 \put(155,23){\vector(1,0){50}}

 \put(45,70){$\P\A_a^j$} 
 \put(125,70){$\P\A_{a_0}^j$}
 \put(210,70){$\P\A_h^j$}
 \put(120,73){\vector(-1,0){50}}
 \put(155,73){\vector(1,0){50}}
 
 \put(55,65){\vector(0,-1){30}}
 \put(135,65){\vector(0,-1){30}}
 \put(215,65){\vector(0,-1){30}}
 
 \put(60,45){$\F_a^j$}
 \put(140,45){$\F_{a_0}^j$}
 \put(220,45){$h^j$}
 
\put(90,78){$\tau_a^j$}
\put(170,78){$f^{(j)}$}
\put(90,28){$\tau_a$}
\put(170,28){$f_|$}

\end{picture}

If $a\in \m$, we can set $\overline T_a=T_a\cup \bigcup_{j=1}^{\infty}T^j_a$, and observe that $\overline T_a$ is a qc homeomorphism
$$\overline T_a: \widehat \C \setminus \Lambda_{a,-} \rightarrow \widehat \C \setminus \overline \D.$$
On the other hand, if $a\notin \m$, we can only iterate the double covering procedure until we reach the critical value $v_a$ of $\F_a$. 
To be precise, if $v_a \in \P\A^n_a$, then the last level to which we can lift the tubing is $T^n_a: \P\A_a^n \to \P\A_h^n$.

\subsubsection{Milnor's model for  $\C \setminus \mathcal{M}_1$}\label{Mil}
In the paper \cite{M1}, Milnor constructs a conformal isomorphism $\Psi$ between $\C \setminus \mathcal{M}_1$ and the punctured disc. 
We now briefly review this construction, as it is at the heart of the new characterisation of $\chi$ on $\m\cap\K$.

Let $Q_{1/4}(z)=z^2+1/4$, let $\mathcal B$ denote the interior of its filled Julia set, in other words the parabolic basin of attraction of $z=1/2$, 
and let $\mathcal P_0 \subset \mathcal B$ be the largest attracting petal of $z=1/2$ such that the Fatou coordinate 
carries $\mathcal P_0$ diffeomorphically onto a right half-plane. Then the critical point $z=0$ belongs to the boundary of $\mathcal P_0$, 
and hence the critical value $z=1/4$ lies on the boundary of
$\mathcal P_{-1}= Q_{1/4}(\mathcal P_0)$. We normalise the Fatou coordinate to send  the critical value to $1$.
Let $\mathcal P_1= Q_{1/4}^{-1}(\mathcal P_0)$, then $\mathcal P_{-1}\subset \mathcal P_0\subset \mathcal P_1$, and recursively 
defining $\mathcal P_{k+1}=Q_{1/4}^{-1}(\mathcal P_k)$, 
we obtain that the parabolic basin $\mathcal B$ is the union (see Figure 5 on page 497 of \cite{M1})
$$\mathcal P_0\subset \mathcal P_1 \subset \ldots \subset \mathcal P_k \subset \ldots.$$

A model space for $\C\setminus \mathcal{M}_1$ is provided by the disc (punctured at the critical value $z=1/2$):
$$(\mathcal B \setminus \mathcal P_{-1})/\alpha$$
where $\alpha$ is the identification $z\sim \bar z$ on $\partial\mathcal P_{-1}$. To each $B$ in the `shift locus' $Per_1(1)\setminus\mathcal{M}_1$
Milnor associates a point in this model space by the following recipe.  
Let $\mathcal Q_{A,0}$ be the largest attracting petal of the parabolic fixed point of $P_A,\,B=1-A^2$,
such that the Fatou coordinate carries $\mathcal Q_{A,0}$ diffeomorphically onto a right half-plane. We normalise the Fatou coordinate to send the first critical value $z=-2+A$ of $P_A$ to $1$.
There exists a unique conformal map isomorphism
$\varphi_0:\mathcal P_0 \to \mathcal Q_{A,0}$ conjugating $Q_{1/4}|\mathcal P_0 $ to $P_A|\mathcal Q_{A,0}$ and sending $0$ to the first critical point $z=-1$ of $P_A$, namely the composition of the Fatou coordinate for $Q_{1/4}$ (normalised by sending $1/4$ to $1$) and the inverse Fatou coordinate for $P_A$ (normalised by sending the $A-2$ to $1$). Note that one can lift $\varphi_0$ back to
$\varphi_k:\mathcal P_k \to \mathcal Q_{A,k}:=P_A^{-k}(\mathcal Q_{A,0})$ up to and including the first value $k$ for which $\mathcal Q_{A,k}$ contains the second critical 
value $v_A=2+A$ of $P_A$.  In the second proof
of Theorem 4.2 in \cite{M1}, using a method suggested by Shishikura, Milnor shows that the assignment $B=1-A^2 \mapsto \varphi_k^{-1}(v_A)$ defines a conformal isomorphism $\Psi$ from $\C\setminus \mathcal \mathcal{M}_1$
to $(\mathcal B \setminus \mathcal P_{-1})/\alpha$.

\subsubsection{The map $\chi$ on $\K\setminus\m$}\label{ext}
Let $\Phi: \mathcal B \rightarrow \widehat \C \setminus \overline \D$ be the conformal conjugacy between $Q_{1/4}|_{\mathcal B}$ and $h|_{\widehat \C \setminus \overline \D}$, given by Fatou coordinates (where $h(z)$, as always, denotes $\frac{z^2+1/3}{z^2/3+1}$).
More specifically, let $\mathcal H_0 \subset \widehat \C \setminus \overline \D$ be the largest attracting petal of the parabolic fixed point $1$
of $h$ such that the Fatou coordinate (normalized to send the critical value $z=3$ of $h$ to the point $1$) carries $\mathcal H_0$ diffeomorphically onto a right half-plane.
%, so $\infty \in \partial \mathcal H_0$, where we normalize Fatou coordinates to send the first critical value $z=3$ of $h$ to the point $1$.
There exists a unique conformal isomorphism
$\phi_0:\mathcal H_0 \to \mathcal P_0$ conjugating $h|\mathcal H_0 $ to $Q_{1/4}|\mathcal P_0$ and sending the critical value $3$ of $h$ in $H_0$ to the critical value $1/4$ of $Q_{1/4}$ in $\mathcal P_0$ (namely the composition of Fatou coordinates for $h$ normalised by sending $3$ to $1$ and the inverse Fatou coordinates for $Q_{1/4}$ sending $1/4$ to $1$). As $h$ and $Q_{1/4}$ both have connected Julia sets, we can lift $\phi_0$ to the whole parabolic basin, obtaining a conformal conjugacy $\Phi: \mathcal B \rightarrow \widehat \C \setminus \overline \D$ between $Q_{1/4}|_{\mathcal B}$ and $h|_{\widehat \C \setminus \overline \D}$.

 For $a\in \K\setminus\m$, let $n(a)$ be the entry time of the critical value $v_a$ into the 
pinched fundamental annulus, that is $\F_a^{n(a)}(v_a)\in \P\A_a$, or equivalently  $v_a\in  \P\A_a^n=\F_a^{-n(a)}(\P\A_a)\cap O_a$.
Then we can write $\chi|_{\m\cap\K}$ as (see Figure \ref{extensionpicture}):
$$\chi: \K \setminus \m \rightarrow \C \setminus \mathcal{M}_1$$
 $$a \rightarrow\Psi^{-1}\circ \Phi^{-1}\circ T^{n(a)}_a(v_a).$$
      \begin{figure}
\centering
\psfrag{T}{\small $T_a$}
\psfrag{a}{\small $a \rightarrow v_a$}
\psfrag{M}{\small $\Psi^{-1}\circ \Phi^{-1}$}
\psfrag{p}{\small $f$}
\psfrag{g}{\small $\F_a$}
\psfrag{f}{\small $\F_{a_0}$}
\psfrag{t}{\small $\tau_a$}
	\includegraphics[width= 9cm]{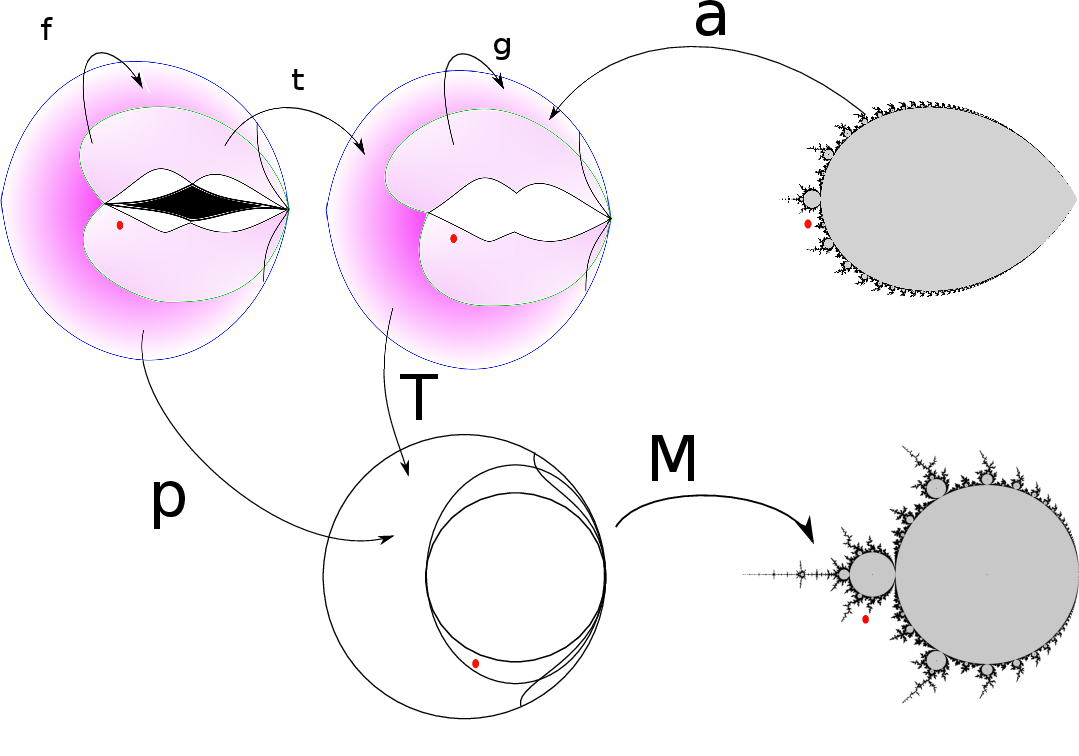}
\caption{\small The map $\chi: \K \setminus \m \rightarrow \C \setminus \mathcal{M}_1$.}
\label{extensionpicture}
\end{figure}

\begin{prop}\label{same}
With a suitable choice of 
%fundamental croissant for $h$ 
$V_h$ in the initial surgery on $\F_{a_0}$, the map
$\chi:\K\setminus\m \to Per_1(1)\setminus{\mathcal M}_1$ becomes the composition $$a \mapsto \Psi^{-1}\circ \Phi^{-1}\circ T^{n(a)}_a(v_a).$$
\end{prop}
\begin{proof}
In our surgery construction for $\F_{a_0}$ (and hence for $\F_a,\,a\in \K$), we defined $V_h$ to be the disc of radius 
$(1+\epsilon)$ tangent to the unit circle at the parabolic fixed point $1$, and $V_h'=h^{-1}(V_h)$. But we could equally well have defined it to be
$\mathcal H_0\subset \widehat \C \setminus \overline \D$ (coloured green in Figure \ref{tiles}), the maximal attracting petal for the parabolic fixed point $z=1$ of $h$, with boundary passing through the critical point, i.e. $\infty \in \partial \mathcal H_0$, and $\mathcal H_1=h^{-1}(\mathcal H_0)$. In this case, $O_h$ is the connected component of $\mathcal H_1 \setminus \gamma_h$ containing the unit disc, and $\P\A_h=V_h\setminus O_h$.
Then, when the map $F_{a_0}$ resulting from surgery to the correspondence $\F_{a_0}$
is straightened to become $P_{A_0}$, 
the pinched fundamental annulus $\P\A_{a_0}$ is identified with the pinched fundamental annulus $\P\A_h$, and so each tile 
$\mathcal Q_{P_0,k+1}\setminus \mathcal Q_{P_0,k}$ is identified with the corresponding $\mathcal H_{k+1}\setminus \mathcal H_k$, where $\mathcal H_k =h^{-k}(\mathcal H_0)$.

Let $\phi_a$ be the straightening map for $\F_a$ constructed in Section \ref{wf}, and let $\mathcal Q_{A,n}=P_A^{-n}(\mathcal Q_{A,0})$. Then for each fixed $a$ the map $\Phi^{-1}\circ T_a^{n(a)}\circ \phi_a^{-1}: \mathcal Q_{A,n(a)}\rightarrow \mathcal P_{n_a} \subset \mathcal B$ is conformal, as 
%it is quasiconformal and 
by construction 
 $(\Phi^{-1}\circ T_a^{n(a)}\circ \phi_a^{-1})^*\mu_0=\mu_0$. This composition conjugates $P_A|_{\mathcal Q_{A,n(a)}}$ to $Q_{1/4}|_{\mathcal P_{n(a)}}$, and it sends the first critical value $z=-2+A$ of $P_A$ to the critical value $1/4$ of $Q_{1/4}$. Therefore, by unicity, $\Phi^{-1}\circ T_a^{n(a)}\circ \phi_a^{-1}=\varphi_{n(a)}^{-1}$, where $\varphi_0:\mathcal P_0 \to \mathcal Q_{A,0}$ is the unique conformal isomorphism used by Milnor to define the conformal isomorphism $\Psi: \C \setminus M_1 \rightarrow (\mathcal B \setminus \mathcal P_{-1})/\alpha$ (see Section \ref{Mil}).
 Hence, for every $a \in \mathcal K \setminus \m$, 
$$ \Psi^{-1}\circ \Phi^{-1}\circ T^{n(a)}_a(v_a)=  \phi_a(v_a),$$
 and the map $\chi|_{\m\cap\K}$ can be written as $$\chi: \K \setminus \m \rightarrow \C \setminus \mathcal{M}_1$$
 $$a \rightarrow\Psi^{-1}\circ \Phi^{-1}\circ T^{n(a)}_a(v_a).$$
\end{proof}
 \begin{figure}[hbt!]
\centering
\includegraphics[width= 5cm]{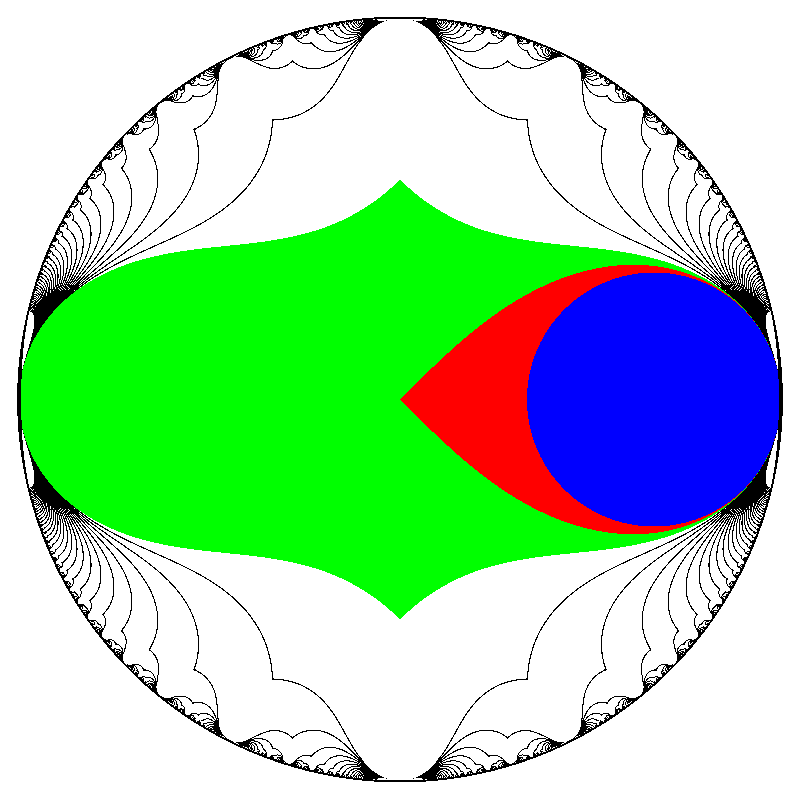}
\caption{\small Milnor tiling for $h$ on the unit disc (compare Fig. 5 in \cite{M1}). 
The petal $\mathcal P_{-1}$ is coloured blue, the set $\mathcal P_0 \setminus \mathcal P_{-1}$ is in red,
and $\mathcal P_1 \setminus \mathcal P_0$ is in green, so $\mathcal P_1$ is the coloured region. We take as 
$V_h\setminus h^{-1}(V_h)$
%fundamental croissant for $h$ 
the tile
$\mathcal P_1 \setminus \mathcal P_0$, or more accurately the corresponding tile in $\widehat \C \setminus \overline \D$.}
\label{tiles}
 \end{figure}

 \begin{prop}\label{extnqr}
 $\chi: \K \setminus \m \rightarrow \C \setminus \mathcal{M}_1$ is locally quasiregular.
 \end{prop}

\begin{proof}
As $\Phi$ and $\Psi$ are conformal, to prove that  $\chi: \K \setminus \m \rightarrow \C \setminus \mathcal{M}_1$ is locally quasiregular, 
it will suffice to show that $a \to T_a^{n(a)}(v_a)$ is locally quasiregular.  As the tubing is defined as the composition $T_a=f \circ {\tilde\tau}_a^{-1}$, where $f$ is quasiconformal, it is enough to show that the map $a \rightarrow \tau_{n(a)}^{-1}(v_a)$ is locally quasiregular. 

Consider the union of the holomorphic motions of $\P\A_{a_0}$ and its lifts as a map
$$\tau\cup\tau_1\cup\ldots\cup\tau_k:  \K \times (\P\A_{a_0}\cup\P\A^1_{a_0}\ldots \cup\P\A^k_{a_0}) \to \K\times \C$$
defined at the points where the various lifts exist. For each $j$ and $z\in \P\A_{a_0}^j$ let $\ell_z$ denote the image of $z$ under the lifted motion $\tau_j$. The image of the map 
$\tau\cup\tau_1\cup\ldots\cup\tau_k$ is foliated by the leaves $\ell_z$.

The set $\{(a_0,z): z \in V_{a_0}\setminus \Lambda_{a_0}\}$ is a curve transversal to this foliation (by definition of the holomorphic motion), and the set $\{(a,v_a):a\in \K\}$ is a holomorphic curve in the domain of the motion, so by Lemma 17.9 in \cite{Lyu} the holonomy along the leaves from the curve $\{(a,v_a):a\in \K\}$ to the transversal 
$\{(a_0,z): z \in V_{a_0}\setminus \Lambda_{a_0}\}$ is locally quasiregular. In the case that each leaf in the foliation is connected, the statement of the Proposition follows.

However, we have not yet proved that $\m$ is connected, so we must allow for the possibility that there exist points $a \in \K$ for which the leaf component through $(a,v_a)$ does not meet the transversal $\{(a_0,z): z \in V_{a_0}\setminus \Lambda_{a_0}\}$, and so the holonomy map is not defined directly. But in this case we can use equivariance to define it indirectly: we send $v_a$ to $\F_a^{n(a)}(v_a)\in \P\A_a$, we then move $\F_a^{n(a)}(v_a)$ to $\P\A_{a_0}\subset (V_{a_0}\setminus \Lambda_{a_0})$ by holonomy, and finally apply the appropriate branch of $\F_{a_0}^{-n(a)}$ within $V_{a_0}\setminus \Lambda_{a_0}$. The result now follows 
as before.
\end{proof}

\subsection{Continuity of $\chi$ on $\m\cap\K$ and holomorphicity on $\mathring \m$}\label{continuity}
We know from the previous section that $\chi$ is continuous on $\K\setminus \m$ (since it is locally quasiregular). 
We next prove that $\chi$ is continuous on $\partial \m\cap \K$ and holomorphic on $\mathring \m$, thus proving continuity everywhere in $\K$.
The proofs in this section follow those formulated by Douady and Hubbard \cite{DH}, and refined by Lyubich \cite{Lyu}, and subsequent authors: we present 
details to make this article as self-contained as possible.  %We will now show that 
%$\chi$ is continuous on $\mathring \m$ and on $\partial \m\cap \K$, thus proving that $\chi$ is continuous everywhere in $\K$.  
%We will start by proving continuity on $\partial \m\cap \K$. For this we shall make use of the decomposition described by Ma\~n\' e-Sad-Sullivan, which tells us that the 
%indifferent parameters live on the boundary $\partial \m$ (see Proposition \ref{neutparametervalues}). As indifferent parameters are mapped 
%under $\chi$ to indifferent parameters (because $\chi(a)=B$ implies that $\F_a|$ is conjugate to $P_A$), we 
%deduce, using precompactness of quasiconformal maps and rigidity on $\partial \mathcal{M}_1$, that $\partial \m\cap \K$ is mapped continuously under $\chi$ to $\partial \mathcal{M}_1$ (see Proposition \ref{conti}).
 %We will then prove holomorphicity of $\chi$ on both hyperbolic and queer components of $\mathring \m$  
%Proposition \ref{hypercom}
%and Proposition \ref{queer}), by using holomorphic motions.
%These proofs follow  a strategy employed by Douady and Hubbard \cite{DH}, Lyubich \cite{Lyu}, and subsequent authors.

\subsubsection{Indifferent periodic points}\label{mss}
We start by considering the Ma\~n\' e-Sad-Sullivan decomposition of the parameter space $\L\f'$.
Recall that for all $a \in \L\f'$, $\F_a|: V'_a \rightarrow V_a$ is a degree $2$ 
holomorphic map depending
holomorphically on the parameter, with a persistent parabolic fixed point at $P_a$. Recall also that by Corollary 1.2 in \cite{BL1}, 
the boundary of the backward limit set $\partial \Lambda_{a,-}$ is the closure of the set of repelling periodic points 
of $\F_a|$.

We define $\mathcal{I}$ to be the set of parameters in $\L\f'$ for which $\F_a$ has a non-persistent indifferent
periodic point, and set $\mathcal{R}= \L\f' \setminus \overline{\mathcal{I}}$. The set $\mathcal{R}$ is open, and there
 $\partial \Lambda_{a,-}$ moves holomorphically (these results follow from the implicit function theorem and the $\lambda$-Lemma, see \cite{MSS}).

\begin{prop}\label{neutparametervalues}
$\mathcal{R}= \L\f' \setminus \partial \m$.\end{prop}
\begin{proof}
As on $\mathcal{R}$ we can construct a holomorphic motion of the set $\partial \Lambda_{a,-}$ (using the implicit function theorem on the repelling periodic points and then the $\lambda$-Lemma, see \cite{MSS}),
and since we cannot map a connected set homeomorphically to a disconnected one and vice versa, $\mathcal{R}$ cannot intersect $\partial \m$, so 
$\mathcal{R} \subset \L\f' \setminus \partial \m.$
Therefore, to prove the statement, it is enough to prove that  $\L\f' \setminus \partial \m \subset\mathcal{R}$. 
It is easy to show that  $\L\f' \setminus \m \subset \mathcal{R}$: for all $a \in \L\f' \setminus \m$ the critical point $c_a$ of $\F_a|_{V_a'}$ is outside $\Lambda_{a,-}$, 
hence $\F_a|_{V'_a}$ cannot have an indifferent periodic point in addition to $P_a$, so $\mathcal{I} \cap (\L\f' \setminus
\m) = \emptyset$, which implies  $\L\f' \setminus \m \subset \mathcal{R}$ as $\L\f' \setminus \m$ is open.
Hence, it remains to prove that $\mathring \m \subset \mathcal{R}$, or equivalently (as $\mathring \m$ is open), that
 $\mathcal{I} \cap \mathring \m = \emptyset$, which we will do by contradiction. 
 
So, let us assume that there exists $a_0 \in \mathring{\m}$ for which $\F_{a_0}|_{V'_{a_0}}$ has an
indifferent periodic point $z_0$ of period $k$, and assume first $(\F^k_{a_0}|_{V'_{a_0}})'(z_0)\neq 1$.
It follows by the Implicit Function Theorem that there exist a neighbourhood $W(a_0)$ of $a_0$ in $\mathring{\m}$ and a neighbourhood $O(z_0)$ of $z_0$ 
where the cycle $\{z^1(a),...,z^k(a)\}$, its multiplier $\rho(a)= (\F^k_a|_{V'_{a_0}})'(z_a)$, and the critical point $c(a)$ move holomorphically with the
parameter, and $a_0$ is the unique parameter in $W(a_0)$ for which the cycle is indifferent with multiplier $\rho(a_0)$. 
Let $(a_n) \in W(a_0)$ be a sequence converging to $a_0$ such that, for all $n$, $|\rho(a_n)|<1$: then for every $n$ there exists
$z^i(a_n) \in \{z^1(a_n),\,...,z^{k}(a_n)\}$ such that $$\F^{i+kp}_{a_n}|_{V'_a}(c_{a_n}) \rightarrow z^i(a_n)\mbox{ as } p \rightarrow \infty$$
(we can assume $i$ independent of $a$ by choosing a subsequence).
On the other hand, the family
$$F_p(a)=\F^{i+kp}_a|_{ V_a'}(c_{a}),\,\,\,a \in W(a_0)$$
is a normal family (as it is analytic and bounded: bounded because $W(a_0) \subset \m$), so there exists a subsequence $F_{p_n}$ converging to some function
$h$, and since $h(a_n)=z^i(a_n)$ for all $n$, $h(a) = z^i(a)$
for all $a \in W(a_0)$.
This implies that $F_p(a) \rightarrow z^i(a)$ for all $a \in W(a_0)$, which is impossible as $W(a_0)$ contains parameters 
$a^*$ for which the cycle is repelling, hence 
$z^i(a^*)$ cannot
attract the sequence $F_p(a^*)$. 

If $(\F^k_{a_0}|_{V'_{a_0}})'(z_0)= 1$, let $U(a_0)$ be a neighbourhood of $a_0 \in \mathring{\m}$,
$a: W(0) \rightarrow U(a_0)$, $t \rightarrow t^2+a_0,$ be a branched covering of $U(a_0)$ branched at $0$
for some neighbourhood $W(0)$ of $0$, and repeat the previous argument.
\end{proof}

\subsubsection{Continuity on $\partial \m\cap \K$}
To prove that $\chi$ is continuous on $\partial \m\cap \K$ we will use precompactness of 
sequences of quasiconformal maps, in the topology of uniform convergence on compact subsets, and quasiconformal rigidity on $\partial \mathcal{M}_1$.

\begin{prop}\label{conti}
The map $\chi:\K \rightarrow \C $ is continuous at every point $a \in \partial \m\cap \K$,
and 
$\chi(\partial\m\cap \K)\subset\partial \mathcal{M}_1$.
\end{prop}
\begin{proof}
The map $\chi$ is continuous at $\hat a \in \partial \m$ if and only if given any sequence $a_n$ converging to $\hat a$ the sequence
$\chi(a_n)$ has a subsequence converging to $\chi(\hat a)$.
We will first prove that $\chi(\hat a) \in \partial \mathcal{M}_1$.
Let $a_n \rightarrow \hat a \in \partial \m$ with $a_n \in \mathcal{I}$ for all $n$. 
For every $n$, our surgery construction provides us with a quasiconformal conjugacy 
$\phi_{a_n}$ between $\F_{a_n}|$ and $P_{\chi(a_n)}$. By construction, for every $a$, the dilatation of the quasiconformal conjugacy $\phi_a$ 
is the dilatation of the tubing $T_a$ (see Remark \ref{dil});
hence by the second $\lambda$-Lemma this dilatation is locally bounded (see the second $\lambda$-Lemma in Chapter 17.4 in \cite{Lyu}).
So, by precompacteness, $\phi_{a_n}$ has a convergent subsequence, say  $\phi_{a_{n_k}} \rightarrow \widehat \phi$ as $k \rightarrow \infty$
and this qc homeomorphism $\widehat \phi$ conjugates $\F_{\hat a}|$ to a rational map $P_{\widehat A}\in Per_1(1)$.
Since $a_n \in \mathcal{I}$ for all $n$, we have that for every $n$ the map $P_{\chi(a_n)}$ has an indifferent periodic point, as $\F_{a_n}$ has. 
Hence for every $n$, $\chi(a_n)\in \partial \mathcal{M}_1$, and so  
its limit too: $\widehat B \in \partial \mathcal{M}_1,\,\,\widehat B=1-\widehat{A}^2$.
Therefore, $\F_{\hat a}$ is quasiconformally conjugate to both $P_{\chi(\hat a)}$ and to $P_{\widehat A}$,  and since we have \textit{rigidity} on  $\partial \mathcal{M}_1$, 
that is, quasiconformal conjugacy implies conformal conjugacy (see \cite{L2}, Proposition 4.2), we obtain $\chi(\hat a)= \widehat B$.
This shows that $\chi(\hat a) \in \partial \mathcal{M}_1$, and more generally that $\chi(\partial\m)\subset \partial{\mathcal M}_1$.

Let us now show continuity. Let $(a_n)$ be a sequence in $\K$ converging to $\hat a \in \partial \m$. As we saw above, for all $n$ the surgery construction provides us 
with a quasiconformal conjugacy 
$\phi_{a_n}$ between $\F_{a_n}|$ and $P_{\chi(a_n)}$, which has locally bounded dilatation, and hence the sequence $\phi_n$ has a converging subsequence,  say  $\phi_{a_{n_k}} \rightarrow \widehat \phi$ as $k \rightarrow \infty$, and $\widehat \phi$ conjugates 
$\F_{\hat a}$ to $P_{\widehat A}= \widehat \phi \circ \F_{\hat a} \circ \widehat \phi^{-1}$. Hence $P_{\chi(\hat a)}$ and $P_{\widehat A}$ are quasiconformally conjugate, 
and as we showed above that $\chi(\hat a) \in \partial \mathcal{M}_1$, by rigidity we have that  $\chi(\hat a)= \widehat B$. 
\end{proof}

\subsubsection{Holomorphicity on $\mathring \m$}
Let $W$ be a connected component of the interior $\mathring \m$ of $\m$. 
We call $W$ \textit{hyperbolic} if there exists (at least) one parameter $a_0 \in W$ for which $\F_{a_0 |}$ has an attracting cycle.
On the other hand, we call $W$ \textit{queer} if there is no parameter $a_0 \in W$ for which $\F_{a_0 |}$ has an attracting cycle.
We will denote hyperbolic components by $H$, and queer components by $Q$, and we will prove that 
 $\chi:\mathring  \m \cap \K \rightarrow \mathcal{M}_1$ is holomorphic first on hyperbolic components (Proposition \ref{hypercom}) 
 and then on queer components (Proposition \ref{queer}).

\begin{prop}\label{hypercom}
 On hyperbolic components, the map $\chi$ is holomorphic and proper.
\end{prop}
 \begin{proof}
 Let $H$ be a hyperbolic component of $\mathring \m$. We are going to prove that 
there exists a hyperbolic component $C \subset \mathring{\mathcal{M}_1}$
such that $\chi|_{H}: H \rightarrow C$ is holomorphic. Since $H$ is a hyperbolic component,
there exists $a_0 \in H$ for which $\F_{a_0}$
has an attracting cycle. 
Since $\F_a$ has no indifferent parameters in $H$ (see Prop. \ref{neutparametervalues}),
by the Implicit Function Theorem for all $a \in H$, $\F_a$ has an attracting cycle. 
Therefore, for all $a \in H$ the hybrid conjugate member of $Per_1(1)$
has an attracting cycle, and hence $C$ is a hyperbolic component of $\mathring{\mathcal{M}_1}$.
Since the attracting cycle with its basin belongs to the filled Julia set $K\ll$ of $P_A$, and the hybrid conjugacy is conformal on the interior 
of the limit set $\Lambda_{a,-}$ ($K\ll$ is the hybrid image of $\Lambda_{a,-}$), the multipliers of the conjugated attracting cycles must coincide. 
Hence, denoting by  $\rho_{H}(a)$ the multiplier map for $\F_a$ on $H$, and
denoting by $\rho_C$ the multiplier map for the family $Per_1(1)$ on $C$, we have that
$$\rho_H(a)= \rho_C(A\ll),$$
and we can write the restriction of the map $\chi$ on $H$ as
$$\chi|_{H}=\rho_C^{-1} \circ \rho_H.$$
Since $\rho_{H}(a)$ is holomorphic, and is holomorphic with degree $1$ (see \cite{PT}), the composition
$\chi|_{H}$ is holomorphic.

By Proposition \ref{conti}, $\chi: H\rightarrow C$ extends continuously to the boundary, and $\chi(\partial H)\subset \partial C$, so that the map $\chi|H$ is proper.
\end{proof}

Note that, if  $a \in \mathring{\m}$ and $\F_a$ does not have an attracting cycle, then $\Lambda_{a,-}$ is connected 
with empty interior. Indeed,
$\F_a$ does not have an attracting nor an indifferent cycle (by the assumption and Proposition \ref{neutparametervalues}), 
and so the hybrid equivalent $P_A$ has no attracting or indifferent cycle. 
 Hence, $K\ll$ is connected with empty interior, and so is $\Lambda_{a,-}$, as these sets are quasiconformally homeomorphic
 (since $\F_a$ is hybrid conjugate to $P_A$ on a doubly pinched neighbourhood of
$\Lambda_{a,-}$). Hence, if $Q$ is a queer component of $\mathring{\m}$, for all $a \in Q$, the set $\Lambda_{a,-}$ is connected and has no interior.

\begin{prop}\label{queer}
 On queer components the map $\chi$ is holomorphic and proper.
\end{prop}

\begin{proof}

 Assume $Q$ is a queer component of $\m$, and take the base point $a_0 \in Q$. 
 Since for every $a \in Q$ the critical point $c_a$ belongs to $\Lambda_{a,-}$, the holomorphic motion on $Q$ extends to
 $$ \tau_a: Q \times (V_{a_0}\setminus \Lambda_{a_0,-}) \rightarrow V_a\setminus  \Lambda_{a,-},$$
 since $\partial V_{a_0}$ and $\partial V_a$ are quasicircles, it extends to
 $$ \tau_a: Q \times (\widehat \C\setminus \Lambda_{a_0,-}) \rightarrow \widehat \C\setminus  \Lambda_{a,-},$$
and since on $Q$ the limit set $\Lambda_{a,-} $ is nowhere dense, by the $\lambda$-Lemma
 we obtain a dynamical holomorphic motion $$\tau_a: Q \times \widehat \C\rightarrow \widehat \C.$$
 The tubing $T_a=f\circ \tau_a^{-1}$, where $f:\overline \C \setminus O_{a_0}\rightarrow \C \setminus \overline \D$ is quasiconformal (see Section \ref{s1}), also extends to (see Section \ref{ehm})
 $$\overline T_a: \widehat \C \setminus \Lambda_{a,-} \rightarrow \widehat \C \setminus \overline \D.$$
 Let $\phi_{a_0}$ be the quasiconformal conjugacy between 
  $$ F_{a_0} =\left\{
\begin{array}{cl}
\F_{a_0} &\mbox{on  } O_{a_0} \\
f^{-1}\circ h\circ f &\mbox{on  } \widehat \C \setminus O_{a}\\
\end{array}\right.$$
  and $P_{A_0}(z)=z+1/z+A_0$ constructed in Section \ref{s1}, so in particular  $\phi_{a_0}$ is a quasiconformal conjugacy between $\F_{a_0}$ and $P_{A_0}$ on $O_{a_0}$.
 Hence,   for all $a \in Q$, the map 
$$\Phi_a:=\tau_a \circ \phi_{a_0}: \widehat \C \rightarrow \widehat \C$$
is a quasiconformal conjugacy between $P_{A_0}$ and 
$$ F_{a} =\left\{
\begin{array}{cl}
\F_{a} &\mbox{on  } O_a \\
(\overline T_a)^{-1}\circ h\circ \overline T_a &\mbox{on  } \widehat \C \setminus O_{a}\\
\end{array}\right.$$
on $\widehat \C$, and between $P_{A_0}$ on $\Phi_a^{-1}(O_a)$ and $\F_a$ on $O_a$. So, in particular it is a quasiconformal conjugacy between $P_{A_0}$ on $K_0=K_{A_0}$ and 
$\F_a$ on $\Lambda_{a,-}$.
As $a$ belongs to a queer component of  $\mathring\m$,  $A_0$ must belong to a queer component of $\mathring{\mathcal{M}}_1$, hence $K_0=J_{P_{A_0}}$, which by Proposition 4.4 in \cite{L2} has positive area.

Let $\nu\ll$ be the family of Beltrami forms on $\widehat \C$ defined as follows:
$$\nu\ll(z):= \left\{
\begin{array}{cl}
(\Phi_a)^* \mu_0 &\mbox{on } K_0 \\
\mu_0 &\mbox{on } \widehat \C \setminus K_0\\
\end{array}\right.
$$
Then, by construction, $area(supp(\nu\ll))>0$, $\nu\ll$ depends holomorphically on $a$, and it is invariant under $P_{A\l0}$.
Let $\psi\ll: \widehat \C \rightarrow \widehat \C$ be the family of integrating maps fixing
$-1,\,1$ and $\infty$. Then the family $P_{A(a)}= \psi\ll \circ P_{A\l0} \circ (\psi\ll)^{-1}$ consists of holomorphic maps of degree $2$ acting on $\widehat \C$ and with a persistent parabolic fixed point 
at $\infty$ and critical points at $\pm 1$, and hence it
has the form $P_{A(a)} (z)= z + 1/z + A(a)$, where $ A(a)$ depends holomorphically on the parameter.

For every $a \in Q$, the map $\psi\ll\circ \phi_{a_0} \circ \tau_a^{-1}$ is a hybrid
conjugacy between $\F_a|$ and $P_{A(a)}$ on 
a doubly pinched neighbourhood of $\Lambda_{a,-}$, hence by unicity,
$P_{A\ll}$ is the member of $Per_1(1)$ hybrid equivalent to $\F_a|$. Hence the map $\chi|_{Q}$ is holomorphic.

As the map $\chi:\m \cap \K\rightarrow \mathcal{M}_1$ is injective on $\m$ (see Proposition \ref{inj}), the restriction $\chi|_{Q}$ cannot be constant, 
and as it extends continuously to the boundary and the extension sends boundaries to boundaries (see Proposition \ref{conti}), the map $\chi|Q$ is proper.
\end{proof}

\subsection{The straightening map $\chi$ is a homeomorphism on $\K$}\label{hom}
So far we know that $\chi$ is well-defined and continuous everywhere on the (open) doubly truncated lune $\K$, that it is injective on $\m$,
and that it is quasiregular on $\K\setminus\m$. We will use elementary degree theory from algebraic topology to deduce that 
$\chi$ is injective on $\K$, hence a homeomorphism from $\K$ onto its image.

\begin{prop}\label{homeo_on_lune}
$\chi$ is a homeomorphism from $\K$ onto $\chi(\K)\subset Per_1(1)$.
\end{prop}

\begin{proof}
First observe that $\chi$ is proper, since it extends continuously to $\partial K$ (we could have taken a larger closed subset $K$ of
$\L\f'$ when we were defining our extension to $\chi$ in the first place).
Hence $\chi$ induces a homomorphism of cohomology with compact supports:
$$\chi^*: H^2_c(\C) \to H^2_c(\K).$$
Since $\K$ and $\C$ are path-connected surfaces, these cohomology groups are both
isomorphic to $\Z$, generated by the respective fundamental classes. The element
$1\in \Z =H^2_c(\C)$ is mapped to a non-zero $d\in \Z=H^2(\K)$ known as the {\textit degree},
$d(\chi)$, of $\chi$. Since we have proved that $\chi$ is injective on 
$\m$, and $\mathring\m\cap\K$ is a non-empty open subset of $\K$,
we know that $d(\chi)=1$.

But $\chi:(\K\setminus\m) \to \chi(\K\setminus\m)$ is locally quasiregular,
so it is a branched covering, and a branched covering of degree $1$ is a homeomorphism. Thus
$\chi$ is injective on $\K\setminus\m$ as well as on $\m$, and, as their images are disjoint,
$\chi$ is therefore injective on their union $\K$. As $\chi$ is continuous on $\K$ the result follows. \end{proof}

\begin{remark} We do not know if there exists a continuous extension of 
$\chi:\m\setminus\{7\} \to \mathcal{M}_1\setminus\{1\}$ to the whole of $\L\f'$, not just to $\K$, but such an
extension would be a homeomorphism from the whole of $\L\f'$ onto its image $\chi(\L\f'$).
%We defined our map $\chi$ on $\K$, for $K$ a compact subset of
%$\L\f'\setminus N$, where $N$ is a neighbourhood of $a=7$. If there exists a continuous extension
%of $\chi:\m\setminus\{7\} \to \mathcal{M}_1\setminus\{1\}$ to the whole of $\L\f'$,
%then this extension is a homeomorphism from the whole of $\L\f'$ onto its image $\chi(\L\f'$).}
\end{remark}

\subsection{Every $B\in \mathcal{M}_1\setminus\{1\}$ is in $\chi(\K)$ for $K=\L\f'\setminus N$ with $N$ sufficiently small}\label{chi_surj}
We remind the reader that our map $\chi$ is only defined on the \textit{doubly}-truncated lune $K=\L\f'\setminus N$, and that this 
definition of $\chi$ is dependent on $K$, 
and hence on $N$, but that on $\m\setminus\{7\}\subset \L\f'$ the definition is independent of $N$.
We start the section with a lemma that we could have stated and proved earlier, but which is now an easy consequence of Proposition \ref{homeo_on_lune}.
We continue by showing that we can decompose $\m$ into a main component and rational limbs (Proposition \ref{limb_properties}), and we prove surjectivity limb by limb (Proposition \ref{surj_prop}).

\begin{lemma}\label{connected}
$\m\setminus\{7\}$ is connected, and so is $\m$.
\end{lemma}
\begin{proof}
If $\m\setminus\{7\}$ is not connected then there exists a point $a\in \m$ and loop $u$ in $\L\f'\setminus \m$ winding once around $a$. This loop lies in 
$\K=L\f'\setminus N$ for any sufficiently small neighbourhood $N$, and since $\chi$ is a homeomorphism the loop
$\chi(u)$ in $\chi(\K\setminus \m)$ winds once around $\chi(a)\in \mathcal{M}_1\setminus\{1\}$. But there can be no such loop since the complement 
of $\mathcal{M}_1$ in $\hat\C$ is simply-connected \cite{M1}. So $\m\setminus\{7\}$ is connected. That $\m$ is also connected follows since the point $a=7$ 
is in the closure of $\m\setminus\{7\}$.
\end{proof}

We next recall from Theorem 1 in \cite{BL2} that for every $a\in \m$ for which the alpha-fixed-point of $\F_a$ is repelling, 
this fixed point has a well-defined {\it combinatorial rotation number} $\rho_a$ which is always a non-zero rational. 
Define the {\it $p/q$-limb} $L_{p/q}$ of $\m$ to be 
the set of $a\in \m$ for which the $\alpha$-fixed point of $\F_a$ is repelling, and has rotation number $\rho_a=p/q$, together with the parameter value $a_{p/q}$ 
at which the derivative at the fixed point 
is $e^{2\pi i p/q}$ (this additional point is called the {\it root} of $L_{p/q}$). Our reason for choosing this definition of $L_{p/q}$, rather than defining it as the 
part of $\m$ which lies between the two external parameter rays which land at $a_{p/q}$, is simply to avoid a diversion into proving that these parameter rays 
are well defined and do indeed land.

\begin{prop}\label{limb_properties}

(i) $\m$ is the disjoint union of the closure of the main component of $\mathring\m$ and the sets $L_{p/q}\setminus\{a_{p/q}\}$, $p/q\in \Q$, $0<p/q<1$;

(ii) for each $p/q$ the intersection of $L_{p/q}$ with the closure of the main component of 
$\mathring\m$ is the single point $\{a_{p/q}\}$; 

(iii) for each $p/q$ the limb $L_{p/q}$ is closed and connected.
\end{prop}
\begin{proof}
Parts (i) and (ii) follow from the fact that the main component of $\mathring\m$ is the set of parameter values for which the alpha-fixed-point of $\F_a$ 
is an attractor (the analytic expression of the derivative of $\F_a$ at its alpha-fixed-point shows we have a unique main component, see Section 7 in \cite{BL2}), the fact that the derivative of $\F_a$ at its alpha-fixed-point is a holomorphic function of $a$, and the fact that $\m$ does not have irrational limbs (see \cite{BL2}). Part (iii) follows from parts (i) and (ii)
since $\m$ is connected (Lemma \ref{connected} above).
\end{proof}

To prove the main results of this subsection and the next, we shall need to use properties of the parabolic Mandelbrot set ${\mathcal M}_1$ 
analogous to properties of $\m$ and the classical Mandelbrot set $\mathcal M$. We shall refer to \cite{PR} for proofs of these: alternatively they may be proved by the methods of \cite{BL2}.
The first concerns the existence of a rational combinatorial rotation number for a repelling alpha-fixed-point.

\begin{prop}\label{landing_M1}
Every repelling fixed point of a member of the family $P_A$, $B\in \mathcal{M}_1$ (where $B=1-A^2$), is the landing point of a periodic ray.
\end{prop}
\begin{proof}
\cite{PR}, Theorem 3.7.
\end{proof}

We can now define the {\it $p/q$-limb} $L'_{p/q}$ of $\mathcal{M}_1$ to be 
the set of $B\in \mathcal{M}_1$ ($B=1-A^2$), for which the combinatorial rotation number of the fixed point of $P_A$ is $p/q$, together with the parameter value 
$B_{p/q}$ for which the derivative at the fixed point of $P_A$
is $e^{2\pi i p/q}$. 

\begin{cor}\label{M1_limb_properties}

(i) $\mathcal{M}_1$ is the disjoint union of the closure of the main component of $\mathring{\mathcal{M}}_1$ and the sets $L'_{p/q}\setminus\{a_{p/q}\}$, $p/q\in \Q$, $0<p/q<1$;

(ii) for each $p/q$ the intersection of $L'_{p/q}$ with the closure of the main component of 
$\mathcal{M}_1$ is the single point $\{B_{p/q}\}$; 

(iii) for each $p/q$ the limb $L'_{p/q}$ is closed and connected.
\end{cor}
\begin{proof}
\cite{PR}, Section 3.2.3 and in particular Theorem 3.14.
\end{proof}

\begin{lemma}\label{enc}
For each integer $n>1$ there exists a loop $u_n\subset\L\f'$ enclosing $\bigcup\{L_{p/q}:1/n \le p/q \le (n-1)/n\}$. 
These loops can be chosen in such a way that every point of $\m\setminus\{7\}$ is enclosed by $u_n$ for some $n$.
\end{lemma}

\begin{proof}
This would be straightforward if we had established the existence of parameter space external rays landing at the roots of limbs, but we shall sidestep 
this question and rely instead on the fact that since $\hat\C\setminus \m$ is a simply-connected open set there exists a dense set of points in 
$\partial \m$ which are accessible  from the complement of $\m$ in $\C$. First we 
observe that there exists a neighbourhood $N_{n+1}$ of $a_{1/(n+1)}$ (the root point of $L_{1/(n+1)}$), in $\L\f'$, such that $N_{n+1}$ 
does not meet any $L_{p/q}$ having $1/n \le p/q \le 1-1/n$. This 
follows from the restrictions that the Yoccoz inequality proved in \cite{BL2} puts on the derivative at the fixed point of $\F_a$ for $a\in L_{p/q}$ 
(see Theorem 2 and 
Figure 3 in \cite{BL2}), together with the continuity of this derivative with respect to $a$. Now let $a'_{n+1}$ be some point of 
$N_{n+1}\cap\partial\m$ accessible from outside $\m$, and define a loop $u_n$ by joining $a_{1/(n+1)}$ to $a'_{1/(n+1)}$ by any path within
$N_{n+1}$, then from $a'_{1/(n+1)}$ to its complex conjugate $a'_{n/(n+1)}$ by any path in $\L\f'\setminus\m$, then from $a'_{n/(n+1)}$ to 
$a_{n/(n+1)}$ in the obvious way, and finally from $a_{n/(n+1)}$ to $a_{1/(n+1)}$ by any path within the main component of $\mathring\m$, 
noting that we may choose these last paths 
in such a way that every point of the main component is enclosed by at least one of the loops $u_n$. Since $u_n$ encloses all the
roots $\{a_{p/q}: 1/n \le p/q \le (n-1)/n\}$, it must also enclose the corresponding limbs.
\end{proof}

\begin{prop}\label{surj_prop}
Every $B\in \mathcal{M}_1\setminus\{1\}$ is in $\chi(\K)$ for $K=\L\f'\setminus N$ with $N=\D(7,r)$ for $r$ sufficiently small.
\end{prop}

\begin{proof}
For a contradiction, suppose that there exists $B\in \mathcal{M}_1\setminus\{1\}$ and a sequence of positive real numbers $r_m$
converging to zero such that for every $K=\L\f'\setminus \D(7,r_m)$ we have $B\notin \chi(\m\cap K)$. 
We prove that the existence of any $B\in \mathcal{M}_1\setminus\chi(\m)$ leads to a contradiction. Let $B$ be such a parameter value. 
The main component of $\mathring{\mathcal M}_1$ is the unit disc and the main component of $\mathring{\m}$ is also a Jordan disc (bounded by a simple closed curve the equation of 
which can be readily computed from the information in Section 7 of \cite{BL2}, in particular Remark 4). So our homeomorphism $\chi$ between these components extends to a homeomorphism
between their closures.
Thus $B$ cannot lie in the closure of the main component of the interior of ${\mathcal M}_1$ and must therefore lie in some limb $L'_{p/q}$ of $\mathcal{M}_1$.
Choose $1/n<p/q$ and let $D_n$ denote the open topological disc in $\L\f'$ bounded by the Jordan curve $u_n$ constructed in the preceding lemma.
Let $r_m$ be small enough that $\D(7,r_m)$ does not intersect $D_n$, and let $K=\L\f'\setminus \D(7,r_m)$, so $D_n\subset K$. Now $B$ cannot lie in $\chi(D_n)$, 
since if it did then as $P_A$ (for $B=1-A^2)$ has connected filled Julia set and $\chi$ is defined by a surgery preserving filled Julia sets, we would have $B=\chi(a)$ for some $a\in \m$.
Hence $B\notin \chi(D_n)$, but now $\chi(D_n \setminus L_{p/q})$ disconnects the limb $L'_{p/q}$ of ${\mathcal M}_1$, contradicting Corollary \ref{M1_limb_properties}(iii).
\end{proof}

Since $\chi(\m)\subset \mathcal{M}_1$ and $\chi(\K\setminus\m)\subset Per_1(1)\setminus \mathcal{M}_1$, we deduce from 
Propositions \ref{homeo_on_lune}  and \ref{surj_prop}:
\begin{cor}\label{almosthomeo}
$\chi:\m\setminus\{7\}\to \mathcal{M}_1\setminus\{1\}$ is a homeomorphism. \qed
\end{cor}

\subsection{Completing the proof of the Main Theorem: $\m$ is homeomorphic to $\mathcal{M}_1$}\label{final}

We use a further property of $\mathcal{M}_1$ analogous  to properties of $\m$ and $\mathcal{M}$:

\begin{prop}\label{M1limbs} 
The limbs $L'_{p/q}$ of $\mathcal M_1$ have diameters which converge to $0$ as $p/q$ converges to $0$, and attaching points $B_{p/q}$ which converge to $B=1$
as $p/q$ converges to $0$.
\end{prop}
\begin{proof}
\cite{PR}, Corollary 3.23. The fact that the roots of the limbs tend to $B=1$ is an immediate consequence of the continuity at $B=1$ of the derivative of $P_A$ at its alpha-fixed-point.
\end{proof}

\begin{cor}\label{chi_on_roots}
The map $\chi:\m\to \mathcal{M}_1$ is a homeomorphism.
\end{cor}

\begin{proof}
We have already established that $\chi$ is a homeomorphism between $\m\setminus\{7\}$ and $\mathcal{M}_1\setminus\{1\}$. Setting $\chi(7)=1$ extends 
$\chi$ restricted to the main components of $\mathring \m$ and $\mathring {\mathcal M}_1$ to a homeomorphism between these components
union the points $a=7$ and $B=1$ respectively. It remains to consider sequences in the limbs of $\m$ and ${\mathcal M}_1$ converging to $a=7$ and
$B=1$ respectively. But when $p/q$ 
converges to zero, the diameters of the $p/q$-limbs converge to zero, and their roots converge to $a=7$ and $B=1$ respectively, so we deduce that 
$\chi$ is a homeomorphism from $\m$ to $\mathcal{M}_1$.
\end{proof}

This completes the proof of the Main Theorem, since we have already proved that $\chi$ has properties (i), (ii) and (iii) of the statement of the theorem.
Property (i) is Proposition \ref{welldefined}, property (ii) is Propositions \ref{hypercom} and \ref{queer}, and property (iii) is Proposition \ref{homeo_on_lune}.

\section{Proof of the Corollary to the Main Theorem}\label{proofcor}
In \cite{BL1} we defined $\mathcal{C}_{\Gamma}$ to be the connectedness locus for the family 
$\F_a$, where we let $a$ vary over the whole of the Klein combination locus $\mathcal{K} \subset \C$.
We defined the modular Mandelbrot set $\m$ to be
the intersection of ${\mathcal C}_\Gamma$ with the closed disc of centre $4$ and radius $3$ (which we know to be contained in 
${\mathcal K}$ since for $a \in {\overline \D(4,3)}$ the standard fundamental domains $\Delta^{st}_{Cov},\Delta^{st}_J$ are 
a Klein combination pair):
$$\m:= \mathcal{C}_{\Gamma} \cap {\overline \D(4,3)}.$$
We can now prove that $\m$ is the whole of $\mathcal{C}_{\Gamma}$ (in other words 
$\mathcal{C}_{\Gamma}\subset {\overline \D(4,3)}$).

\begin{mcor}\label{5}
 The modular Mandelbrot set $\m$ is the whole connectedness locus $\mathcal{C}_{\Gamma}$ of the family $\F_a$.
\end{mcor}

\begin{proof}
 Suppose there exists $a \in \mathcal{C}_{\Gamma} \setminus \m$.
 Then the corresponding $\F_a$ is a mating between the modular group and a rational map $P_A(z)= z+1/z + A,\,\,A \in \C$ (see the Main Theorem in \cite{BL1}). Since $a \in \mathcal{C}_{\Gamma}$, the correspondence $\F_a$ has connected limit set, therefore $P_A$ has connected Julia set, 
 and so $B=1-A^2 \in \mathcal{M}_1$. Moreover $\F_a$ is hybrid conjugate to $P_A$ on doubly pinched neighbourhoods of $\Lambda_-(\F_a)$ and $K(P_A)$ respectively, by the Main Theorem in \cite{BL1}.
 Since $\chi: \m \rightarrow \mathcal{M}_1$ is a homeomorphism, there exists $a' \in \m$ such that $\chi(a')=B$, and $\F_{a'}$ is hybrid equivalent to
 $P_A$ on doubly pinched neighbourhoods of $\Lambda_-(\F_a)$ and $K(P_A)$ respectively. But then $\F_a$ is hybrid equivalent to $\F_{a'}$ on quadruply pinched neighbourhoods of $\Lambda_a=\Lambda_{a,-}\cup \Lambda_{a,+}$ and $\Lambda_{a'}=\Lambda_{a',-}\cup \Lambda_{a',+}$ respectively by Proposition \ref{doublypinchedconjugacy}, 
and the proof of Proposition \ref{inj} shows that $\F_a$ is conformally conjugate to $\F_{a'}$, and hence that $a=a'$ by Lemma \ref{unique}.
\end{proof}

\appendix
\section{Appendix: Dynamical Lunes}
\medskip
Recall that in the parameter plane ${\mathcal L}_\theta$ denotes the open lune bounded by the two arcs of circles passing through 
$a=1$ and $a=7$ which at $a=1$ have tangents at angles $\pm \theta$ to the positive real axis. Analogously,
in the dynamical plane, coordinatised by $Z$, we let $L_a$ denote the open lune bounded by the two 
arcs of circles passing through 
$Z=1$ and $Z=a$ which at $Z=1$ have tangents at angles $\pm \theta$ to the positive real axis. Recall that $Z=1$ is the persistent 
parabolic fixed point, and that our coordinate-free notation for this point is $P_a$.
In this Appendix we 
prove: 

\begin{prop} \label{dynamic-lune}
Given any $\theta$ in the range $\pi/3\le \theta \le \pi/2$, for every parameter value $a \in {\mathcal L}_\theta\cup \{7\}$ the image
${\mathcal F}_a (\overline{L}_a)$ of  $\overline{L}_a$ is contained in $L_a \cup\{P_a\}$.
\end{prop}

\medskip
Note that here ${\mathcal F}_a: \overline{L}_a \to {\mathcal F}_a (\overline{L}_a)$ is a $1$-to-$2$ correspondence. 
Note also that in the coordinate $z'=(a-1)(Z-1)/(a-Z)$, the points $Z=1$ and $Z=a$ become 
$z'=0$ and $z'=\infty$ respectively and that as the coordinate change from $Z$ to $z'$ has 
derivative $1$ at $Z=1$, 
the boundary of the lune $L_a$ is carried to a pair of fixed straight lines through the origin in the $z'$-plane 
at angles $\pm\theta$ to the positive real axis. Thus in the coordinate $z=(Z-1)/(a-Z)$ the lune $L_a$ is no longer independent 
of $a$, but the points in it move holomorphically with $a$.

\medskip
Consider $J(L_a)$, and observe that since $J:z\leftrightarrow -z$ in the $z$-coordinate, 
$J(L_a)=-L_a$ in this coordinate. 
We remark that $J(L_a)$ is contained in the 
{\it standard fundamental domain} $\Delta_J$ for $J$, which is defined in the $Z$-coordinate \cite{BL1} 
as the complement
of the round disc which has centre on the real $Z$-axis and boundary circle passing 
through the points $Z=1$ and $Z=a$. The following is an immediate consequence of 
Proposition \ref{dynamic-lune}:

\begin{cor}\label{universal-lune}
For $\pi/3\le \theta\le \pi/2$,  

$a\in {\mathcal L}_\theta \cup \{7\}  \Rightarrow 
{\mathcal F}_a^{-1}(\overline{J(L_a)})\subset J(L_a)\cup\{P_a\}$, and hence 

(i) $\Lambda_{-,a} \subset J(L_a)\cup\{P_a\}$;

(ii) $\Lambda_{-,a} \cap \partial J(L_a)=\{P_a\}.$

 \end{cor}
%\textcolor{blue}{As in \cite{BL2} we proved that $\m \setminus \{1\} \subset \L\f$, this imply that, for all $a$ in the pinched neighborhood $\L\f$ of $\m$, the holomorphically moving dynamical lune $J(L_a)%$ contains $\Lambda_{-,a}\setminus \{P_a\}$}.
In \cite{BL2} we proved that there exists $\theta$ in the half-open interval $[\pi/3, \pi/2)$ such that $\L\f$ is a neighbourhood of $\m\setminus\{7\}$, pinched at $a=7$. 
Corollary \ref{universal-lune} tells us that for all $a$ in this pinched neighbourhood, the holomorphically moving dynamical lune $J(L_a)$ contains $\Lambda_{-,a}\setminus \{P_a\}$.

\subsubsection{Proof of Proposition \ref{dynamic-lune}}

Working in the $Z$-coordinate, let $C_d$ denote the circle in the dynamical plane which passes through the 
points $Z=1$ and $Z=7$ and which has 
centre at the point $4-di$ (where $d$ is real). As usual let $Q(Z)=Z^3-3Z$.

\begin{lemma}\label{intersections}
For every $-\sqrt{3} \le d \le \sqrt{3}$, the circle $C_d$ meets its image under $Cov_0^Q$ uniquely at the point $Z=1$.
\end{lemma}

Assuming this lemma (proved below), the proof of Proposition \ref{dynamic-lune} proceeds as follows. 
Let $\pi/3\le\theta\le \pi/2$.
Setting $d=3\cot{\theta}$, we note that the upper boundary arc of the parameter space lune ${\mathcal L}_\theta$ is the 
upper arc of $C_d$ from $a=1$ to $a=7$, and the lower boundary arc of ${\mathcal L}_\theta$ is the lower arc of $C_{-d}$ 
between these two points. 

Now let  $a \in {\mathcal L}_\theta\cup \{7\}$, and let $C_{a,\theta}$ denote the circle in dynamical space through $Z=1$ and 
$Z=a$ which has tangent at angle $\theta$ to the real axis at $Z=1$. The upper boundary of the dynamical 
space lune $L_a$ is the upper arc of $C_{a,\theta}$ from $Z=1$ to $Z=a$ and the lower boundary is the lower 
arc of $C_{a,-\theta}$ between these two points. If we let $D_d$ denote the open disc bounded by $C_d$, and $D_{a,\theta}$ 
denote the open disc bounded by $C_{a,\theta}$, we see that

(i) $D_{a,\theta}\subseteq D_d$ and $D_{a,-\theta}\subseteq D_{-d}$;

(ii) $L_a=D_{a,\theta}\cap D_{a,-\theta}$.

But from Lemma \ref{intersections}  we know that $Cov^Q_0(\overline{D}_d)$ meets $\overline{D}_d$ at the single point $Z=1$,
and since the involution $J$ sends the exterior of $D_a$ to its interior, we deduce from (i) that 
${\mathcal F}_a(\overline{D}_{a,\theta})\subset D_{a,\theta}\cup\{1\}$. Similarly 
${\mathcal F}_a(\overline{D}_{a,-\theta})\subset D_{a,-\theta}\cup\{1\}$ and the observation (ii) yields the 
statement of the Proposition. \qed

% CHECK THE REASONING THROUGH FOR THE SPECIAL CASE $a=7$ IN THE STATEMENT OF THE PROPOSITION, 
% AND FOR THE %POINT $Z=1$ WE HAVE ADDED TO THE OPEN LUNE?

\subsubsection{Proof of Lemma \ref{intersections}}

The correspondence $Cov_0^Q$ sends $Z \to Z'$ where
$$\frac{Q(Z)-Q(Z')}{Z-Z'}=0 \quad\quad i.e. \quad Z'^2+ZZ'+Z'^2=3$$
$$ i.e. 
\quad Z'=\frac{-Z}{2}\pm\sqrt{3\left(1-\left(\frac{Z}{2}\right)^2\right)}$$

%A numerical investigation of the values of $d$ for which all the points of $Cov_0^Q(C_d\setminus\{1\})$ lie at distance 
%from $4-di$ (the centre of $C_d$) greater than $\sqrt{9+d^2}$ (the radius of $C_d$) reveals that this is 
%true for $|d|<1.73$ and false for $|d|>1.735$ so it seems reasonable to believe that there is a transition at $|d|=\sqrt{3}$. 
%We shall prove that this is indeed the case. 

%\begin{figure}\label{int}
%\begin{center}
%\scalebox{.45}{\includegraphics{int.pdf}}
%\caption{$C_d$ and $Cov_0^Q(C_d)$ in the $Z$-plane, and their lifts to the $W$-plane.}
%\end{center}
%\end{figure}

\begin{figure}
\begin{center}
\scalebox{.4}{\includegraphics{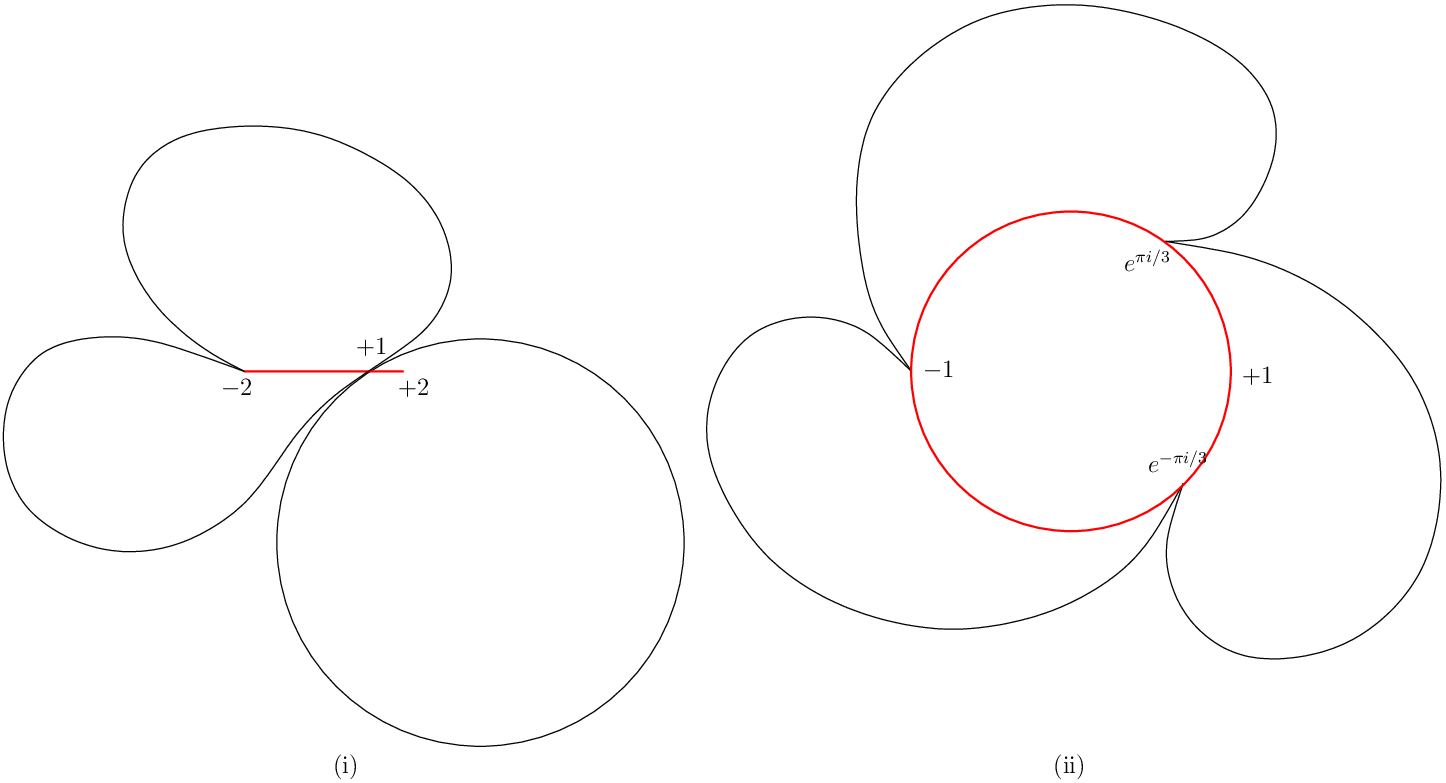}}
\caption{\small Sketch of (i) $C_d$ and $Cov_0^Q(C_d)$ in the $Z$-plane, and (ii) their lifts to the $W$-plane. (Only the lifts outside
the circle $|W|=1$ are shown: applying $W\to 1/W$ to these gives copies inside the circle).}\label{int}
\end{center}
\end{figure}

The image of $C_d$ under $Cov_0^Q$ is a topological circle which wraps twice around $C_d$ under the inverse of 
$Cov_0^Q$, a two-to-one map (see Figure \ref{int}(i)). This topological circle has a cusp at $Z=-2$ (the `other' inverse image of the fixed point 
$Z=1$) and it has a self-intersection near to the cusp if and only if it intersects $C_d$ near $Z=1$. The tangent and curvature of 
$Cov_0^Q(C_d)$ at $Z=1$ are identical to those of $C_d$ at $Z=1$, whatever the value of $d$, and the transition of 
intersection behaviour at the point $Z=1$ when $d=\sqrt{3}$ is associated with the values of the third (or higher) derivatives 
of the two curves at this point.  

The full picture is best seen in the $W$-plane double-covering the $Z$-plane via $W\to Z=W+1/W$ (Figure \ref{int}(ii)).
To investigate the intersections we pass to the $W$-plane, where $Cov^Q_0$ 
lifts to the $2$-valued map defined by multiplication by $e^{\pm 2\pi i /3}$, and where we can 
write explicit polynomial equations for the lifts of the various curves.

Writing $Z=X+iY$ and $W=U+iV$, the equation $W+1/W=Z$ gives:
$$U+iV+\frac{U-iV}{U^2+V^2}=X+iY.$$
Thus
$$X=U\left(\frac{U^2+V^2+1}{U^2+V^2}\right) \quad {\rm and} \quad Y=V\left(\frac{U^2+V^2-1}{U^2+V^2}\right).$$
The equation for $C_d$ is:
$$(X-4)^2+(Y+d)^2=9+d^2$$
and the equation of the lift ${\tilde C}_d$ of $C_d$ to the $W$-plane is therefore:
$$\left(U\left(\frac{U^2+V^2+1}{U^2+V^2}\right)-4\right)^2+\left(V\left(\frac{U^2+V^2-1}{U^2+V^2}\right)+d\right)^2=9+d^2.$$
Multiplying through by $(U^2+V^2)^2$ and re-arranging terms we find an expression which has $U^2+V^2$ as a
factor, and when this factor is divided out we obtain as equation for ${\tilde C}_d$:
\begin{equation} \label{1}
(U^2+V^2)(U^2+V^2-8U+7) +2d V(U^2+V^2-1) +2U^2-2V^2-8U+1 =0
\end{equation}

The map $\times e^{2\pi i/3}$ sends $(U,V)$ to $(U',V')$ where
$$\left(\begin{array}{c} U' \\ V' \end{array}\right)=
\left(\begin{array}{c} -U/2-\sqrt{3}V/2 \\ \sqrt{3}U/2-V/2 \end{array}\right).$$
That is,
$$\left(\begin{array}{c} U \\ V \end{array}\right)=
\left(\begin{array}{c} -U'/2+\sqrt{3}V'/2 \\ -\sqrt{3}U'/2-V'/2 \end{array}\right).$$
Substituting these expressions into equation (\ref{1}) gives us
$$(U'^2+V'^2)\left(U'^2+V'^2+4(U'-\sqrt{3}V')+7\right)-d(\sqrt{3}U'+V')(U'^2+V'^2-1)$$
$$-U'^2+V'^2-2\sqrt{3}U'V'+4(U'-\sqrt{3}V')+1=0.$$
Thus after rotation through $2\pi/3$ about $(0,0)$ in the $(U,V)$-plane, ${\tilde C}_d$ becomes

\begin{multline} \label{2}
 (U^2+V^2)\left(U^2+V^2+4(U-\sqrt{3}V)+7\right)-d(\sqrt{3}U+V)(U^2+V^2-1)\\
 -U^2+V^2-2\sqrt{3}UV+4(U-\sqrt{3}V)+1=0
\end{multline}
and after rotation through $-2\pi i/3$ the curve ${\tilde C}_d$ becomes:
\begin{multline}\label{3}
 (U^2+V^2)\left(U^2+V^2+4(U+\sqrt{3}V)+7\right)-d(-\sqrt{3}U+V)(U^2+V^2-1)\\
 -U^2+V^2+2\sqrt{3}UV+4(U+\sqrt{3}V)+1=0
\end{multline}
(which is just  (\ref{2}) with $\sqrt{3}$ replaced by $-\sqrt{3}$).
We now show that provided $|d| \le \sqrt{3}$ the curves (\ref{1}) and (\ref{2}) meet at the unique point 
$(U,V)=(1/2,\sqrt{3}/2)$. By rotational symmetry an equivalent problem is to show that the curves (\ref{2}) and 
(\ref{3}) have unique real intersection point $(U,V)=(-1,0)$.

The $U$-coordinates of the points of intersection of (\ref{2}) and (\ref{3}) are the zeros of the {\it resultant} of (\ref{2}) and (\ref{3}),
where these are regarded as polynomials in $V$ with coefficients polynomials in $U$. 

Recall that the resultant of the polynomials $a_4V^4+a_3V^3+a_2V^2+a_1V+a_0$ and $b_4V^4+b_3V^3+b_2V^2+b_1V+b_0$ is
the determinant
$$\left|\begin{array}{cccccccc}
a_4 & a_3 & a_2 & a_1 & a_0 & 0 & 0 & 0 \\
0 & a_4 & a_3 & a_2 & a_1 & a_0 & 0 & 0 \\
0 & 0 & a_4 & a_3 & a_2 & a_1 & a_0 & 0 \\
0 & 0 & 0 & a_4 & a_3 & a_2 & a_1 & a_0 \\
b_4 & b_3 & b_2 & b_1 & b_0 & 0 & 0 & 0 \\
0 & b_4 & b_3 & b_2 & b_1 & b_0 & 0 & 0 \\
0 & 0 & b_4 & b_3 & b_2 & b_1 & b_0 & 0 \\
0 & 0 & 0 & b_4 & b_3 & b_2 & b_1 & b_0 \\
\end{array}\right| $$
The resultant of (\ref{2}) and (\ref{3}) is the polynomial $P(U)$
given by the determinant above with

$$a_4=1, \quad a_3=-4\sqrt{3}-d, \quad a_2=2U^2+(4-d\sqrt{3})U+8$$
$$a_1=(-4\sqrt{3}-d)U^2-2\sqrt{3}U+(d-4\sqrt{3})$$
$$a_0=U^4+(4-d\sqrt{3})U^3+6U^2+(4+d\sqrt{3})U+1$$
and with $b_j (0\le j \le 4)$ obtained by substituting $-\sqrt{3}$ for $\sqrt{3}$ in $a_j$.

A computation using MAPLE gives:
$$P(U)=2304(d^2+9)(U+1)^4Q(U)$$
where
$$Q(U)=(d^2+25)U^4+40U^3+(96-12d^2)U^2+(64+16d^2)U+64.$$
When $d=0$, $Q(U)$ is a quartic which takes values $>0$ for all $U\in {\mathbb R}$, since then:
$$Q(U)=25U^4+40U^3+96U^2+64U+64=U^2(5U+4)^2+80(U+2/5)^2+256/5.$$
When $d=\pm\sqrt{3}$ we have
$$Q(U)=28U^4+40U^3+60U^2+112U+64=(U+1)^2(28U^2-16U+64)$$
which has unique real zero $U=-1$.

As we increase $|d|$ from $0$ to $\sqrt{3}$, the only way that a real zero of $Q(U)$ can be born is as a 
repeated zero, thus at a value of $d$ where the resultant of $Q(U)$ and its derivative $Q'(U)$ is zero. 
Another appeal to MAPLE tells us that this resultant is equal to:
$$-143327232d^4(d^2+25)(d^2-3)(d^2+24)^2$$ $$\left(=-2^{16}3^7d^4(d^2+25)(d^2-3)(d^2+24)^2\right)$$
and therefore that $Q(U)$ has no real zero for any value of $|d|$ in the interval $0<|d|<\sqrt{3}$, 
completing the proof of the lemma. \qed

\begin{remark}
By considering the expression for $P(U)$ for values of $d$ with $|d|$ equal to $\sqrt{3}+\varepsilon$ 
one can show that new real intersections of (2) and (3) bifurcate out of the intersection point 
$(U,V)=(-1,0)$ as $|d|$ passes through $\sqrt{3}$. Thus the value $\sqrt{3}$ in the statement 
of Lemma \ref{intersections} (and hence the value $\pi/3$ in the statement of Proposition \ref{dynamic-lune}) is best possible.
\end{remark}

\bigskip
\noindent
School of Mathematical Sciences,
Queen Mary University of London, 
London E1 4NS, UK 
\hfill s.r.bullett@qmul.ac.uk

\medskip

\noindent
Instituto de Matem\'atica Pura e Aplicada, 
Estrada Dona Castorina 110, Jardim Bot\^anico, Rio de Janeiro, RJ, 22460-320, Brasil
\hfill luna@impa.br

%\textcolor{blue}{Check all refs above cited (done). Insert e-mail and snail-mail addresses} 

\end{document}